\documentclass[a4paper,11pt]{amsart}

\usepackage{amsmath}
\usepackage{amsfonts}
\usepackage{amssymb}
\usepackage{amsthm}
\usepackage{xcolor}
\usepackage{enumitem}
\usepackage{ifthen}
\usepackage{mathrsfs}
\usepackage[margin=1in]{geometry}
\usepackage{tikz}
\usetikzlibrary{shapes.geometric}

\definecolor{lightgray}{RGB}{220, 220, 220}

\DeclareMathOperator{\Aut}{\mathrm{Aut}}

\newtheorem{thm}{Theorem}[section]
\newtheorem{rem}[thm]{Remark}
\newtheorem{prop}[thm]{Proposition}
\newtheorem{exam}[thm]{Example}
\newtheorem{lem}[thm]{Lemma}
\newtheorem{defn}[thm]{Definition}

\newtheorem{cor}[thm]{Corollary}
\newtheorem{claim}[thm]{Claim}

\newcommand{\emptyword}{\perp}
\newcommand{\Z}{\mathbf{Z}}
\newcommand{\N}{\mathbf{N}}
\newcommand{\R}{\mathbf{R}}

\newcommand{\cL}{\mathcal{L}}
\newcommand{\lang}{\mathscr{L}}
\newcommand{\C}{\mathcal{C}}
\newcommand{\eps}{\varepsilon}

\newcommand{\K}{\mathrm{K}}
\newcommand{\Shp}{\K_p^\sigma(\C^{\Z})}

\DeclareMathOperator{\Homeo}{Homeo}
\DeclareMathOperator{\Cont}{Cont}
\DeclareMathOperator{\offset}{offset}
\newcommand{\W}{\mathcal{W}}
\newcommand{\V}{\mathcal{V}}
\DeclareMathOperator{\ext}{ext}
\newcommand{\g}{\mathfrak{g}}
\DeclareMathOperator{\prob}{Prob}
\newcommand{\concat}{\mathbin{\smallfrown}}
\DeclareMathOperator{\Seg}{Seg}
\DeclareMathOperator{\Min}{Min}

\newcommand{\depth}[1]{\textrm{d}( #1 )}
\newcommand{\parent}[1]{\textrm{par}( #1 )}
\newcommand{\angl}[1]{\langle #1 \rangle}
\newcommand{\treevertices}[2]{V_{#1} ({#2})}
\newcommand{\treegroup}[2]{G_{#1} ({#2})}
\newcommand{\treehom}[2]{\rho_{#1} ({#2})}
\newcommand{\gbar}[1]{\widehat{#1}}
\newcommand{\Auto}[1]{{\Aut}(#1)}
\newcommand{\AutQ}[1]{{\Aut}'(#1)}
\newcommand{\gengroup}[1]{\langle {#1} \rangle}
\newcommand{\rhofactor}[2]{\rho_{#1}({#2})}

\newcommand{\strongfactor}[0]{automorphism-compatible}
\newcommand{\rev}[1]{\textrm{rev}(#1)}
\newcommand{\revop}{\mathop{\mathrm{rev}}\nolimits}
\newcommand{\blockexp}[1]{
    \ifthenelse{ \equal{#1}{*} }{ ^* }{ ^{[#1]} }
}
\newcommand{\bA}{\mathbf{A}}
\newcommand{\Set}[2]{\left\{#1\mathrel{}\middle|\mathrel{}#2\right\}}
\newcommand{\Trees}{\mathit{Trees}}
\newcommand{\IllFounded}{\mathit{IF}}
\newcommand{\WellFounded}{\mathit{WF}}

\newcommand{\prefix}{\sqsubseteq}
\newcommand{\cont}{\textrm{cont}}

\newcommand{\bx}{\mathbf{x}}
\newcommand{\ShAp}{\K_p^{\sigma}(\bA^{\Z})}
\newcommand{\ctup}[2]{\boldsymbol{#2}^{({#1})}}
\newcommand{\Xother}[0]{\bar{X}}

\newcommand{\piother}[0]{\bar{\pi}}
\newcommand{\treevertex}[0]{t}

\begin{document}
\title{The conjugacy and flip conjugacy problem for Cantor minimal systems}
\author[K. Deka]{Konrad Deka}
\email{deka.konrad@gmail.com}
\author[F. Garc\'{\i}a-Ramos]{Felipe García-Ramos}
\address{Unidad Cuernavaca, Instituto de Matemáticas, Universidad Nacional Autónoma de México (UNAM)}
\address{Faculty of Mathematics and Computer Science, Jagiellonian University in Kraków, Poland}
\email{felipegra@gmail.com}
\author[K. Kasprzak]{Kosma Kasprzak}
\address{Faculty of Mathematics and Computer Science, Jagiellonian University in Kraków, Łojasie\-wicza 6, 30-348 Kraków, Poland}
\email{kosma.kasprzak@student.uj.edu.pl}
\author[P. Kunde]{Philipp Kunde}
\address{Department of Mathematics, Oregon State University, Kidder Hall 064, Corvallis, OR 97331, USA}
\email{kundep@oregonstate.edu}
\author[D. Kwietniak]{Dominik Kwietniak}
\address{Faculty of Mathematics and Computer Science, Jagiellonian University in Kraków, Łojasie\-wicza 6, 30-348 Kraków, Poland}
\email{dominik.kwietniak@uj.edu.pl}

\subjclass[2020]{Primary 37B05; Secondary 03E15, 37B10, 54H20, 54H05, 05C15}
\keywords{Cantor minimal system, topological conjugacy, flip conjugacy,
complete analytic set, Borel reducibility, automorphism group, Borel equivalence relation}
\begin{abstract}
We prove that topological conjugacy and flip conjugacy of minimal
homeomorphisms of a Cantor space are complete analytic relations and hence
are not Borel. We also prove that mutual reducibility by
injective continuous graph homomorphisms and topological graph isomorphism
are complete analytic relations on the space of nonempty compact graphs on a fixed
Cantor vertex space with continuous chromatic number two. Both relations remain complete analytic on the larger space of
nonempty compact graphs on the same Cantor vertex space with continuous
chromatic number two or three.
\end{abstract}

\maketitle

\section{Introduction}

We prove that conjugacy and flip conjugacy of Cantor minimal
systems are complete analytic subsets of the square of their Polish
parameter space. In particular, neither relation is Borel. Two
homeomorphisms are \emph{flip conjugate} if one is conjugate to the
other or to its inverse.

To compare classification problems, one seeks to encode objects from
one problem as objects of another while preserving and reflecting
equivalence. We require the encoding to be Borel so that Borel
assignments of invariants for the target problem pull back to Borel
assignments for the original problem. More precisely, let $E$ and $F$
be equivalence relations on topological (usually Polish) spaces $X$
and $Y$. We say that $E$ is \emph{Borel reducible} to $F$, written
$E\le_B F$, if there exists a Borel map $f\colon X\to Y$ such that
\[
x_1 E x_2
\quad\Longleftrightarrow\quad
f(x_1)Ff(x_2)
\]
for all $x_1,x_2\in X$. A classification of the $F$-classes by invariants
assigned by a Borel map then yields one for the $E$-classes by composition
with $f$. In this sense, $E\le_B F$ means that the classification problem
for $E$ is no more complicated than that for $F$. Upper bounds on an
isomorphism relation correspond to classification results; lower bounds
rule out specified types of classification.

This approach applies, for example, to countable graphs, countable
groups, Polish metric spaces up to isometry, compact metric spaces, ergodic automorphisms of the standard Lebesgue space, 
and Banach spaces up to isomorphism; see
\cite{FRW,GerberKunde,lecomte2023continuous,SabokCstar,ZielinskiCompact}.
For the general theory and more examples, see \cite{GaoBook}.

We also consider the descriptive complexity of an equivalence relation
as a subset of the square of its parameter space $X$. Saying that
conjugacy is \emph{complete analytic} in this paper means complete
analyticity as a subset of $X^2$.

We consider topological dynamical systems on Cantor spaces
(\emph{Cantor systems}), and in particular \emph{minimal systems}, in
which every orbit is dense. Write $\C=\{0,1\}^{\N}$ for the Cantor set,
$\Homeo(\C)$ for its group of homeomorphisms, and $\Min(\C)$ for the set
of its minimal homeomorphisms. For background on Cantor minimal systems,
see \cite{Putnam}. Conjugacy is the usual conjugacy in the group
$\Homeo(\C)$.

Camerlo and Gao \cite{CamerloGao} proved that conjugacy of Cantor systems
is Borel bireducible with a universal orbit equivalence relation of
$S_\infty$, the permutation group of $\N$ with the topology of pointwise
convergence. In particular, conjugacy is not Borel. Vejnar proved that
this maximal complexity persists for conjugacy of transitive Cantor
homeomorphisms with dense periodic points
\cite[Theorem~3.4]{VejnarChaotic}. Here transitivity means that some
forward orbit is dense. Vejnar's result does not settle the minimal case:
a Cantor minimal system has no periodic points.

For Cantor minimal systems, we prove the following. Here $\Trees$ is
the Polish space of
trees on $\N$ of unbounded depth defined in Section~\ref{section:trees};
a tree is \emph{ill-founded} if it has an infinite branch and
\emph{well-founded} otherwise. For $U, V \in \Homeo(\C)$ we write
$U \cong V$ if $U$ and $V$ are conjugate.
\begin{thm} \label{thm:main-thm}
There is a continuous map
$\Trees \ni T \mapsto \left( \Phi_1(T), \Phi_2(T) \right) \in \Min(\C)^2$
such that, if $T$ is ill-founded, then $\Phi_1(T) \cong \Phi_2(T)$,
and if $T$ is well-founded, then 
 $\Phi_1(T) \not \cong \Phi_2(T)$ and $\Phi_1(T) \not \cong \Phi_2(T)^{-1}$. 
\end{thm}

\begin{cor}\label{cor:main}
The sets 
	\begin{align*}
		\{ (U, V) \in \Min(\C)^2 \colon U, V & \textrm{ are conjugate} 		\}, \\
		\{ (U, V) \in \Min(\C)^2 \colon U, V & \textrm{ are flip conjugate} 	\}
	\end{align*}
are complete analytic, hence they are not Borel.
\end{cor}

This result solves two open questions of Gao \cite[Questions 1.6 and 1.7]{buzzi2023open}.

\subsection*{Related classification results}

In dynamics, classification usually concerns either measure-theoretic
isomorphism or topological conjugacy. The two settings run parallel in
many respects: ergodicity corresponds to minimality or transitivity,
Rokhlin towers to Kakutani--Rokhlin partitions, and cutting-and-stacking
to the concatenation of words. The parallel serves as a guide, not a dictionary:
a theorem in one setting requires a separate proof in the other.

For measure-preserving systems, fix a standard nonatomic probability
space $(X,\mu)$. We write $\Aut(X,\mu)$ for its group of invertible
measure-preserving transformations, identified modulo equality almost
everywhere and equipped with the weak topology.

In 1970, Ornstein proved that Bernoulli systems with the same entropy
are isomorphic \cite{Ornstein,OrnsteinInfinite}. Entropy is a complete
Borel invariant on the Bernoulli class, with values in $[0,\infty]$; hence, it gives a Borel reduction
to equality on $\R$; thus isomorphism of Bernoulli systems is
\emph{smooth}. Beyond the Bernoulli class, equal entropy does not imply
isomorphism. This raises the question of which other kinds of invariants
can classify measure-preserving systems.

An equivalence relation $E$ on a Polish space is \emph{classifiable by
countable structures} if $E \le_B {\cong_L}$, where $\cong_L$ is the
isomorphism relation on the Polish space of structures with underlying
set $\N$ in some countable language $L$. The classification of ergodic
transformations with discrete spectrum by their countable groups of
eigenvalues, due to Halmos and von Neumann, is of this kind. Here the
eigenvalue groups are subgroups of the unit circle, not merely abstract
groups; their embeddings can be recorded by countably many predicates
for rational arcs \cite[p.~278]{ForemanWeiss2004}.

Hjorth's notion of turbulence is an obstruction to such classifications.
A continuous action of a Polish group $G$ on a nonempty Polish space $Y$
is \emph{turbulent} if every orbit is dense and meager, and if, for every
$y \in Y$, every open neighborhood $U$ of $y$, and every open
neighborhood $V$ of the identity of $G$, the set of points reachable
from $y$ by finitely many steps $z \mapsto gz$ with $g \in V$ that all
stay in $U$ is somewhere dense. If the action is turbulent, then its
orbit equivalence relation is not classifiable by countable structures.
More precisely, every Borel map from $Y$ to countable structures that
sends points in the same orbit to isomorphic structures is constant
\emph{up to isomorphism} on a comeager set
\cite[Chapter~3]{HjorthBook}; see also \cite{GaoBook}.

Hjorth proved that conjugacy of arbitrary measure-preserving
transformations is non-Borel \cite[Theorem~1.1]{Hjorth1}. In particular,
there is no complete numerical Borel invariant. His proof of
non-Borelness uses nonergodic transformations. Separately, he proved
nonclassifiability by countable structures, even for ergodic
transformations of rank-two generalized discrete spectrum
\cite[Theorem~5.15 and Corollary~5.16(b)]{Hjorth1}.

Foreman and Weiss subsequently proved that the action of $\Aut(X,\mu)$
by conjugation on the ergodic transformations is turbulent, and that
isomorphism cannot be classified by countable structures on any dense
$G_\delta$ subset of the space of measure-preserving transformations
\cite[Theorem~12 and Corollary~13]{ForemanWeiss2004}.
Foreman, Rudolph, and Weiss later proved that conjugacy of ergodic
transformations is complete analytic and hence non-Borel
\cite[Theorem~7 and Corollary~8]{FRW}.

The anti-classification results of Foreman, Rudolph, and Weiss have since
been extended in several directions; see
\cite{ForemanWeiss2022,GerberKunde,KundeWeakMixing}.
For a survey of anti-classification results in ergodic theory and
further developments, see Gerber and Kunde \cite{GerberKundeICM}.

For topological conjugacy, Bruin and Vejnar proved that conjugacy of
continuous selfmaps of $[0,1]$ is Borel bireducible with isomorphism of
countable graphs. They also proved that conjugacy of Hilbert cube
homeomorphisms is Borel bireducible with a universal orbit equivalence
relation of Polish group actions \cite{BruinVejnar}.

By the Curtis--Hedlund--Lyndon theorem, every conjugacy between
finite-alphabet subshifts is given by a block code. Consequently,
conjugacy of subshifts is a countable Borel equivalence relation. 
Clemens proved that this relation is universal among countable Borel
equivalence relations: every countable Borel equivalence relation is
Borel reducible to it \cite{clemens2009isomorphism}.
For minimal subshifts, Gao, Jackson, and Seward proved that conjugacy
is not smooth, even over a binary alphabet
\cite[Corollary~1.5.4]{gao2016group}.
The Borel upper bound extends from subshifts to expansive Cantor
systems by Borel symbolic coding with finite clopen generating
partitions; see \cite[Theorem~1.5.6]{gao2016group}.
Finer bounds are known for several classes of Toeplitz subshifts
\cite{DekaPeng,GaoLiPengSun,KayaToeplitz,SabokTsankov}.

Kaya \cite{Kaya} proved that conjugacy of pointed Cantor minimal systems
is Borel bireducible with $=^+$ and hence Borel. The objects are triples
$(\C,T,p)$ with $T\in\Min(\C)$ and $p\in\C$, and conjugacies must
preserve the marked points. Here $=^+$ is the relation on $\R^{\N}$
defined by
$f=^+g$ if and only if $\{f(n):n\in\N\}=\{g(n):n\in\N\}$.

Corollary~\ref{cor:main} rules out extending the Borel classifications
of expansive and of pointed Cantor minimal systems to all Cantor minimal
systems. More generally, non-Borelness excludes any Borel assignment of
complete invariants whose comparison relation is Borel, including real
numbers, countable sets of reals, and countable structures with a Borel
isomorphism relation. It does not exclude classification by countable
structures whose isomorphism relation is non-Borel.

The group $\Homeo(\C)$, with the topology of uniform convergence, is
\emph{non-Archimedean}: for every finite clopen partition $\mathcal P$
of $\C$, the homeomorphisms mapping each atom of $\mathcal P$ onto itself
form an open subgroup, and these subgroups form a neighborhood basis
of the identity. Non-Archimedean Polish groups are precisely the Polish
groups topologically isomorphic to closed subgroups of $S_\infty$.
Every orbit equivalence relation of a Borel action of such a group on
a standard Borel space is classifiable by countable structures
\cite{BeckerKechris}.

For conjugacy of Cantor systems, the countable structure is explicit.
Let $\operatorname{Clop}(\C)$ be the countable Boolean algebra of clopen
subsets of $\C$, and assign to $T\in\Homeo(\C)$ the structure
$(\operatorname{Clop}(\C);\cup,\cap,\complement,T_*)$, where
$T_*(A)=T(A)$. By Stone duality, two Cantor systems are conjugate if and
only if the assigned structures are isomorphic. The assignment is
continuous once the clopen sets are enumerated because, for clopen
$A,B\subseteq\C$, the set of $S\in\Homeo(\C)$ with $S(A)=B$ is clopen.
Thus conjugacy of Cantor systems, and its restriction to $\Min(\C)$,
is classifiable by countable structures.

Together with Corollary~\ref{cor:main}, this gives a non-Borel conjugacy
relation that is classifiable by countable structures. The converse
separation also occurs in dynamics: conjugacy of rank-one ergodic
transformations is Borel \cite[Theorem~51]{FRW}, but it is not
classifiable by countable structures. Indeed, rank-one transformations
form a dense $G_\delta$ class \cite[Section~10]{FRW}, to which
\cite[Corollary~13]{ForemanWeiss2004} applies.

By Hjorth's theorem, no turbulent action on a nonempty Polish space has
an orbit equivalence relation Borel reducible to conjugacy of Cantor
minimal systems. In particular, the action of $\Homeo(\C)$ by conjugation
on $\Min(\C)$ is not turbulent, and neither is its restriction to any
nonempty invariant Polish subspace. The turbulence conclusion of
Foreman and Weiss therefore has no counterpart for conjugacy of Cantor
minimal systems.

The weak topology on $\Aut(X,\mu)$ differs from the topology on
$\Homeo(\C)$ in this respect. It is generated by the maps
$T\mapsto\mu(T(A)\cap B)$, with $A,B$ measurable, so a basic neighborhood
prescribes finitely many of these numbers up to a positive tolerance.
The group is contractible \cite{Keane}; hence it has no proper open
subgroup and every locally constant function on it is constant.
Every aperiodic conjugacy class is dense, and the ergodic transformations
form a dense $G_\delta$ \cite{Halmos}; see also
\cite[pp.~280--284]{ForemanWeiss2004}.
For the local-orbit condition, Foreman and Weiss combine approximation
by transformations with the same orbits almost everywhere with
conjugations by transformations close to the identity, keeping the
intermediate transformations in the prescribed neighborhood
\cite[Lemma~9, Theorem~12, and Claim~14]{ForemanWeiss2004}.

\subsection*{Methods and further results}

We adapt techniques of Foreman, Rudolph, and Weiss \cite{FRW} and of
Gerber and Kunde \cite{GerberKunde} to Cantor minimal systems. Their
symbolic cutting-and-stacking methods produce subshifts by concatenating
words, along with shift-invariant measures on them, and every
infinite minimal subshift is a Cantor minimal system.
Finite-alphabet symbolic models already suffice for the
measure-theoretic hardness result: the reduction in \cite{FRW} takes
values in zero-entropy ergodic transformations, and isomorphism on that
class is complete analytic \cite[Section~9.1.1]{FRW}.
In topological dynamics, finite-alphabet coding forces expansiveness,
and conjugacy of expansive Cantor systems is Borel
(\cite{clemens2009isomorphism} for subshifts,
\cite[Theorem~1.5.6]{gao2016group} for the expansive case). Thus,
finite-alphabet models support a
complete analytic measure-theoretic isomorphism relation but a Borel
topological conjugacy relation. 
We use inverse limits of subshifts instead. The passage to inverse limits
requires control of all factor maps
between the coordinates in the sequence. We construct an inverse system
$(X_j,\pi_j)$ of infinite minimal subshifts such that every factor map
$X_i\to X_j$, for $i\ge j$, is the composite bonding map followed by
an automorphism of $X_j$. This property forces every conjugacy between
two inverse limits with the same coordinate systems to be given
coordinatewise by compatible automorphisms
(Lemma~\ref{lem:inv-lim-isomorphism}).

Given a tree $T$, we realize the groups and homomorphisms associated
with the grafted tree $\Psi(T)$ as the reduced automorphism groups
(the automorphism groups modulo the powers of the shift) of the coordinate
systems and the maps induced by the bonding maps.
We then replace $\pi_j$ by $\alpha_j\pi_j$ for suitable automorphisms
$\alpha_j$ of $X_j$. Modulo shift powers, compatibility of coordinate
automorphisms becomes the system of group equations in
Lemma~\ref{lem:trees-and-groups}, whose solvability is equivalent to
the existence of an infinite branch of $T$. Conversely, a solution
lifts to exactly compatible automorphisms after correcting their
shift powers. Thus the two inverse limits are conjugate precisely
when $T$ is ill-founded (Lemma~\ref{lem:main-reduction}).

The word construction also excludes factor maps from any $X_j$ onto
$\rev{X_1}$. A conjugacy to the inverse of the second limit would
produce such a factor map, so this exclusion gives the flip-conjugacy
conclusion as well. Lemmas~\ref{lem:word-construction}
and~\ref{lem:getting-inv-system} establish these properties
simultaneously, and Lemma~\ref{lem:parameter-continuity} proves
continuous dependence of both limits on the tree.

Lemma~\ref{lem:main-reduction} separates this abstract inverse-limit
mechanism from the probabilistic word construction used to realize it.
The lemma may be useful for minimal Cantor systems with additional
dynamical conditions.

An \emph{automorphism} of a Cantor minimal system $T\colon \C \to \C$ is a homeomorphism of $\C$ that commutes with $T$. 
The automorphism group contains $\Set{T^n}{n\in \Z}$. We say that
$T$ has a trivial automorphism group if these are its only automorphisms.

\begin{thm} \label{thm:centralizer}  The set $\Set{T\in \Min(\C)}{T\text{ has nontrivial automorphism group}}$ is a complete analytic subset of $\Min(\C)$.\end{thm}

We also apply the flip-conjugacy result to graphs on a fixed Cantor
space. Here, the edge relation is compact, colorings are continuous,
and graph isomorphisms are required to be homeomorphisms of the vertex
space. Lecomte relates injective continuous homomorphisms between the
graphs of minimal systems to flip conjugacy
\cite[Lemma~7.11 and Theorem~13.1]{lecomte2023continuous}.
Section~\ref{sec:graph-application} gives the Polish coding and proves
the following theorem.
\begin{thm}\label{thm:graph-intro}
On the space of nonempty compact graphs on $\C$ with continuous
chromatic number two, both topological graph isomorphism and mutual
injective continuous graph homomorphism are complete analytic relations.
Both relations remain complete analytic on the larger space of
nonempty compact graphs on $\C$ with continuous chromatic number
two or three.
\end{thm}

Appendix~\ref{sec:legacy-reversal} gives an alternative construction
using reversal.

\subsection*{Acknowledgments}

These results grew out of discussions in a joint research seminar
involving all the authors at the Jagiellonian University in 2022.
The main part of the proof appeared earlier in Konrad Deka's PhD
thesis \cite{DekaThesis}, defended at the Jagiellonian University
under the supervision of Dominik Kwietniak.

We thank Marcin Sabok for his enthusiasm and encouragement, and for
many fruitful conversations. We are also grateful to Matthew
Foreman, Su Gao, Marlies Gerber, Bo Peng, and Benjy Weiss for many useful discussions.

We thank Mar{\'\i}a Isabel Cortez, Fabien Durand, and Kostya Medynets,
the referees of Konrad Deka's PhD thesis, for their remarks, which
helped improve both the thesis and this text. A special thank-you goes
to Bo Peng for giving us the nudge we needed to finish a draft that
had waited rather too long for publication.

The mathematical content and the final text are the authors' own.
Computers were used to polish the English, proofread and organize the
manuscript, check arguments and references, and suggest revisions.
Any remaining errors are, of course, ours.

\section{Preliminaries}
We write $\N:=\{1,2,\ldots\}$ and $\N_0:=\{0,1,2,\ldots\}$.

\subsection{General topology}
A subset of a topological space is \emph{perfect} if it has no isolated points. A topological space is \emph{totally disconnected} if the only nonempty connected subsets are the singletons. A \emph{Cantor space} is a nonempty, compact, metrizable, totally disconnected, perfect topological space. 
The Cantor set $\C = \{0,1\}^ {\N}$ carries the product topology,
with $\{0,1\}$ discrete. By Brouwer's theorem
\cite[Theorem 7.4]{Kec}, every Cantor space is homeomorphic to $\C$.
For a topological space $X$, let $\K(X)$ be the collection of nonempty compact subsets of $X$, and let $\K_p(X)$ consist of its perfect elements. We equip $\K(X)$ with the \emph{Vietoris topology}. A basis for the Vietoris topology consists of sets
\[
\{A\in \K(X): A\subseteq (U_1\cup\ldots\cup U_k),\text{ and }A\cap U_1\neq\emptyset,\ldots,A\cap U_k\neq\emptyset\}
\]
where $k\in\N$ and $U_1,\ldots,U_k$ range over the nonempty open subsets of $X$. 
If $d$ is a metric for a compact metrizable space $X$, then the \emph{Hausdorff metric} $d_H$ on $\K(X)$ is compatible with the Vietoris topology. For $x\in X$ and nonempty $A\subseteq X$, define the \emph{distance from $x$ to $A$} by
\[
d_A(x):=\inf\{d(x,y):y\in A\}.
\]
The \emph{$\eps$-hull of $A$} is 
\[
A^{\eps}:=\{x\in X: d_A(x)<\eps\}.
\]
The Hausdorff metric between $A,B\in \K(X)$ is defined as
\[
d_H(A,B):=\inf\{\eps>0: A\subseteq B^{\eps}\text{ and }B\subseteq A^{\eps}\}.
\]
The set $\K_p(X)$ is a $G_{\delta}$-subset of $\K(X)$.

Every nonempty closed perfect subset of a Cantor space is a Cantor space in the subspace topology. Hence,
\[
\K_p(\C)=\{A\subseteq \C: A\text{ is nonempty, closed and perfect}\}.
\]

\subsection{Polish spaces and Borel reducibility}
A topological space is \emph{Polish} if it is separable and completely metrizable.

Let $X$ and $X'$ be Polish spaces. 
We say that $A\subseteq X$ is \emph{Borel (continuously) reducible} to $A'\subseteq X'$ if there is a Borel (continuous) function $g\colon X\to X'$ such that $x\in A$ if and only if $g(x)\in A'$. 
Equivalently, $g^{-1}(A')=A$. 
We call $g$ a \emph{Borel (continuous) reduction} and say that it
\emph{Borel (continuously) reduces} $A$ to $A'$.  
We also write $A\le_B A'$ ($A\le_{\cont} A'$). If $A$ is not Borel and $A\le_B A'$, then $A'$ is not Borel.
For Borel sets, reducibility formalizes comparison of complexity;
see \cite{Kec}. We say that $A$ and $A'$ are \emph{Borel bireducible}
if $A\le_B A'$ and $A'\le_B A$. 

A subset $A\subset X$ is \emph{analytic} if it is a continuous image
of a Polish space: there are a Polish space $Y$ and a continuous map
$f\colon Y\to X$ with $f(Y)=A$. An analytic set $A$ is
\emph{complete analytic} if every analytic set is Borel reducible to $A$. Every Borel set is analytic, but not every analytic set is Borel. In particular, complete analytic sets are not Borel. A classic example of a complete analytic set is the set of ill-founded trees, which we describe below.

For equivalence relations $R$ on $X$ and $R'$ on $X'$, we use
\emph{Borel (continuous) reducibility} in the following sense: there
is a Borel (continuous) map $g\colon X\to X'$ such that $x_1Rx_2$
is equivalent to $g(x_1)R'g(x_2)$ for every $(x_1,x_2)\in X\times X$.
We write $R\le^2_B R'$ ($R\le^2_{\cont} R'$). 
We say that $R$ and $R'$ are \emph{Borel bireducible} if $R\le^2_B R'$ and $R'\le^2_B R$. If $R\le^2_B R'$ via $g$, then $g\times g$ witnesses $R\le_B R'$. We say that $R$ and $R'$ are \emph{topologically isomorphic} if there is a homeomorphism $g\colon X\to X'$ that reduces $R$ to $R'$.

We call an equivalence relation $R$ on a Polish space $X$
\emph{complete analytic} when it is complete analytic as a subset of $X^2$.
This does not assert universality among analytic equivalence relations
under $\le^2_B$.

\subsection{Words}\label{subsec:words}
Let $\bA$ be a nonempty set, called the \emph{alphabet}.
A \emph{word over $\bA$} is a finite tuple of elements of $\bA$. We denote the empty word by $\emptyword$. The set of all words over $\bA$ is denoted by $\bA^{<\N}$; that is $\bA^{<\N}:=\bigcup_{n\in \N_0}\bA^n$, with the convention that $\bA^0=\{\emptyword\}$. The \emph{length} $|w|\in\N_0$ of a word $w$ is its number of entries. Given a word $s\in \bA^n$, 
we index its letters from zero and write $s=s_0s_1\dots s_{n-1}$, so that $s_i\in\bA$ is the $i$th coordinate of the word $s$ for $0\le i<n$.
 Given $s\in\bA^n$  and $t\in \bA^m$ we define the \emph{concatenation} of $s$ and $t$, $s \concat t\in\bA^{n+m}$,  by 
 \[
 (s \concat t)_j=\begin{cases}
     s_j&\text{ for }0\le j<n,\\
     t_{j-n}&\text{ for }n\le j<n+m.
 \end{cases}\] For simplicity we will often write $s \concat t = st= s_0\dots s_{n-1}t_0 \dots t_{m-1}$.
For words $s, s' \in \bA^{ < \N}$, we write $s \prefix s'$ if $s$ is a \emph{prefix} of $s'$, that is, if for some $t\in\bA^{<\N}$ we have $s'=s\concat t$.
Similarly, we define \emph{(bi-)infinite words over $\bA$} to be (bi-)infinite sequences with entries in $\bA$, that is elements of the Cartesian products $\bA^{\N}$ (or $\bA^{\Z}$).
For $\bx\in \bA^{\N}$ (respectively, $\bx\in \bA^{\Z}$), we write $\bx=(x_j)_{j\in\N}=x_1x_2\ldots$ (respectively, $\bx=(x_j)_{j\in\Z}=\ldots x_{-2}x_{-1}.x_0x_1x_2\ldots$), where $x_j\in \bA$ is the $j$th coordinate of $\bx$. Thus one-sided infinite words are indexed from one, like the points of $\C=\{0,1\}^{\N}$, whereas finite words are indexed from zero. The time coordinates of bi-infinite sequences are indexed by $\Z$, with zero as the distinguished origin. For $i,j\in \Z$, we let $[i,j]$ represent the set $\{i,i+1,\ldots,j\}$ of consecutive integers from $i$ to $j$ if $i\le j$ and the empty set $\emptyset$ otherwise. A similar convention applies to all other types of intervals: $[i,j)$, $(i,j)$, and $(i,j]$. The convention naturally extends to unbounded intervals, like $(-\infty,i]$, $(j,\infty)$, etc. Given $i,j\in\Z$ and a finite or (bi-)infinite word $x$ and $[i,j]$ contained in the domain of $x$, we let $x_{[i,j]}$ represent the word $x_ix_{i+1}\ldots x_j\in \bA^{j-i+1}$, if $i\le j$, or the empty word $\emptyword$ if $j<i$; thus $(x_{[i,j]})_k=x_{i+k}$ for $0\le k\le j-i$, and $w_{[0,n)}=w$ for a finite word $w$ of length $n$. The restriction of $x$ to other types of intervals is defined analogously. We define the \emph{cylinder} of $w=w_0\ldots w_{n-1}\in \bA^{n}$ in $\bA^{\N}$ as \[
[w]=\{x\in \bA^{\N}:x_{[1,n]}=w\}.
\]
\subsection{Trees} \label{section:trees}
Recall that $\N^{<\N}$ stands for the set of all finite words over $\N$. 
Fix an enumeration $(w^{(i)})_{i\in\N_0}$ of $\N^{<\N}$ with
$w^{(0)}=\emptyword$ such that $w^{(i)}\prefix w^{(j)}$ implies
$i\le j$. Thus every word occurs exactly once, after all its proper
prefixes.

A \emph{tree} is a nonempty subset $T$ of $\N^{< \N}$ such that if $t \in T$ and $s \prefix t$, then $s \in T$. 
By identifying a tree $T$ with its characteristic function,
we consider a tree to be an element of $2^{\N ^{ < \N}}$. 

The words $t\in T$ are the \emph{vertices of $T$}. Every tree contains
$\emptyword$, its \emph{root}. The \emph{depth} of a vertex $t$ is
$\depth{t}:=|t|$; the root is the unique vertex of depth $0$.
For $j\in \N_0$, define $$\treevertices{j}{T} := \{ t \in T \colon \depth{t} = j\}.$$ 
We say that a tree $T$ has \emph{unbounded depth} if $\treevertices{j}{T} \neq \emptyset$ for all $j \in\N_0$.
We define the \emph{depth of a tree} $T$ as $\depth T:=\sup\{j\in\N_0:\treevertices{j}{T} \neq \emptyset\}$.
We denote the set of all trees with unbounded depth by $\Trees$. 

An \emph{infinite branch} of $T$ is a sequence
$w=w_1w_2w_3\ldots\in \N^{\N}$ such that $w_{[1,n]}\in T$
for every $n\in\N$. A branch is not itself a vertex of $T$.
A tree is \emph{well-founded} if it does not have an infinite branch. A tree with at least one infinite branch is \emph{ill-founded}. We write $\IllFounded := \{ T \in \Trees : T \textrm{ is ill-founded} \}$ and $\WellFounded:=\Trees\setminus\IllFounded$.

Let $\mathscr T$ be the space of all trees on $\N$, including those
of bounded depth, with the topology inherited from $2^{\N^{<\N}}$ through the identification of a tree with its characteristic function.
Prefix closure and nonemptiness are closed conditions, so $\mathscr T$ is a closed subspace
of $2^{\N^{<\N}}$ and is compact metrizable. We give $\Trees$ the subspace topology. Since
\[
\Trees=\bigcap_{k\ge1}\bigcup_{|s|=k}\{T\in\mathscr T:s\in T\},
\]
the space $\Trees$ is a $G_\delta$ subspace of $\mathscr T$ and is Polish.
Ill-foundedness on $\mathscr T$
is complete analytic \cite[Theorem 27.1]{Kec}. The grafting map below transfers this
fact to $\Trees$ without changing whether a tree has an infinite branch.

For a nonroot vertex $t=t_0\ldots t_{n-1}\in T$, define its
\emph{parent} by $\parent{t} := t_0\dots t_{n - 2}\in \treevertices{n-1}{T}$.
The corresponding graph joins each nonroot vertex to its parent.

We also represent infinite branches by sequences of vertices
$(t_j)_{j\in\N}$ with $t_j\in\treevertices{j}{T}$ and
$\parent{t_{j+1}}=t_j$ for every $j\in\N$.  

\subsection{Tree-directed inverse systems of groups}\label{subsec:tree-groups}
An \emph{inverse system of groups} is a sequence
$(G_j,\rho_j)_{j\in \N}$ in which each $G_j$ is a group and each
$\rho_j \colon G_{j+1} \to G_j$ is a group homomorphism. 
Two such inverse systems $(G_j,\rho_j)$ and $(H_j,\zeta_j)$ are \emph{conjugate} if there are group isomorphisms $\phi_j:G_j\to H_j$ satisfying $\zeta_j\phi_{j+1}=\phi_j\rho_j$ for every $j\ge1$. The \emph{inverse limit} of an inverse system of groups $(G_j,\rho_j)_{j\in \N}$ is the set
\[
\varprojlim (G_j,\rho_j)_{j\in \N}:=\{(g_j)_{j\in\N}\in\prod_{j\in\N}G_j:\rho_j(g_{j+1})=g_j \text{ for every }j\in\N\}.
\]
Under coordinatewise multiplication, the inverse limit
$\varprojlim (G_j,\rho_j)_{j\in \N}$ is a subgroup of
$\prod_{j\in\N}G_j$. It contains the sequence of identity elements
and is therefore nonempty.
The inverse limits of conjugate inverse systems of groups are isomorphic groups, but the converse does not always hold. 

Given a set $X$, possibly empty, we define the \emph{group of involutions generated by $X$} as
\[
\Z_2(X) := \left\{ f\colon X \to \Z_2 | \; f(x)=1 \textrm{ for at most finitely many } x\in X \right\}.
\]
Under pointwise addition modulo $2$, $\Z_2(X)$ is an abelian group
in which every nonidentity element has order $2$. It is trivial when
$X=\emptyset$. For $x \in X$, we define $\gbar{x}\in \Z_2(X)$ as the function $X\to\Z_2$ satisfying $\gbar{x}(x') = 1$ if and only if $x' = x$. Clearly, $\{\gbar{x}:x\in X\}$ generates $\Z_2(X)$. 
Equivalently, $\Z_2(X)=\bigoplus_{x\in X}\Z_2$ is the vector
space of finitely supported functions from $X$ to the field $\Z_2$,
with canonical basis $\{\gbar{x}:x\in X\}$.

 We associate a group to each positive level of a tree $T$. The vertices
at that level form its canonical basis, and the parent map induces the
bonding homomorphisms.
\begin{defn}
Let $T\in\Trees$. For every $j\in\N$, define
\[
\treegroup{j}{T}:=\Z_2(\treevertices{j}{T}).
\]
For every $j\in\N$, define the homomorphism
\[
\treehom{j}{T}:\treegroup{j+1}{T}\longrightarrow\treegroup{j}{T}
\]
by its action on the canonical basis:
\[
\treehom{j}{T}(\gbar{t})=\gbar{\parent{t}}
\qquad(t\in\treevertices{j+1}{T}).
\]
These groups and homomorphisms form the \emph{$T$-directed system of
 groups of involutions}.
\end{defn}
The resulting inverse system is
\[
\treegroup{1}{T}\stackrel{\treehom{1}{T}}{\longleftarrow}
\treegroup{2}{T}\stackrel{\treehom{2}{T}}{\longleftarrow}
\cdots\stackrel{\treehom{n-1}{T}}{\longleftarrow}
\treegroup{n}{T}\stackrel{\treehom{n}{T}}{\longleftarrow}\cdots.
\]

The inverse limit contains the sequence of identity elements. Its nontriviality is determined by the branches of the tree.

\begin{lem}
\label{lem_trivial}
Let $T\in \Trees$. The inverse limit of the $T$-directed inverse system of groups of involutions  $(\treegroup{j}{T},\treehom{j}{T})_{j\in \N}$ is nontrivial if and only if $T$ is ill-founded.
\end{lem}
\begin{proof}
An infinite branch $(t_j)_{j\ge1}$ gives the nonzero coherent sequence
$(\gbar{t_j})_{j\ge1}$. Conversely, let $(g_j)_{j\ge1}$ be a nonzero
coherent sequence. Write $g_j$ as a finite sum of distinct level-$j$
generators. If the coefficient of a vertex $t$ in $g_j$ is one, coherence
says that the sum modulo two of the coefficients of its children in
$g_{j+1}$ is one. At least one such child has coefficient one.
Starting with a nonzero coefficient and choosing a child successively
produces an infinite branch; its earlier vertices are the prefixes of
the starting vertex.
\end{proof}
\begin{rem}\label{rem:unique-branch-group}
The same argument shows that every vertex occurring in a nonzero
coherent sequence lies on an infinite branch. If the tree has exactly
one infinite branch, its inverse-limit group therefore consists of
zero and the indicator sequence of that branch.
\end{rem}

For integers $k, \ell \in \N $, we define $\ctup{k}{\ell}$ to be the word over $\N$ of length $k$ all of whose entries equal $\ell$. Let $R=\{\emptyword\}\cup\{\ctup{k}{\ell}:1\le k\le\ell,\ \ell\in\N\}$ be the rooted tree whose maximal finite branches have endpoints $\ctup{k}{k}$, one for each $k\in\N$.
For a tree $T$, define its \emph{grafting} $\Psi(T)$ by
$$
	\Psi(T) := \{\emptyword\}\cup\{ \ctup{k}{\ell} : 1 \le k \le \ell,\ \ell\in\N \} \cup
		\{ \ctup{k}{k} \concat t : k \ge 1, t \in T \}\in \Trees.
$$
The tree $\Psi(T)$ is obtained by attaching a copy of $T$ at each
endpoint $\ctup{k}{k}$ of $R$, as in Figure~\ref{fig:Psi-T}.
The root of the $k$th copy lies at depth $k$, so $\Psi(T)$ has unbounded
depth. This definition applies to every $T\in\mathscr T$. Membership of a fixed vertex in $\Psi(T)$ is either fixed by $R$ or determined by membership of one vertex in $T$. Hence $\Psi:\mathscr T\to\Trees$ is
continuous. An infinite branch of $\Psi(T)$ chooses its first symbol $k$,
traverses the trunk ending at $\ctup{k}{k}$, and then follows a branch of
$T$. Conversely, a branch of $T$ extends any one of these trunks.
Thus $T$ is ill-founded if and only if $\Psi(T)$ is ill-founded.
This continuously reduces ill-foundedness on $\mathscr T$ to
$\IllFounded\subseteq\Trees$; the latter set is analytic by projection
of the closed branch relation. Consequently $\IllFounded$ is complete
analytic in the stated domain $\Trees$.

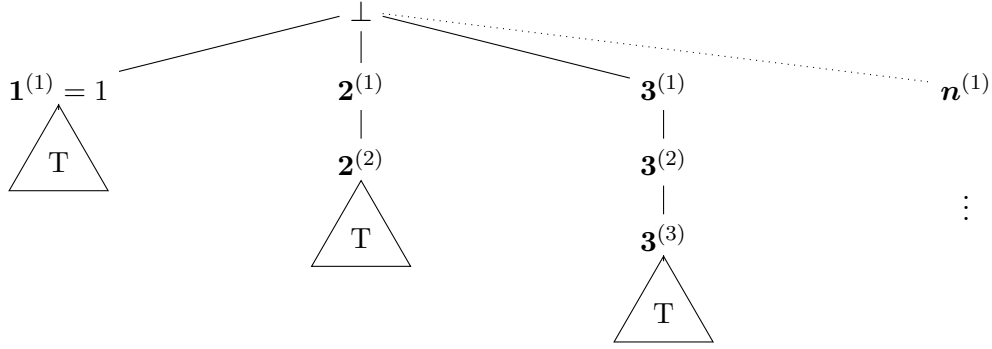
\begin{figure}
\centering
\begin{tikzpicture}
  
  \begin{scope}[xscale=2]
    \node (root) at (0,0) {$\emptyword$};
    
    \node (node1) at (-2,-1) {$\ctup{1}{1}=1$};
    \draw (root) -- (node1);
    \node[regular polygon, regular polygon sides=3, draw, fill=white, minimum size=1cm] (triangle1) at (-2,-2) {T};
    \draw (node1) -- (triangle1);

    \node (node2) at (0,-1) {$\ctup{1}{2}$};
    \draw (root) -- (node2);
    \node (node3) at (2,-1) {$\ctup{1}{3}$};
    \draw (root) -- (node3);
    \node (node4) at (4,-1) {$\ctup{1}{n}$};
    \draw[dotted] (root) -- (node4); 
    \node[regular polygon, regular polygon sides=3, draw, fill=white, minimum size=1cm] (triangle2) at (0,-3) {T};

    \node at (4,-2.5) {$\vdots$};
    
    \node (node7) at (0,-2) {$\ctup{2}{2}$};
      \draw (node2) -- (node7);
    \node (node5) at (2,-2) {$\ctup{2}{3}$};
    \draw (node3) -- (node5);
    \node (node6) at (2,-3) {$\ctup{3}{3}$};
    \draw (node5) -- (node6);
    \node[regular polygon, regular polygon sides=3, draw, fill=white, minimum size=1cm] (triangle3) at (2,-4) {T};
    \draw (node6) -- (triangle3);
  \end{scope}
\end{tikzpicture}
    \caption{The first branches of the tree $\Psi(T)$.}
    \label{fig:Psi-T}
\end{figure}

The next lemma characterizes ill-foundedness of $T$ by equations in
the $\Psi(T)$-directed inverse system. For each $j\in\N$, the group
$G_j(\Psi(T))$ has infinitely many generators.

\begin{lem} \label{lem:trees-and-groups}
Let $T\in \Trees$ and $(G_j,\rho_j)_{j\in \N}:=(G_j(\Psi(T)),\rho_j(\Psi(T)))_{j\in \N}$ be the $\Psi(T)$-directed inverse system of groups.
The following conditions are equivalent:
\begin{enumerate}[label=(\roman*)]
	\item \label{cond:trees-groups-i} $T\in\IllFounded$
	\item \label{cond:trees-groups-ii} there is a sequence $(g_j)_{j\in\N}$ such that for every $j\in\N$ we have  
		\begin{equation} \label{eqn:groupbraiding}
			\text{$g_j \in G_j$ and }\gbar{\ctup{j}{j}} + \rho_j(g_{j+1}) = g_j.
		\end{equation}
\end{enumerate}
\end{lem}

\begin{proof}
\ref{cond:trees-groups-i}$\Rightarrow$\ref{cond:trees-groups-ii} 
Suppose that
$(t_j)_{j\in \N}$ is an infinite branch of $T$, presented as a sequence of vertices of $T$ such that $\parent{t_{j+1}} = t_j$ for every $j \in\N$. 
For $j,k\in\N$ define $t'_{j, k} := \ctup{k}{k} \concat t_j$, and for $j\in\N$ define $g_j := \sum_{1 \le k < j} \gbar{t'_{j-k, k}}$.  The empty sum is the identity, so $g_1$ is the identity of $G_1$.
The definitions give \eqref{eqn:groupbraiding}.  

\ref{cond:trees-groups-ii}$\Rightarrow$\ref{cond:trees-groups-i} Assume that $(g_j)_{j\in \N}$ satisfies \eqref{eqn:groupbraiding}.  
Since $G_1$ has infinitely many generators, there exists $k$ such that $g_1(\ctup{1}{k}) = 0$. 
Note that for $j=1, \dots, k-1$ the vertex
$\ctup{j+1}{k}$ is the unique child of $\ctup{j}{k}$ (unique vertex $v\in\treevertices{j+1}{\Psi(T)}$ such that
$\parent{v}=\ctup{j}{k}$), therefore
\begin{equation}\label{eqn:ind-rhoj}
    \rho_j(g_{j+1})(\ctup{j}{k}) = g_{j+1}(\ctup{j+1}{k})\text{ for $j=1, \dots, k-1$}.
\end{equation}
Evaluating \eqref{eqn:groupbraiding} at $\ctup{j}{k}$ for $j<k$, where $\gbar{\ctup{j}{j}}(\ctup{j}{k})=0$ since $j<k$, and using \eqref{eqn:ind-rhoj}, we get $g_j(\ctup{j}{k})=g_{j+1}(\ctup{j+1}{k})$ for $j=1,\dots,k-1$; together with $g_1(\ctup{1}{k}) = 0$ this yields $g_j(\ctup{j}{k}) = 0$ for $j=1, \dots, k$.  Plugging $j=k$ and $g_k(\ctup{k}{k})=0$ into \eqref{eqn:groupbraiding} and evaluating at $\ctup{k}{k}$ we obtain
\[1 + \rho_k(g_{k+1})(\ctup{k}{k}) = 
\gbar{\ctup{k}{k}}(\ctup{k}{k}) + \rho_k(g_{k+1})(\ctup{k}{k}) = g_k(\ctup{k}{k})=0
\]
which yields 
\begin{equation}\label{base-case}
\rho_k(g_{k+1})(\ctup{k}{k}) = 1.     
\end{equation}

We construct an infinite branch $(t_j)_{j\in \N}$ of $T$
inductively, with $\parent{t_{j+1}}=t_j$ for each $j\in\N$.
Set $t_0=\emptyword$.
For the inductive step, assume that $n\in\N$ and $t_0, \dots, t_{n-1}$ are already defined so that $\parent{t_{j}}=t_{j-1}$ for each $1\le j<n$ and for $k$ defined as above we have
\begin{equation}\label{ind-case}
\rho_{k+n-1}(g_{k+n})(\ctup{k}{k} \concat t_{n-1}) = 1.    
\end{equation}
Note that \eqref{base-case} and $t_0=\emptyword$ yield \eqref{ind-case} for $n=1$. 

The definitions of $\rho_{k+n-1}$ and $\Psi(T)$ give
\begin{equation}\label{eqn:from-def-of-rho}
\rho_{k+n-1}(g_{k+n})(\ctup{k}{k} \concat t_{n-1})  = 
	\sum_{\substack{t \in \treevertices{n}{T}\\  \parent{t} = t_{n-1}}} g_{k+n}(\ctup{k}{k} \concat t) \; \pmod{2}.
\end{equation}
By \eqref{ind-case} the left-hand side of \eqref{eqn:from-def-of-rho} is $1$, thus we can choose $t_n\in\treevertices{n}{T}$ such that $\parent{t_n} = t_{n-1}$ and $g_{k+n}(\ctup{k}{k} \concat t_n) = 1$. 
Using \eqref{eqn:groupbraiding}, we conclude that 
\begin{equation*}
	\rho_{k+n}(g_{k+n+1})(\ctup{k}{k} \concat t_n) = 
	\gbar{ \ctup{k+n}{k+n} }(\ctup{k}{k} \concat t_n) + 
	g_{k+n}(\ctup{k}{k} \concat t_n) = 1.
\end{equation*}
Hence, \eqref{ind-case} holds with $n-1$ replaced by $n$, completing the inductive step.
\end{proof}

\subsection{Spaces of homeomorphisms}
For compact metrizable $X$ and $Y$, let $\Cont(X,Y)$ denote the continuous maps from $X$ to $Y$. We equip $\Cont(X, Y)$ with the \emph{compact-open topology}, which is the topology generated by sets  $V(K, U) := \{ f \in \Cont(X, Y) : f(K) \subseteq U \}$ where $K \subseteq X$ is compact and $U \subseteq Y$ is open. If $X=Y$, composition $(g,f)\mapsto g\circ f$ is continuous on $\Cont(X,X)^2$. The \emph{uniform metric} on $\Cont(X,Y)$ is given for $f,g\in \Cont(X, Y)$ by
$$
d_u(f, g) := \sup \{ d_Y(f(x), g(x)) : x \in X \}.
$$
Here $d_Y$ is a compatible metric on $Y$. The uniform metric is complete and induces the compact-open topology; the resulting space $\Cont(X,Y)$ is separable, hence Polish.

Let $\Homeo(X) := \{ f \in \Cont(X, X) : f \textrm{ is a homeomorphism}\}$. The metric
$$
d_{\Homeo}(f, g) := \max \{ d_u(f, g) , d_u(f^{-1}, g^{-1}) \}$$
 is complete and compatible with the subspace topology  on $\Homeo(X)$ induced from $\Cont(X, X)$. Under composition, $\Homeo(X)$ is a Polish topological group. Hence, the conjugacy relation
 \[
[\cong]_{\Homeo(X)}
:=\{(f,g)\in\Homeo(X)\times\Homeo(X):\text{ $\exists \varphi\in\Homeo(X)$ with $\varphi\circ f=g\circ \varphi$}\}
 \]
is an analytic subset of $\Homeo(X)\times\Homeo(X)$.

\subsection{Topological dynamical systems}\label{subsec:tds}
For a compact metrizable space $X$, fix a compatible metric $d_X$, also
written $d$. Write $d_u$ for the associated uniform metric on
$\Cont(X,X)$ and $d_H$ for the Hausdorff metric on $\K(X)$. 
A \emph{topological dynamical system} (\emph{TDS}) is a pair $(X,f)$,
where $X$ is a nonempty compact metrizable space and $f\colon X\to X$
is a homeomorphism. It is a \emph{Cantor TDS}, or \emph{Cantor system},
if $X$ is a Cantor space.

Let $(X,f)$ and $(X',f')$ be two TDSs. A continuous map $h\colon X\to X'$ is \emph{equivariant} if $h\circ f=f'\circ h$. We say that $(X,f)$ and $(X',f')$ are \emph{(topologically) conjugate}
if there exists an equivariant homeomorphism $h\colon X\to X'$.
In this case, we call $h$ a \emph{conjugacy}
between $(X,f)$ and $(X',f')$ and we write $(X,f)\cong (X',f')$. A surjective equivariant map $h\colon X\to X'$
is called a \emph{factor map}. If there is a factor map from $(X,f)$ to $(X',f')$, then we say that $(X,f)$ is an \emph{extension} of $(X',f')$, while  $(X',f')$ is a \emph{factor} of $(X,f)$. 
An injective equivariant map $h\colon X\to X'$ is an \emph{embedding}.
If such a map exists, $(X,f)$ is a \emph{subsystem} of $(X',f')$;
the subsystem is \emph{proper} if $h$ is not surjective. We identify
subsystems of $(X,f)$ with nonempty closed subsets $Y\subseteq X$
satisfying $f(Y)=Y$, and set
\begin{align*}
\K^f(X):=&\{Y\subseteq X:\text{$Y\neq\emptyset$, $Y=\overline{Y}$, and $f(Y)=Y$}\},\\
\K^f_p(X):=& \{Y \in \K^f(X) : Y\text{ is a perfect subset of $X$} \}.
\end{align*}
We give these spaces the Vietoris topology inherited from $\K(X)$. 
The space $\K^f(X)$ (subsystems of $(X,f)$) is a nonempty closed subset of $\K(X)$, and the space $\K^f_p(X)$ is a $G_{\delta}$-set in $\K(X)$, 
hence it is a Polish space.  

A TDS is \emph{minimal} if it has no proper subsystems. We write $\Min(X)\subseteq \Homeo(X)$ for the set of all $f\in\Homeo(X)$ such that $(X,f)$ is minimal. We write $\Min^f(X)\subseteq \K^f(X)$ for the space of minimal subsystems of $(X,f)$, and $\Min^f_p(X):=\Min^f(X)\cap\K_p(X)$ for the space of perfect minimal subsystems. We include a proof of the following standard fact. 

\begin{prop}\label{prop:min-gdelta}
If $X$ is a compact metrizable space, then  $\Min(X)$ is a $G_{\delta}$-set in $\Homeo(X)$.
\end{prop}
\begin{proof}Given $n,m\in\N$ we set
\[G_{n,m}=\{f\in\Homeo(X):\text{$\{x,f(x),\ldots,f^{n-1}(x)\}^{1/m}=X$ for every $x\in X$}\}.\]
Thus $f\in G_{n,m}$ precisely when $\{x,f(x),\ldots,f^{n-1}(x)\}$
is $1/m$-dense in $X$ for every $x\in X$.
Using a well-known characterization of minimality, 
we have
\begin{equation}\label{eq:min}
\Min(X)=\bigcap_{m=1}^\infty\bigcup_{n=1}^\infty G_{n,m}.
\end{equation}
By \eqref{eq:min}, it suffices to show that each $G_{n,m}$ is open.
Fix $m,n\in\N$ and $f\in G_{n,m}$. The map
$\Theta\colon X\times X\to\R$ given by
\[
    \Theta(x, y) = \min\{d(f^i (x), y) : i=0, \dots, n-1\}
\]
is continuous and has values contained in the interval $[0,1/m)$. By compactness of $X\times X$, the map $\Theta$ attains its maximum $c < 1/m$. Let $g\in\Homeo(X)$ be sufficiently close to $f$ to guarantee that
\[
d_u(f^i,g^i)<1/m - c\text{ for }i=0, \dots, n-1.
\]
The triangle inequality gives $g\in G_{n,m}$. 
\end{proof}
It follows that $\Min(X)$ with the subspace topology inherited from $\Homeo(X)$ is a Polish space and 
the conjugacy relation
 \[
[\cong]_{\Min(X)}
:=\{(f,g)\in\Min(X)\times\Min(X):
\text{ there exists $\varphi\in\Homeo(X)$ with $\varphi\circ f=g\circ \varphi$}\}
 \]
is an analytic subset of $\Min(X)\times\Min(X)$.

For a TDS $(X,T)$, the \emph{automorphism group} $\Auto{X,T}$ consists
of the homeomorphisms of $X$ that commute with $T$. 
Note that $\gengroup{T} := \{ T^n \colon n \in \Z \}$
is a normal subgroup of $\Aut(X,T)$. 
We define the \emph{reduced automorphism group of $(X,T)$} to be the quotient group $\AutQ{X,T} := \Auto{X,T} / \gengroup{T}$. 

Let $(X,T)$ and $(Y,S)$ be TDSs and let $\pi\colon X\to Y$ be a factor map.
We say that $\phi\in\Auto{X,T}$ and $\psi\in \Auto{Y,S}$ are \emph{compatible over $\pi$} or \emph{$\pi$-compatible} (or simply \emph{compatible} if $\pi$ is clear from the context) if $\pi \phi = \psi \pi$.
Given $\phi\in \Auto{X,T}$, a compatible $\psi\in\Auto{Y,S}$ need not exist.
If it exists, surjectivity of $\pi$ determines it uniquely by
$\psi(\pi(x))=\pi(\phi(x))$ for $x\in X$.
We then write $\rhofactor{\pi}{\phi}:=\psi$, viewing $\rho_\pi$
as a function defined on a subset of $\Auto{X,T}$.
A compatible $\psi$ exists if and only if $\phi$ and $\phi^{-1}$ preserve the fibres of $\pi$, that is, if and only if
\[
\pi(x)=\pi(x')\quad\Longleftrightarrow\quad
\pi(\phi(x))=\pi(\phi(x'))\qquad(x,x'\in X).
\]
If the fibres are preserved, $\psi(\pi(x)):=\pi(\phi(x))$ is well
defined and continuous because $\pi$ is a quotient map. Applying the
same construction to $\phi^{-1}$ gives its continuous inverse.
Equivariance of $\pi$ and $\phi$ implies that $\psi$ commutes with $S$.
Conversely, compatibility with an automorphism $\psi$ gives the displayed
equivalence.
The automorphisms admitting a compatible automorphism of $(Y,S)$ form a subgroup $\Aut_\pi(X,T)$ of $\Auto{X,T}$ containing $\gengroup{T}$, since $\rhofactor{\pi}{\textrm{id}_X} = \textrm{id}_Y$, 
$\rhofactor{\pi}{\phi_1}\circ \rhofactor{\pi}{\phi_2} = \rhofactor{\pi}{\phi_1 \phi_2}$ whenever both sides are defined, $\rhofactor{\pi}{\phi}^{-1}$ is compatible with $\phi^{-1}$ whenever $\rhofactor{\pi}{\phi}$ is defined (the displayed fibre criterion is symmetric under $\phi\mapsto\phi^{-1}$), and $\rhofactor{\pi}{T^n}=S^n$ for every $n\in\Z$. 
Thus $\rho_{\pi} \colon \Aut_\pi(X,T) \to \Auto{Y,S}$ is a group homomorphism, and it induces a group homomorphism of the quotients,
$$
\rho'_{\pi} \colon \Aut_\pi(X,T)/\gengroup{T}\to\AutQ{Y,S},\qquad \rho'_{\pi}( \phi \gengroup{T}):= \rhofactor{\pi}{\phi} \gengroup{S}.
$$
We say that $\pi$ is \emph{\strongfactor} if $\Aut_\pi(X,T)=\Auto{X,T}$, that is, if $\rhofactor{\pi}{\phi}$ exists for every $\phi \in \Auto{X,T}$; in this case $\rho'_{\pi} \colon \AutQ{X,T} \to \AutQ{Y,S}$.
As observed in \cite{DDMP}, $\rho_\pi$ need be neither surjective
nor injective, even when defined on all of $\Aut(X,T)$. Indeed, the group of automorphisms of a Sturmian extension $(X_\alpha,\sigma)$ of an irrational rotation $R_\alpha$ of the circle $\mathbb{S}^1$ is isomorphic to $\Z$ by \cite{Olli}. Furthermore, the factor map from $(X_\alpha,\sigma)$ to the irrational rotation is compatible with $\Aut(X_\alpha,\sigma)$ by \cite[Lemma 5.7]{DDMP}, while the group of automorphisms of an irrational rotation of the
circle is isomorphic to $\mathbb{S}^1$. For noninjectivity, take the projection $\pi$ of a TDS $(X,T)$ with
$T\ne\text{id}_X$ onto the one-point system. Then $\rho_\pi$ is defined
on all of $\Aut(X,T)$ and sends $T$ to the identity. 
\subsection{General shift spaces and subshifts}
Let $\bA$ be a compact metrizable space. We endow $\bA^{\Z}$ with the product topology. We define the shift map $\sigma\colon \bA^{\Z}\to \bA^{\Z}$ by  $\sigma(\bx)_n = \bx_{n+1}$, where $\bx\in\bA^{\Z}$ and $n\in\Z$. Then $\sigma$ is a homeomorphism and the TDS $(\bA^{\Z},\sigma)$ is the \emph{full shift space over $\bA$}.
A \emph{shift space over $\bA$} is a nonempty closed set
$X\subseteq\bA^{\Z}$ with $\sigma(X)=X$. We write
$\sigma_X=\sigma|_X$, or simply $\sigma$ when the domain is clear.
We call $X$ a \emph{Cantor shift space} if its alphabet $\bA$ is a
Cantor space.  
By dynamical properties of a shift space $X$, we implicitly mean the properties of the topological dynamical system $(X,\sigma_X)$.
In particular, $X\cong Y$ means $(X,\sigma_X)\cong(Y,\sigma_Y)$.
We use the restrictions of conjugacy to $\ShAp$ and
$\Min^{\sigma}_p(\bA^{\Z})$:
\begin{align*}
    [\cong]_{\ShAp}&=\{(X,Y)\in\ShAp\times\ShAp: X\cong Y\},\\
    [\cong]_{\Min^{\sigma}_p(\bA^{\Z})}&=\{(X,Y)\in\Min^{\sigma}_p(\bA^{\Z})\times\Min^{\sigma}_p(\bA^{\Z}): X\cong Y\}.
\end{align*}

We reserve the term \emph{subshift} for a shift space over a discrete compact space $A$ with at least two elements. We refer to the finite set $A$ as the \emph{alphabet} of the subshift and elements of $A$ are its \emph{symbols}. If $A$ is a finite discrete space with at least two elements, then $A^{\Z}$ is a Cantor space, so $(A^{\Z},\sigma)$, as well as every subshift in $\K_p^{\sigma}(A^{\Z})$ and $\Min^{\sigma}_p(A^{\Z})$, is a Cantor system. 
Thus subshifts are a special case of shift spaces. A perfect shift
space over $\C$ is a Cantor system because $\C^{\Z}$ is homeomorphic
to $\C$.
In contrast to subshifts, Cantor shift spaces need not be expansive.

Every TDS $(X,f)$ has a standard representation as a shift space over $X$.
The \emph{orbit shift} of $(X,f)$ is the shift space $X_f$ over $X$ given by
\[
X_f := \{ \bx \in X^{\Z} : \bx_{i+1} = f(\bx_i) \textrm{ for all } i \in \Z \}.
\]
\begin{lem}
	\label{lem:reduction1}
	 Let $(X,f)$ be a TDS. 
  \begin{enumerate}
      \item The map $X\ni x\mapsto (f^j(x))_{j\in\Z}\in X_f$ is a conjugacy between $(X,f)$ and $(X_f,\sigma)$.
    \item The map $\Homeo(X)\ni f \mapsto X_f\in \K^{\sigma}(X^{\Z})$ is a continuous reduction from $[\cong]_{\Homeo(X)}$ to $[\cong]_{\K^{\sigma}(X^{\Z})}$ and its restriction to $\Min(X)$ continuously reduces $[\cong]_{\Min(X)}$ to $[\cong]_{\Min^{\sigma}(X^{\Z})}$.
  \end{enumerate}
\end{lem}

\begin{proof}
For the first assertion, the inverse map is projection to coordinate zero.
For the second assertion, suppose $f_k\to f$ in $\Homeo(X)$.
Continuity of composition and inversion gives $f_k^j\to f^j$ uniformly
for each $j\in\Z$.
Therefore, for every $\eps>0$ and every $J\in\N$ for all sufficiently large $k$ and every $x\in X$ we have 
\[
d(f^j_k(x),f^j(x))<\eps/2\quad\text{for }|j|\le J.
\]
Hence, if $J$ is sufficiently large, we see that for every sufficiently large $k$ we have that $(f^j_k(x))_{j\in\Z}$
is $\eps$-close to $(f^j(x))_{j\in\Z}$ in $X^{\Z}$. We conclude that $X_{f_k}\to X_f$ as $k\to\infty$ in the topology induced by the Hausdorff metric on $\K(X^{\Z})$. This proves that the map $f\mapsto X_f$ is continuous. 
Now if $f, g \in \Homeo(X)$, then $(X, f) \cong (X_f, \sigma)$ and $(X, g) \cong (X_g, \sigma)$. Hence, 
$f \cong g$ if and only if $X_f \cong X_g$ and $f\in\Min(X)$ if and only if $X_f\in\Min^{\sigma}(X^{\Z})$. This shows that the map $f\mapsto X_f$ reduces $[\cong]_{\Homeo(X)}$ to $[\cong]_{\K^{\sigma}(X^{\Z})}$ and $[\cong]_{\Min(X)}$ to $[\cong]_{\Min^{\sigma}(X^{\Z})}$.
\end{proof}

\subsection{Polish spaces of Cantor systems}
A \emph{Polish space model} for Cantor systems is a family containing
a representative of every conjugacy class, equipped with a Polish
topology. We use two such models. The first is $\Homeo(\C)$ with the
topology of uniform convergence, together with the analytic relations
$[\cong]_{\Homeo(\C)}$ and $[\cong]_{\Min(\C)}$.

The second model uses the universal Cantor system $(\C^{\Z},\sigma)$:
by Lemma~\ref{lem:reduction1}, every Cantor system is conjugate to one
of its subsystems.
The perfect subsystems form the Polish space $\Shp$.
The minimal perfect subsystems form a Polish subspace as well:
\begin{lem}\label{lem:minimal-subsystems-polish}
If $Z$ is a compact zero-dimensional metrizable space and
$F\in\Homeo(Z)$, then $\Min^F(Z)$ is a $G_\delta$ subset of $\K^F(Z)$.
Consequently $\Min_p^F(Z)$ is Polish.
\end{lem}
\begin{proof}
Fix a countable clopen basis $\mathcal U$ of $Z$.
An invariant compact set $A$ is minimal precisely when, for each
$U\in\mathcal U$, either $A\cap U=\varnothing$ or
\[
 A\subseteq\bigcup_{|k|\le N}F^{-k}U\quad\text{for some }N\ge0.
\]
Indeed, minimality and compactness give a finite subcover whenever $U$
meets $A$. Conversely, this condition makes every orbit in $A$ meet
every basic open set that meets $A$. For a fixed $U$, the displayed
alternative is open in the hyperspace. Intersecting these open conditions
over $\mathcal U$ proves the claim. Perfectness is a $G_\delta$ condition,
and $\K^F(Z)$ is closed in the compact hyperspace.
\end{proof}
We use the conjugacy relations on these subsystem spaces. Their analyticity
will follow from the objectwise continuous transfer below and the analytic
conjugacy relation on $\Homeo(\C)$.

For Cantor systems, Lemma~\ref{lem:reduction1} takes the following form.

\begin{lem}\label{lem:reduction2}
	 Let $(\C,f)$ be a Cantor system. 
  \begin{enumerate}
      \item The map $\C\ni x\mapsto (f^j(x))_{j\in\Z}\in {\C}_f\subseteq\C^{\Z}$ is a conjugacy between $(\C,f)$ and $({\C}_f,\sigma)$.
    \item The map $\Homeo(\C)\ni f \mapsto {\C}_f\in \Shp$ is a continuous reduction from $[\cong]_{\Homeo(\C)}$ to $[\cong]_{\Shp}$ and its restriction to $\Min(\C)$ continuously reduces $[\cong]_{\Min(\C)}$ to $[\cong]_{\Min_p^{\sigma}(\C^{\Z})}$.
  \end{enumerate}
\end{lem}
\begin{proof}
Apply Lemma~\ref{lem:reduction1}. Since ${\C}_f$ is homeomorphic to
$\C$, it is perfect, so the images lie in the stated subsystem spaces.
\end{proof}
\begin{lem}\label{lem:canonical}
For each $A\in\K_p(\C)$ there is a canonically specified homeomorphism
$\varphi_A:\C\to A$ such that
\[
\K_p(\C)\ni A\longmapsto\varphi_A\in\Cont(\C,\C)
\]
is continuous, where $\varphi_A$ is viewed as a map with image $A$.
Moreover, for every $A$ and $n\ge1$ there are a neighborhood $U$ of $A$
and an integer $\ell$ such that, if $B\in U$, $y\in A$, $z\in B$, and
$y,z$ agree on their first $\ell$ entries, then $\varphi_A^{-1}(y)$ and
$\varphi_B^{-1}(z)$ agree on their first $n$ entries.
\end{lem}
\begin{proof}
A finite binary word $v$ is a prefix of $A$ if $[v]\cap A\ne\varnothing$.
Call it special if both $v0$ and $v1$ are prefixes of $A$.
Every prefix of $A$ has a unique shortest extension that is special.
Indeed, until the first split there is only one possible next symbol;
absence of any split would give an isolated point of $A$.
The shortest special extension of a prefix $v$ may be $v$ itself.

Define special words $v_u^A$, indexed by finite binary words $u$,
recursively. Let $v_\emptyword^A$ be the shortest special prefix of $A$.
Given $v_u^A$, let $v_{u i}^A$ be the shortest special extension of
$v_u^A i$, for $i\in\{0,1\}$. For each $n$, the sets
$A\cap[v_u^A]$, $|u|=n$, form a clopen partition of $A$; the partition
for $n+1$ refines that for $n$, and $|v_u^A|\ge |u|$.
Define $\varphi_A(x)$ as the unique point of
\[
\bigcap_{n\ge0}\bigl(A\cap[v_{x_1\ldots x_n}^A]\bigr).
\]
Nested compactness gives existence, and the increasing prefix lengths
give uniqueness. The partitions show that $\varphi_A$ is a continuous
bijection from $\C$ onto $A$, hence a homeomorphism onto $A$.

Fix $A,n$, and choose $\ell$ at least as large as the lengths of all
$v_u^A$ with $|u|\le n$. Agreement of the prefix sets of $A$ and $B$
through length $\ell+1$ determines the same splits through these
levels, so $v_u^B=v_u^A$ for $|u|\le n$. This agreement holds in a
Hausdorff neighborhood of $A$: each cylinder of length $\ell+1$ is
clopen and there are only finitely many such cylinders.
For every $x$, the points $\varphi_A(x)$ and $\varphi_B(x)$ then share
a prefix of length at least $n$, proving uniform continuity of the
assignment at $A$. If $y\in A$ and $z\in B$ share a prefix of length
$\ell$, they belong to the same member indexed by a word $u$ of length
$n$ of the two corresponding partitions. Their inverse images thus
share the prefix $u$, proving the final assertion.
\end{proof}

\begin{lem}\label{lem:transport}
Fix $F\in\Homeo(\C)$. For $A\in\K_p^F(\C)$ set
\[
f_A=\varphi_A^{-1}\circ F|_A\circ\varphi_A:\C\to\C.
\]
Then $A\mapsto f_A$ is continuous into $\Homeo(\C)$, and
$\varphi_A$ conjugates $(\C,f_A)$ to $(A,F|_A)$.
Consequently this assignment continuously reduces conjugacy on
$\K_p^F(\C)$ to conjugacy on $\Homeo(\C)$ and preserves minimality.
\end{lem}
\begin{proof}
The conjugacy assertion follows from the defining equation.
For continuity, fix $A$ and $n\ge1$. We seek a neighborhood of $A$
on which $f_A(x)$ and $f_B(x)$ agree on their first $n$ coordinates
for every $x$.
Apply the inverse-image assertion of Lemma~\ref{lem:canonical} to
obtain a neighborhood of $A$ and an integer $\ell$.
Uniform continuity of $F$ and continuity of $B\mapsto\varphi_B$ imply
that, after shrinking this neighborhood, $F\varphi_A(x)$ and
$F\varphi_B(x)$ agree on their first $\ell$ coordinates for every $x$.
The inverse-image assertion now gives agreement of the first $n$
coordinates of $f_A(x)$ and $f_B(x)$. Applying the same argument to
$F^{-1}$ proves uniform convergence of the inverses as well, hence
continuity in $\Homeo(\C)$.
\end{proof}

\begin{lem}\label{lem:conjugacy-hyperspace}
Let $h:(X,f)\to(Y,g)$ be a conjugacy of TDSs.
The induced map $A\mapsto h(A)$ is a homeomorphism of their subsystem
spaces, preserves the conjugacy relation, and restricts to the minimal
and perfect subsystem spaces.
\end{lem}
\begin{proof}
The maps $h$ and $h^{-1}$ induce inverse continuous maps of the compact
hyperspaces. Their restrictions carry invariant sets to invariant sets,
and their restrictions to each such set are conjugacies. Minimality
and absence of isolated points are invariant under conjugacy.
\end{proof}

\begin{thm} \label{thm:subshifts-and-systems-are-bireducible}
The conjugacy relations $[\cong]_{\Homeo(\C)}$ and $[\cong]_{\Shp}$ are continuously bireducible, that is, 
$[\cong]_{\Homeo(\C)}\leq ^2_{\cont}[\cong]_{\Shp}$ 
and 
$[\cong]_{\Shp}\leq ^2_{\cont}[\cong]_{\Homeo(\C)}$. The same holds for their restrictions to minimal homeomorphisms/subsystems, that is,
\[
[\cong]_{\Min(\C)}\leq ^2_{\cont}[\cong]_{\Min_p^\sigma(\C^{\Z})} 
\]
and 
\[[\cong]_{\Min_p^\sigma(\C^{\Z})}\leq ^2_{\cont}[\cong]_{\Min(\C)}. 
\]
\end{thm}

\begin{proof}
The map in Lemma~\ref{lem:reduction2} sends a homeomorphism to its orbit
shift, a conjugate system, and is continuous.
For the converse choose a fixed homeomorphism $h:\C^{\Z}\to\C$ and put
$F=h\sigma h^{-1}$. Define
\[
\Xi(X)=\varphi_{h(X)}^{-1}\circ F|_{h(X)}\circ\varphi_{h(X)}.
\]
The map $X\mapsto h(X)$ is continuous by Lemma~\ref{lem:conjugacy-hyperspace}, so Lemma~\ref{lem:transport} makes $\Xi$ continuous. The map
$h^{-1}\varphi_{h(X)}$ conjugates $(\C,\Xi(X))$ to $(X,\sigma)$.
Both assignments therefore preserve each represented system up to
conjugacy, and in particular preserve minimality. They also preserve
flip conjugacy: replacing either representative by a conjugate system
does not change whether it is conjugate to the other system or to its
inverse.
\end{proof}

\begin{rem}
If we restrict both models to the Polish subspaces determined by a conjugacy-invariant property (e.g.\ minimality), the same arguments give continuous bireducibility of the restricted conjugacy relations.     
\end{rem}

\subsection{Subshifts}

Let $X$ be a subshift. We define the \emph{language} of $X$ to be
\[\lang(X) := \{w \in A^{< \N} \colon w=x_{[i,j]} \textrm{ for some } x \in X \textrm{ and } i, j \in \Z\}.\]

Whenever $X$ is a subshift or a shift space, we write $\Auto{X}$ instead of $\Auto{X,\sigma}$ for the  \emph{group of automorphisms of $X$}, and $\AutQ{X}$ for the reduced automorphism group $\Auto{X}/\gengroup{\sigma_X}$. If $\pi\colon X\to Y$ is an \strongfactor{} factor map between shift spaces, then $\rho_\pi\colon\Auto{X}\to\Auto{Y}$ and $\rho'_\pi\colon\AutQ{X}\to\AutQ{Y}$ are the homomorphisms defined in Section~\ref{subsec:tds}.

An \emph{inverse system of subshifts} is a sequence
$(X_j,\pi_j)_{j\in\N}$ in which each $X_j$ is a subshift and each
$\pi_j\colon X_{j+1}\to X_j$ is a factor map.

 For $i>j$ we define $\pi_{i, j} \colon X_i \to X_j$ by $\pi_{i, j} := \pi_{j} \circ \pi_{j+1}\circ\cdots\circ\pi_{i-1}$, and we put $\pi_{j, j} := \textrm{id}$.

If all bonding maps $\pi_j$ are automorphism-compatible, then $(\AutQ{X_j},\rho'_{\pi_j})_{j\in \N}$ is an inverse system of groups.

The \emph{inverse limit} of $(X_j,\pi_j)_{j\in\N}$ is
\begin{equation} \label{eqn:defn-of-inverse-limit}
	\varprojlim (X_j,\pi_j)_{j\in\N} := \left\{ x \in \prod_{j \in \N} X_j \colon \pi_j(x_{j+1}) = x_j \textrm{ for all } j \ge 1 \right\}.
\end{equation}
Endowed with the topology inherited from $\prod_{j \in \N} X_j$, $\varprojlim (X_j,\pi_j)_{j\in\N} $ is compact.
We equip it with the coordinatewise shift
$\sigma\bigl((x_j)_{j\in\N}\bigr):=(\sigma x_j)_{j\in\N}$; since the
bonding maps are equivariant, $\sigma$ maps the inverse limit onto
itself, and the inverse limit is a TDS. All dynamical properties of an
inverse limit refer to this shift.
Its topology has the following description.
\begin{lem} \label{lem:clopen_x_basis}
If $X=\varprojlim (X_j,\pi_j)_{j\in\N}$ is the inverse limit of an inverse system of subshifts, then
\[
\left\{ V(k,U): k\in\N\text{ and $U \subseteq X_k$ clopen}\right\},
\]
where 
\begin{equation} \label{eqn:clopen_x_basis}
V(k, U) := \left\{ x \in X : x_k  \in U \right\}
\end{equation}
is a clopen basis for the topology on $X$. Furthermore, every clopen subset of $X$ is of the form $V(k, U)$ for some $k\in\N$ and clopen $U\subseteq X_k$.
\end{lem}

\begin{proof}
The subspace topology on $X$ has a basis of sets of the form 
$$
	\left\{x \in X \colon x_j \in U_j \textrm{ for all } j=1 \dots n \right\},
$$
for all $n \ge 1$, and $U_j \subseteq X_j$ open. 
Since $X_j$ are zero-dimensional, it is enough to consider only clopen $U_j$. 
The basis set above is of the form (\ref{eqn:clopen_x_basis}), since it is equal to 
$$
	\left\{x \in X \colon x_n \in \bigcap_{j = 1}^{n} \pi_{n, j}^{-1} (U_j) \right\}.
$$

For the second claim, let $A \subseteq X$ be clopen. 
Write $A$ as a union of basic sets. Compactness gives a finite
subcover, so $A = \bigcup_{i = 1}^{n} V(k_i, U_i)$ 
for some $n$, $k_i$ and $U_i$. Let $k := \max(k_1, \dots k_n)$. Then 
$$
A = \bigcup_{i = 1}^{n} V \left( k, \pi_{k, k_i}^{-1}(U_i) \right) = V \left( k, \bigcup_{i = 1}^{n} \pi_{k, k_i}^{-1}(U_i) \right),
$$
which is the desired form.
\end{proof}

The inverse limit is nonempty because the bonding maps are surjective:
any finite compatible tuple extends to a longer one, and compactness
gives an infinite compatible tuple. The same argument shows that every
coordinate projection is surjective. Lemma~\ref{lem:clopen_x_basis} shows
that the inverse limit is zero-dimensional. It embeds equivariantly in
$\C^{\Z}$ by encoding the product of the component alphabets in $\C$;
the embedding intertwines the coordinatewise shift with the shift of
$\C^{\Z}$.

\begin{lem}\label{lem:thesis-inverse-minimal}
The inverse limit $X=\varprojlim(X_j,\pi_j)_{j\ge1}$ is minimal if and
only if every $X_j$ is minimal. If these conditions hold and at least
one $X_j$ is infinite, then $X$ is a Cantor space.
\end{lem}
\begin{proof}
Suppose every $X_j$ is minimal. A nonempty basic set $V(k,U)$ is met
by the orbit of any $x\in X$, since the orbit of $x_k$ meets $U$.
Thus every orbit in $X$ is dense. Conversely, if $X$ is minimal, each
coordinate projection is an onto factor map. Given $y\in X_j$, choose
$x\in X$ projecting to $y$. The image of the dense orbit of $x$ is
the orbit of $y$, which is therefore dense in $X_j$.

If some $X_j$ is infinite, surjectivity of the coordinate projection
makes $X$ infinite. An infinite compact minimal system has no isolated
points: if $x$ were isolated, the dense orbit of every point would meet
$\{x\}$. Every point would then belong to the orbit of $x$, and all
points of $X$ would be isolated. Compactness would force $X$ to be
finite. Since $X$ is compact, metrizable, and zero-dimensional, it is
a Cantor space under the stated additional hypotheses.
\end{proof}

We also use the following language characterization of minimality.
\begin{lem} \label{lem:syndetic-minimal}
    Let $X$ be a subshift. The following are equivalent:
    \begin{enumerate}[label=(\roman*)]
        \item $X$ is minimal,
        \item for every $u \in \lang(X)$, there exists $n_u\in \N$ such that
            for every $v \in \lang(X)$, with $|v| \ge n_u$, we have that $u$ is a subword of $v$.
    \end{enumerate}
\end{lem}

\subsection{Odometers}
We recall the odometer notation used below.

A \emph{scale} is a strictly increasing sequence of positive integers $(L_n)_{n \in \N}$,
such that $L_n$ divides $L_{n+1}$ for all $n \ge 1$.
Consider the set
$$
	\mathcal{L} := \left\{ x \in \prod_{n \in \N} \Z / L_n \Z
	\colon
	x_{n+1} \equiv x_n \pmod{L_n} \textrm{ for all } n \ge 1\right\},
$$
equipped with the product topology. Note that $\cL$ is a Cantor space.
For $x, y \in \cL$, we define $x+y$ to be the unique element of $\cL$ such that
$ (x+y)_n \equiv x_n + y_n \pmod{L_n} \textrm{ for all } n \ge 1 $.
With this operation, $\cL$ is an abelian topological group. Its neutral element is $\mathbf{0} = (0, 0, 0\dots)$.
In other words, $\cL$ is the inverse limit of the groups $\Z / L_n \Z$.
Define $\mathbf{1} := (1, 1, 1 \dots) \in \cL$, and a map $f \colon  \cL \to \cL$ by $f(x) = x + \mathbf{1}$.
The TDS $(\cL,f)$ is the \emph{odometer} with scale $(L_n)_{n\in\N}$.
It is minimal.

Note that $\Z \ni k \mapsto f^k(\mathbf{0}) \in \cL$ is an injective group homomorphism.
We identify an integer $k \in \Z$ with the element $f^k(\mathbf{0}) \in \cL$.
An element $x\in\cL$ is an \emph{integer} if
$x=f^k(\mathbf{0})$ for some $k\in\Z$.

If $x \in \cL$, we may also use $x \pmod{L_n}$ to refer to its $n$-th coordinate.

\section{Conjugacy between inverse limits}\label{sec:conj-inv-lims}

We give a criterion for conjugacy between inverse limits of subshifts.
When a point $y$ of a subshift already carries a subscript, we write
$y(k)$ for its $k$th coordinate; thus $\Phi(x)_I(k)$ is the $k$th
coordinate of the $I$th component of $\Phi(x)$.

\begin{lem}\label{lem:small-inv-lim-isomorphism}
Let $(X_j,\pi_j)_{j\ge1}$ and $(\Xother_j,\piother_j)_{j\ge1}$ be
inverse systems of subshifts over finite alphabets, with limits
$X=\varprojlim(X_j,\pi_j)$ and
$\Xother=\varprojlim(\Xother_j,\piother_j)$.
\begin{enumerate}[label=(\roman*)]
\item\label{part-a} If each $f_j:X_j\to\Xother_j$ is a conjugacy and
$f_j\pi_j=\piother_jf_{j+1}$, then
$F(x)_j=f_j(x_j)$ defines a conjugacy $F:X\to\Xother$.
\item\label{part-b} If $\Phi:X\to\Xother$ is a conjugacy, then for each
$I\ge1$ there are $J\ge I$ and a factor map
$\phi_I:X_J\to\Xother_I$ such that
$\Phi(x)_I=\phi_I(x_J)$ for every $x\in X$.
\end{enumerate}
\end{lem}
\begin{proof}
In (i), compatibility makes $F$ well-defined, and continuity and equivariance
hold coordinatewise. The inverse maps satisfy
$\pi_jf_{j+1}^{-1}=f_j^{-1}\piother_j$, so they define its continuous inverse.

For (ii), let $A$ be the finite alphabet of $\Xother_I$.
The sets $\{\bar x:\bar x_I(0)=a\}$, $a\in A$, form a clopen partition
of $\Xother$. By Lemma~\ref{lem:clopen_x_basis}, their inverse images under
$\Phi$ depend on one common coordinate $J\ge I$. Thus there is a clopen
partition $(U_a)_{a\in A}$ of $X_J$, allowing empty members, with
\[
 \Phi(x)_I(0)=a\quad\Longleftrightarrow\quad x_J\in U_a.
\]
Surjectivity of $X\to X_J$ ensures that the sets $U_a$ partition $X_J$. Define $\phi_I(z)(k)=a$ when $\sigma^kz\in U_a$.
This is continuous and equivariant and satisfies the asserted identity.
Every $z\in X_J$ lifts to $X$, so its image belongs to $\Xother_I$.
Surjectivity of $\Phi$ and of $\Xother\to\Xother_I$ makes $\phi_I$ onto.
\end{proof}

\begin{defn}\label{def:blended}
An inverse system of subshifts $(X_j,\pi_j)_{j\in\N}$ is
\emph{blended} if for every $i\ge j\ge1$ and every factor map
$\zeta:X_i\to X_j$ there is $\phi\in\Auto{X_j}$ such that
$\zeta=\phi\pi_{i,j}$.
\end{defn}
It suffices to assume this property for $i>j$. If $r:X_j\to X_j$ is a
factor map, apply the property to $r\pi_j$ and cancel the surjection
$\pi_j$ to obtain $r\in\Auto{X_j}$. We use the term ``blended'' for this factor-map decomposition property.

\begin{lem} \label{lem:inv-lim-isomorphism}
Let $(X_j,\pi_j)_{j\in\N}$ be a blended inverse system of subshifts,
and let $\piother_j\colon X_{j+1}\to X_j$ be a factor map for each
$j\in\N$. Set $X=\varprojlim(X_j,\pi_j)_{j\in\N}$ and
$\Xother=\varprojlim(X_j,\piother_j)_{j\in\N}$. The following are equivalent:
\begin{enumerate}[label=(\roman*)]
	\item\label{cond-i} There are automorphisms $f_j\in\Auto{X_j}$, $j\ge1$, satisfying
$f_j\pi_j=\piother_jf_{j+1}$ for every $j\ge1$. 
	\item\label{cond-ii} $X \cong \Xother$. 
\end{enumerate}
\end{lem}

\begin{proof}
Compatible automorphisms $(f_j)$ give a coordinatewise conjugacy by
Lemma~\ref{lem:small-inv-lim-isomorphism}(i).
Conversely, let $\Phi:X\to\Xother$ be a conjugacy. For each $j$,
Lemma~\ref{lem:small-inv-lim-isomorphism}(ii) factors its $j$th coordinate
through a factor map $\phi_j:X_J\to X_j$ for some $J\ge j$.
By the blended property, $\phi_j=f_j\pi_{J,j}$ with
$f_j\in\Auto{X_j}$. Thus $\Phi(x)_j=f_j(x_j)$.
For every $x\in X$,
\[
\piother_j f_{j+1}(x_{j+1})
 =\piother_j\Phi(x)_{j+1}=\Phi(x)_j
 =f_j\pi_j(x_{j+1}).
\]
Surjectivity of the projection $X\to X_{j+1}$ gives the required equality
of maps on $X_{j+1}$.
\end{proof}

For the convergence estimates below, fix a compatible metric $d_{\C}$
 on $\C$ bounded by one, and put $\mathscr S=\C^{\Z}$. We use
\begin{align}
 d_{\Sigma}(x,y)
   &:=\sum_{i\in\Z}2^{-|i|}d_{\C}(x_i,y_i),
      &&x,y\in\mathscr S,\label{eq:ambient-shift-metric}\\
 \varrho(x,y)
   &:=\sup_{j\ge1}2^{-j}d_{\Sigma}(x_j,y_j),
      &&x,y\in\mathscr S^{\N}.\label{eq:ambient-product-metric}
\end{align}
The uniformly small tails show that these metrics induce the respective
product topologies, and $d_{\Sigma}\le3$. We denote their Hausdorff
metrics by $d_{\Sigma,H}$ and $\varrho_H$, respectively. These metric choices do not change the hyperspace topologies.

\begin{lem}[Continuity of inverse limits]\label{lem:inv-lim-continuity}
For each $k\in\N_0$, let $(A_j^{(k)},q_j^{(k)})_{j\ge1}$ be an
inverse system of nonempty compact shift spaces in $\mathscr S$, with
continuous surjective bonding maps commuting with the shifts. Write
\[
 Y^{(k)}=\varprojlim(A_j^{(k)},q_j^{(k)})
       \subseteq\mathscr S^{\N}.
\]
Suppose that for each $j\ge1$ the following hold:
\begin{enumerate}[label=\textup{(\alph*)}]
\item\label{ilc:components}
$d_{\Sigma,H}(A_j^{(k)},A_j^{(0)})\to0$ as $k\to\infty$.
\item\label{ilc:common-maps}
There are a closed set $E_j\subseteq\mathscr S$, a continuous map
$Q_j:E_j\to\mathscr S$, and $k_j\in\N$ such that
\[
 A_{j+1}^{(k)}\subseteq E_j,\qquad
 q_j^{(k)}=Q_j|_{A_{j+1}^{(k)}}
 \quad\text{for }k=0\text{ and for every }k\ge k_j.
\]
\end{enumerate}
Then $\varrho_H(Y^{(k)},Y^{(0)})\to0$ as $k\to\infty$.
The threshold $k_j$ in \textup{(b)} may depend on $j$. Commutation
with the shifts is used only to ensure that the limits are shift spaces.
\end{lem}
\begin{proof}
Fix $\eta>0$ and choose $m\ge2$ such that $3\cdot2^{-(m+1)}<\eta$.
Let $E\subseteq E_{m-1}$ consist of the points at which all the
compositions
\[
 R_j=Q_j\circ\cdots\circ Q_{m-1},\qquad 1\le j<m,
\]
are defined. To construct this domain, start with $E_{m-1}$ and
successively require each intermediate image to belong to the next
closed domain $E_{m-2},\ldots,E_1$. Each step restricts a continuous
map to a closed subset of its current compact domain. Thus $E$ is
compact and every $R_j:E\to\mathscr S$ is continuous. For $m=2$,
this just means $E=E_1$ and $R_1=Q_1$.

By \textup{(b)}, $A_m^{(0)}\subseteq E$, and the same inclusion holds
for all sufficiently large $k$. Moreover, for a point $x\in Y^{(k)}$
with such $k$, or with $k=0$, we have $x_j=R_j(x_m)$ for $j<m$.
Uniform continuity of the finitely many $R_j$ gives
$\delta\in(0,\eta)$ such that
\[
 z,z'\in E,\quad d_{\Sigma}(z,z')<\delta
 \quad\Longrightarrow\quad
 d_{\Sigma}(R_j(z),R_j(z'))<\eta\quad(1\le j<m).
\]
Choose $k_0$ large enough that \textup{(b)} holds simultaneously for
$j<m$ whenever $k\ge k_0$, and that
$d_{\Sigma,H}(A_m^{(k)},A_m^{(0)})<\delta$ for all $k\ge k_0$.

For $k\ge k_0$ and $x\in Y^{(0)}$, choose $z\in A_m^{(k)}$ with
$d_{\Sigma}(x_m,z)<\delta$. Surjectivity of the coordinate projection
$Y^{(k)}\to A_m^{(k)}$ gives $y\in Y^{(k)}$ with $y_m=z$.
For $j<m$ the preceding estimate applies to $R_j(x_m)$ and $R_j(z)$,
so $d_{\Sigma}(x_j,y_j)<\eta$; this also holds for $j=m$.
For $j>m$,
\[
 2^{-j}d_{\Sigma}(x_j,y_j)\le3\cdot2^{-(m+1)}<\eta.
\]
It follows that $\varrho(x,y)<\eta$. Interchanging $Y^{(0)}$ and
$Y^{(k)}$ proves the other Hausdorff inclusion, so
$\varrho_H(Y^{(k)},Y^{(0)})\le\eta$. Since $\eta>0$ was arbitrary,
the conclusion follows.
\end{proof}

\section{The reduction}\label{sec:reduction}
Recall that a finite word of length $L$ is written $w=w_0\ldots w_{L-1}$,
so that $w_{[0,L)}=w$ in the interval notation of
Section~\ref{subsec:words}. For a finite word $w=w_0\ldots w_{L-1}$, let
$\rev{w}=w_{L-1}\ldots w_0$. For a bi-infinite sequence $x$, let
$\rev{x}_i=x_{-i}$. These operations are applied elementwise to word sets
and shift spaces. Reversal is an involution and conjugates the left
shift to its inverse; thus $(\rev{X},\sigma)\cong(X,\sigma^{-1})$.

Given a tree $T$, we construct an inverse system of minimal subshifts
$(X_j,\pi_j)_{j\in\N}$ whose factor maps and conjugacies are described
by $\Psi(T)$. Its inverse limit $X$ is a Cantor minimal system. We then
choose $\alpha_j\in\Auto{X_j}$ and form a second inverse limit
$\Xother$ with bonding maps $\piother_j=\alpha_j\pi_j$.
The preceding lemmas will show that $X\cong\Xother$ if and only if
$T$ has an infinite branch. We also arrange that $X$ is never conjugate
to $\rev{\Xother}$. The resulting map reduces ill-foundedness to both
conjugacy and flip conjugacy.

\begin{defn}\label{def:adjacent}
An inverse system of subshifts $(X_j,\pi_j)_{j\ge1}$ is
\emph{adjacent to $T\in\Trees$} if its bonding maps are
\strongfactor{} and the inverse system of groups
$(\AutQ{X_j},\rho'_{\pi_j})_{j\in\N}$ is conjugate to the
$T$-directed inverse system of groups, that is, there are group
isomorphisms
\[
 \phi_j:\AutQ{X_j}\longrightarrow\treegroup{j}{T},\qquad
 \treehom{j}{T}\phi_{j+1}=\phi_j\rho'_{\pi_j}\quad(j\ge1).
\]
Here $\rho'_{\pi_j}:\AutQ{X_{j+1}}\to\AutQ{X_j}$ is the homomorphism induced by $\pi_j$ (Section~\ref{subsec:tds}).
When these isomorphisms have been fixed, we use them to identify the
two group systems.
\end{defn}

\begin{lem} \label{lem:main-reduction}
Let $T\in\Trees$, let $(X_j,\pi_j)_{j\in\N}$ be a blended inverse system of subshifts adjacent to $\Psi(T)$, and let $X$ be its inverse limit. For every $j\in \N$, let $\alpha_j \in \Auto{X_j}$ with 
$\alpha_j \gengroup{\sigma} = \gbar{ \ctup{j}{j} }$, and let $\Xother$ be the inverse limit of $(X_j,\alpha_j\pi_j)_{j\in\N}$. 
Then $X$ and $\Xother$ are conjugate
if and only if $T$ has an infinite branch.

Furthermore, if $\rev{X_1}$ is not a factor of $X_j$ for any $j \in \N$, then $X$ and $\Xother$ are flip conjugate 
if and only if $T$ has an infinite branch.
\end{lem}

\begin{proof}
Throughout the proof we write the abelian groups $\AutQ{X_j}$
additively, in accordance with their identification with the groups
$\treegroup{j}{\Psi(T)}$ in Definition~\ref{def:adjacent}.

Assume that $X$ and $\Xother$ are conjugate. By Lemma \ref{lem:inv-lim-isomorphism}, there exists a sequence $f_j \in \Auto{X_j}$ 
such that $ \alpha_j \pi_j f_{j+1} = f_j \pi_j$ for $j \ge 1$. 
Since $\pi_j f_{j+1}=\rhofactor{\pi_j}{f_{j+1}}\pi_j$, cancelling
the surjection $\pi_j$ gives $\alpha_j\rhofactor{\pi_j}{f_{j+1}}=f_j$
in $\Auto{X_j}$. Passing to $\AutQ{X_j}$ gives
$
	\gbar{ \ctup{j}{j} } + \treehom{j}{\Psi(T)}(f_{j+1} \gengroup{ \sigma } ) = f_j \gengroup{ \sigma }, 
$
which by Lemma \ref{lem:trees-and-groups} implies that $T$ has an infinite branch.

Conversely, suppose $T$ has an infinite branch. By
Lemma~\ref{lem:trees-and-groups}, choose
a sequence $g_j \in \treegroup{j}{\Psi(T)}$ such that $ \gbar{\ctup{j}{j}} + \treehom{j}{\Psi(T)}(g_{j+1}) = g_j $ for $j \ge 1$. 
Viewing $g_j$ as elements of $\AutQ{X_j}$, we can write $g_j = h_j \gengroup{ \sigma } $ for some $h_j \in \Auto{X_j}$. 
The last equation becomes
$$
	\alpha_j \gengroup{ \sigma } + \rhofactor{\pi_j}{h_{j+1}} \gengroup{ \sigma } = h_j \gengroup{ \sigma },
$$
which implies $\alpha_j \pi_j h_{j+1} = h_j \sigma^{n_j} \pi_j $ for some $n_j \in \Z$. 
Set $S_j=\sum_{i<j}n_i$ and $f_j=h_j\sigma^{-S_j}$.
Since $S_{j+1}=S_j+n_j$ and the factor maps commute with the shifts,
\[
\piother_j f_{j+1}
 =\alpha_j\pi_jh_{j+1}\sigma^{-S_{j+1}}
 =h_j\sigma^{n_j-S_{j+1}}\pi_j=f_j\pi_j.
\]
Lemma~\ref{lem:inv-lim-isomorphism} now gives $X\cong\Xother$.

Under the additional assumption, suppose for a contradiction that
$(X,\sigma)\cong(\Xother,\sigma^{-1})\cong(\rev{\Xother},\sigma)$. 
The system $(\rev{\Xother}, \sigma)$ is the inverse limit of the inverse system
$(\rev{X_j},\textrm{rev} \circ \piother_j \circ \textrm{rev})_{j\in \N}$.
By Lemma \ref{lem:small-inv-lim-isomorphism}(ii) applied to $X$ and $\rev{\Xother}$,
we have a factor map $X_j \to \rev{X_1}$ for some $j$, but this contradicts the
additional assumption. So 
\begin{equation*}
    X\text{ and } \Xother \textrm{ are flip conjugate }
    \; \Leftrightarrow \;
    (X, \sigma) \cong (\Xother, \sigma)
    \; \Leftrightarrow \;
    T \textrm{ has an infinite branch.} \qedhere
\end{equation*}
\end{proof}

It remains to construct an inverse system satisfying the hypotheses
of Lemma~\ref{lem:main-reduction}.

\section{Construction sequences}\label{sec:constr-seq}

Let $\bA$ be an alphabet, and put $\bA^+:=\bigcup_{\ell\ge1}\bA^\ell$.
A \emph{code} over $\bA$ is a nonempty family $C\subseteq \bA^+$ of
nonempty finite words. Its elements are called \emph{code words}.
Write $\Seg(C)$ for the set of bi-infinite concatenations of words
from $C$. Thus $x\in\Seg(C)$ if and only if there is a sequence \[\ldots <t_{-2}<t_{-1}<t_0\le 0< t_1<t_2<\ldots\] such that $x_{[t_{j-1},t_j)}\in C$ for every $j\in \Z$.
For this topological assertion, equip $\bA$ with a Hausdorff topology.
If $C$ is finite, its word lengths are bounded, so the set of pairs
$(x,S)$ in which $S$ is the cut set of a $C$-segmentation of $x$ is
closed in $\bA^{\Z}\times\{0,1\}^{\Z}$. Since
$\{0,1\}^{\Z}$ is compact, its projection $\Seg(C)$ is closed in
$\bA^{\Z}$.
For $L\in\N$, an \emph{$L$-code} is a code $C\subseteq \bA^L$.
An $L$-code $C$ is \emph{nonoverlapping} if, whenever $u,v,w\in C$
and $u$ occurs in $vw$ at offset $r$, either $r=0$ and $u=v$, or
$r=L$ and $u=w$. In particular, no occurrence starts at an internal
offset $0<r<L$. Such a code is uniquely decomposable: two different
segmentations of the same bi-infinite sequence would put a codeword
at an internal offset in two consecutive words of the other
segmentation. When $3\mid L$, an $L$-code
has \emph{distinctive cores} if the restrictions of its distinct words
to $[L/3,2L/3)$ are distinct.

We use the alphabets $F_p:=\{0,1\}^p$ for
$p\in\N\cup\{\infty\}$, where $F_\infty:=\{0,1\}^{\N}=\C$.
A symbol is a column of $0$s and $1$s, with rows numbered from one,
top to bottom. A word of length $q$ over $F_p$ is a $p\times q$ \emph{block},
and a point of $F_p^{\Z}$ is an array with columns indexed by $\Z$.
The shift moves the columns one position to the left.

For $t\in\N$ with $t\le p$ and $x\in F_p$, write
$[x]_t:=x_1\dots x_t\in\{0,1\}^t$; when $p=\infty$, any finite $t$
is allowed. We apply $[\,\cdot\,]_t$ letterwise to words,
$[w_1\dots w_k]_t:=[w_1]_t\dots[w_k]_t$, and coordinatewise to points
of $F_p^{\Z}$. For sets of words and points, write
$[W]_t:=\{[w]_t:w\in W\}$ and $[X]_t:=\{[x]_t:x\in X\}$.
This truncation to the first $t$ rows commutes with the shift and with
concatenation. It is distinct from the cylinder notation $[w]$ in
Section~\ref{subsec:words}.

We use the following definition of a construction sequence.

\begin{defn}\label{defn:constrSeq}
Let $\bA$ be a finite alphabet. A \emph{construction sequence} with
scale $(L_n)_{n\in\N}$ is a sequence $\W=(W^n)_{n\ge n_0}$ of finite
word sets, where $n_0\ge1$, satisfying the following conditions for
every $n\ge n_0$:
\begin{enumerate}[label=(\alph*)]
  \item \label{def:constr-seq:1} $W^n\subseteq \bA^{L_n}$ contains at least two elements,
  \item \label{def:constr-seq:2} each element of $W^{n+1}$ is a concatenation of some words from $W^{n}$,
  \item \label{def:constr-seq:3} for every $u,v \in W^{n}$ and $w$ in $W^{n+1}$, the concatenation $uv$ appears as a subword of $w$, 
  \item \label{def:constr-seq:4} if $u,v,w\in W^n$ and $u$ occurs in $vw$ at position $j$, then
  either $u=v$ and $j=0$ or $u=w$ and $j=L_n$.
\end{enumerate}
\end{defn}
Only the sets $W^n$ with $n\ge n_0$ are part of the construction sequence.
Given a construction sequence $\mathcal{W}=\{W^n\}_{n \ge n_0}$, we define
\begin{align*}
X_{\mathcal{W}} &:=
\{ x \in \bA^{\Z} : \forall m\; x_{[-m , m]} \textrm{ is a subword of some } w \in \bigcup_{n \ge n_0} W^n \},\\
\Seg(W^n) &:= 
\{ \sigma^k(x) \in \bA^{\Z} : k\in \Z\text{ and }x \textrm{ is an infinite concatenation of words from } W^n \}.
\end{align*}
Foreman, Rudolph, and Weiss \cite{FRW} used the term ``construction
sequence'' for a related but different definition. To compare the
conditions in Definition~\ref{defn:constrSeq}, introduce
\begin{enumerate}[label=(c$'$),ref=(c$'$)]
    \item \label{def:constr-seq:3prim} \textit{for all $w \in W^{n}$ and $w'$ in $W^{n+1}$, $w$ is a subword of $w'$.}
\end{enumerate}
\noindent We refer to \ref{def:constr-seq:3prim} as \emph{faithfulness}
and to \ref{def:constr-seq:3} as \emph{double faithfulness}.
The conditions \ref{def:constr-seq:1}, \ref{def:constr-seq:2} and \ref{def:constr-seq:3prim}
together imply that $X_{\W}$ is minimal (Lemma~\ref{lem:thesis-faithfulness} below). For example, every aperiodic Toeplitz subshift can be obtained from
sequences satisfying these three conditions.

Condition~\ref{def:constr-seq:3} implies
\ref{def:constr-seq:3prim} and gives the intersection description in
Lemma~\ref{lem:constr_seq_as_intersection}. The inclusions
$\Seg(W^n)\supseteq\Seg(W^{n+1})$ show that deleting finitely many
initial terms leaves $\bigcap_{n\ge n_0}\Seg(W^n)$ unchanged.
The resulting subshift is therefore unchanged, and we allow any
initial index $n_0\ge1$.

\begin{lem}[Faithfulness and the generated language]
\label{lem:thesis-faithfulness}
Let $\bA$ be a finite alphabet, $(L_n)_{n\ge n_0}$ a scale, and
$W^n\subseteq \bA^{L_n}$ finite sets satisfying
\ref{def:constr-seq:1}, \ref{def:constr-seq:2}, and
\ref{def:constr-seq:3prim}. Let $\mathscr L$ be the set of finite
subwords of the words in $\bigcup_{n\ge n_0}W^n$, and define $X_{\W}$
by the same language rule as above. Then $X_{\W}$ is nonempty and
minimal, and its language is exactly $\mathscr L$.
\end{lem}
\begin{proof}
Fix $u\in\mathscr L$ and choose $n$ such that $u$ is a subword of a
word in $W^n$. Faithfulness implies that $u$ occurs in every word of
$W^{n+1}$ and hence in every word of $W^m$ for $m\ge n+1$.
For any $R\ge1$, choose such an $m$ with $L_m\ge R$. Since successive
length ratios are integers greater than one, a word of $W^{m+2}$
contains at least four consecutive $W^m$-blocks. An occurrence of $u$
in its second block has at least $R$ letters available on both sides
within that word. Center these occurrences at the same coordinates
and use compactness of $\bA^{\Z}$ to obtain a point of $X_{\W}$
containing $u$. Every finite window of the limit occurs in one of the
chosen higher-level words. This proves nonemptiness and
$\mathscr L\subseteq\lang(X_{\W})$; the reverse inclusion follows
from the definition of $X_{\W}$.

Every word in $\mathscr L$ of length at least $3L_{n+1}$ occurs in
some higher-level word tiled by $W^{n+1}$-words. It contains a complete
$W^{n+1}$-block and therefore contains $u$. The language criterion in
Lemma~\ref{lem:syndetic-minimal} now gives minimality. Neither
nonoverlap nor double faithfulness is used in this argument.
\end{proof}

\begin{lem} \label{lem:constr_seq_as_intersection}
Let $\mathcal W$ be a construction sequence. Then $X_{\mathcal W}$
is a minimal subshift and
 $X_{\mathcal{W}} = \bigcap_{n \ge n_0} \Seg(W^n)$.
\end{lem}

\begin{proof}
Each $\Seg(W^n)$ is a nonempty compact shift-invariant set, and these
sets are decreasing. Their intersection is therefore nonempty.
If $x\in X_{\W}$, then every window $x_{[-m,m]}$ occurs in a word of some
$W^{n'}$, and, by faithfulness, we may take $n'\ge n$; such a word is tiled by $W^n$-words. Hence each window
is covered by a tiling by $W^n$-words at some offset $r_m\in\{0,\ldots,L_n-1\}$ (the two blocks at the ends of the window may be cut). Some offset $r$ occurs for infinitely many $m$; for this $r$, every block $x_{[r+iL_n,\,r+(i+1)L_n)}$, $i\in\Z$, lies inside one of these windows and therefore belongs to $W^n$.
Thus $x\in\Seg(W^n)$ for every $n\ge n_0$.
Conversely, take $x$ in the intersection and choose $n$ with
$L_n\ge2m+1$. Its window $x_{[-m,m]}$ lies in a pair of consecutive
$W^n$-words. By double faithfulness this pair occurs in a $W^{n+1}$-word,
so $x\in X_{\W}$.

If $u$ belongs to the language of $X_{\W}$, take $n$ with $L_n\ge|u|$.
The word $u$ is a subword of a pair of $W^n$-words, and this pair occurs
in every $W^{n+1}$-word. Every point is tiled by these higher-level
words, so $u$ occurs with gaps bounded by $2L_{n+1}$.
Lemma~\ref{lem:syndetic-minimal} gives minimality.
\end{proof}

\begin{lem}\label{lem:finite-language}
Let $\W$ be a construction sequence and let $n\ge n_0$.
For $1\le q\le L_n$, the length-$q$ language of $X_{\W}$ is exactly
the set of length-$q$ subwords of $uv$ with $u,v\in W^n$.
In particular every $W^n$-word occurs in $X_{\W}$.
\end{lem}
\begin{proof}
One inclusion follows from the $W^n$-segmentation of every point.
For the converse, every pair $uv$ occurs in every $W^{n+1}$-word.
Any point of the nonempty set $X_{\W}$ contains such words and hence
contains $uv$ and each of its subwords.
\end{proof}

\begin{cor}[A common-stage Hausdorff estimate]
\label{cor:common-stage-hausdorff}
Let $j\ge1$ and let $\W=(W^n)$ and $\W'=(W'^n)$ be construction
sequences over $\{0,1\}^j$, viewed as a subset of $\C$ by appending
zeros. Suppose $W^N=W'^N$ at an index belonging to both sequences,
and let $L$ be the common length of these words. For the Hausdorff
metric induced by \eqref{eq:ambient-shift-metric},
\[
 d_{\Sigma,H}(X_{\W},X_{\W'})
       \le 2^{1-\lfloor L/2\rfloor}.
\]
No agreement at later stages is required.
\end{cor}
\begin{proof}
Put $h=\lfloor L/2\rfloor$ and take $x\in X_{\W}$.
Choose a boundary $k\in(-L,0]$ of a $W^N$-segmentation of $x$.
If $k\le-h$, put $J=[k,k+2L)$; otherwise put $J=[k-L,k+L)$.
In either case $J$ contains the integer interval $[-h,h]$.
Indeed, in the first case its left endpoint is at most $-h$ and
its last integer is at least $L$. In the second case its left
endpoint is at most $-L$ and its last integer is at least $L-h$.

The word $x|_J$ is a concatenation of two words of $W^N=W'^N$.
By double faithfulness and the occurrence argument in
Lemma~\ref{lem:finite-language}, this pair occurs in $X_{\W'}$.
Shift a point realizing the pair to obtain $y\in X_{\W'}$ with
$y|_J=x|_J$. Then
\[
 d_{\Sigma}(x,y)
 \le\sum_{|i|>h}2^{-|i|}=2^{1-h},
\]
since $d_{\C}\le1$. Interchanging the two construction sequences
proves the reverse Hausdorff inclusion.
\end{proof}

For an $L$-code $C$, an integer $c$ is a \emph{$C$-cut point} of
$x\in\Seg(C)$ if
\[
 x_{[c+iL,c+(i+1)L)}\in C\quad\text{for every }i\in\Z.
\]
The nonoverlap condition gives recognizability on this entire segmented
space, before any construction sequence or minimality is imposed.

\begin{lem}[Recognizability on segmented spaces]
\label{lem:unique-readability}
Let $\bA$ be a finite alphabet, let $L\ge1$, and let
$C\subseteq \bA^L$ be a nonempty nonoverlapping code. For each
$x\in\Seg(C)$, every occurrence of a word of $C$ starts at a
$C$-cut point, and the set of all $C$-cut points is a single coset
of $L\Z$. Its unique representative $k_C(x)\in(-L,0]$ is determined
by $x_{[-L+1,L)}$. Thus $k_C:\Seg(C)\to\Z$ is locally constant and
\[
 k_C(\sigma x)\equiv k_C(x)-1\pmod L.
\]
\end{lem}
\begin{proof}
Choose a $C$-cut point $c$ of $x$. Suppose a word $u\in C$ occurs
at position $p$, and choose $i\in\Z$ with
$c+iL\le p<c+(i+1)L$. This occurrence lies in the two consecutive
codewords occupying $[c+iL,c+(i+2)L)$, at offset $p-c-iL$.
Nonoverlap forces that offset to be zero. Therefore every occurrence
starts in $c+L\Z$. Conversely, every position in $c+L\Z$ is a
$C$-cut point, and every cut point starts an occurrence. This proves
both assertions about cuts.

There is exactly one cut point in $(-L,0]$. It is the unique
$p\in\{-L+1,\ldots,0\}$ such that $x_{[p,p+L)}\in C$.
All these tests are determined by $x_{[-L+1,L)}$, proving local
constancy. Applying the left shift moves every cut from $c$ to $c-1$,
which gives the congruence.
\end{proof}

For a construction sequence $\W=(W^n)_{n\ge n_0}$ we write
$k_n(x)=k_{W^n}(x)$ for every $x\in\Seg(W^n)$, extending the notation
from $X_{\W}$ to this larger domain. Its $W^n$-cuts are precisely
$k_n(x)+L_n\Z$. Every $W^{n+1}$-cut of
$x\in\Seg(W^{n+1})$ is a $W^n$-cut: concatenate the
$W^n$-decompositions of all its $W^{n+1}$-blocks. In particular,
\[
 \Seg(W^{n+1})\subseteq\Seg(W^n),\qquad
 k_{n+1}(x)\equiv k_n(x)\pmod{L_n}
 \quad(x\in\Seg(W^{n+1})).
\]
We next record how this segmentation behaves under reversal.

\begin{lem}[Reversal of construction sequences]
\label{lem:reversal-construction}
If $\W=(W^n)_{n\ge n_0}$ is a construction sequence over $\bA$, then
$\rev{\W}:=(\rev{W^n})_{n\ge n_0}$ is a construction sequence
with the same scale, and
\[
 X_{\rev{\W}}=\rev{X_{\W}}.
\]
For $n\ge n_0$ and $x\in\Seg(W^n)$, an integer $c$ is a
$W^n$-cut point of $x$ if and only if $1-c-L_n$ is a
$\rev{W^n}$-cut point of $\rev{x}$.
\end{lem}
\begin{proof}
Reversal preserves word lengths and cardinalities. It reverses the
order of the factors in a concatenation, so the concatenation
condition is preserved as well. For double faithfulness, take
$u,v\in W^n$ and $w\in W^{n+1}$. Since $vu$ occurs in $w$,
the word $\rev{u}\,\rev{v}=\rev{vu}$ occurs in $\rev{w}$.

For nonoverlap, suppose $\rev{u}$ occurs at position $r$ in
$\rev{v}\,\rev{w}=\rev{wv}$, where $u,v,w\in W^n$.
Then $u$ occurs at position $L_n-r$ in $wv$.
Nonoverlap for $W^n$ gives either $r=L_n$ and $u=w$, or
$r=0$ and $u=v$. These are exactly the allowed occurrences for the
reversed code. Thus all four construction-sequence conditions hold.

The reversal of the block $x_{[c,c+L_n)}$ occupies the interval
$[1-c-L_n,1-c)$ in $\rev{x}$. Reversing every block of a segmentation
therefore gives the asserted correspondence of cuts. Applying reversal
again proves the converse and yields
\[
 \Seg(\rev{W^n})=\rev{\Seg(W^n)}.
\]
Reversal is a bijection, so it commutes with intersections of sets.
Apply Lemma~\ref{lem:constr_seq_as_intersection} to obtain
$X_{\rev{\W}}=\rev{X_{\W}}$.
\end{proof}

\begin{lem}\label{lem:canonical-odometer}
Let $\cL$ be the odometer with scale $(L_n)$ and transformation $+1$.
For $n\ge n_0$ define
\[
(\pi_{\W}(x))_n\equiv-k_n(x)\pmod{L_n};
\]
for $n<n_0$ define the coordinate by reduction modulo $L_n$.
Then $\pi_{\W}:X_{\W}\to\cL$ is a factor map.
\end{lem}
\begin{proof}
The refinement observation after Lemma~\ref{lem:unique-readability}
gives $k_{n+1}(x)\equiv k_n(x)\pmod{L_n}$, so the coordinates are
coherent. Local constancy in that lemma makes every coordinate of
$\pi_{\W}$ continuous. Its shift congruence gives
\[
k_n(\sigma x)\equiv k_n(x)-1\pmod{L_n},
\qquad \pi_{\W}(\sigma x)=\pi_{\W}(x)+1.
\]
The image is nonempty, compact, and invariant under $+1$.
Every $+1$-orbit is dense in the odometer, since it visits every residue
in each finite quotient. Therefore the image is all of $\cL$.
\end{proof}

\begin{defn} \label{def:aligned-factor}
Let $\mathcal{V}$ and $\W$ be two construction sequences with the same scale, $\phi \colon X_{\mathcal{V}} \to  X_{\W}$ a factor map, and $x\in X_{\mathcal{V}}$. We define $$\offset(\phi) = \pi_{\mathcal{V}}(x) - \pi_{\W}(\phi(x)) \in \cL.$$
Note that $\offset(\phi)$ is independent of the choice of $x$.
We say that $\phi$ is \emph{aligned} if $\offset(\phi) = 0$.
\end{defn}

To justify the independence claim, let $f(x) := \pi_{\mathcal{V}}(x) - \pi_{\W}(\phi(x))$.	
This is a continuous map $X_{\mathcal{V}} \to \cL$.
Note that $f(\sigma x) = f(x)$, so $f$ is constant on the orbit of $x$.
Since $X_{\mathcal{V}}$ is minimal, this orbit is dense, so $f$ is constant.

\begin{lem}\label{lem:offset-additive}
    Let $\mathcal{U}, \mathcal{V}, \mathcal{W}$ be three construction sequences with the same scale,
    $\phi \colon X_{\mathcal{U}} \to X_{\mathcal{V}}$, 
    $\psi \colon X_{\mathcal{V}} \to X_{\mathcal{W}}$ factor maps.
    Then $\offset(\psi \phi) = \offset(\psi) + \offset(\phi)$.
\end{lem}

\begin{proof}
For $x\in X_{\mathcal U}$,
\begin{align*}
\offset(\psi\phi)
 &=\pi_{\mathcal U}(x)-\pi_{\mathcal W}(\psi\phi x)\\
 &=\bigl(\pi_{\mathcal U}(x)-\pi_{\mathcal V}(\phi x)\bigr)
   +\bigl(\pi_{\mathcal V}(\phi x)-\pi_{\mathcal W}(\psi\phi x)\bigr)\\
 &=\offset(\phi)+\offset(\psi).\qedhere
\end{align*}
\end{proof}

\begin{cor} \label{cor:aligned-almost}
With the notation of Definition \ref{def:aligned-factor}, and any integer $k$, we have
$$
	\offset(\sigma^k \phi) = \offset(\phi) - k.
$$
\end{cor}

\begin{proof}
	Since $\pi_{\W}$ is a factor map, we have $\pi_{\W}(\sigma x)=1+\pi_{\W}(x)$ for $x\in X_{\W}$.
    Thus $\offset(\sigma)=-1$, and Lemma~\ref{lem:offset-additive}
gives the assertion.
\end{proof}

For the rest of this section, let $\W$ and $\mathcal V$ be construction
sequences with the same scale $(L_n)_{n\in\N}$. We study factor maps
$X_{\W}\to X_{\mathcal V}$ and give additional conditions under which
they admit a block representation.

Let $n$ be at least as large as both initial indices, and let
$f:W^n\to V^n$ be a function. Set $f\blockexp{n}=f$. Applying $f$
to each $W^n$-block defines extensions
$f\blockexp{m}:W^m\to(V^n)^{L_m/L_n}$ for $m>n$ and
$f\blockexp{*}:X_{\W}\to\Seg(V^n)$, the latter using the canonical
segmentation. Their images need not equal $V^m$ and $X_{\mathcal V}$,
respectively. The next lemma relates these image equalities to aligned factor maps.
\begin{lem} \label{lem_constr_seq_isom}
	Let $\W$ and $\mathcal V$ be construction sequences with the same
scale. Let $n$ be at least as large as both initial indices, and let
$f:W^n\to V^n$ be a function. Then
  \begin{itemize}
    \item $f \blockexp{m} (W^m) = V^m$ for all $m \ge n$, if and only if 
    $f \blockexp{*}$ is an aligned factor map from $X_{\W}$ to $X_{\mathcal{V}}$, and
    \item if $f \blockexp{m} (W^m) = V^m$ for all $m \ge n$ and $f$ is a bijection, then $f \blockexp{*}$ is an aligned conjugacy between $X_{\W}$ and $X_{\mathcal{V}}$.
  \end{itemize}
\end{lem}
\begin{proof}
Suppose $f\blockexp{m}(W^m)=V^m$ for every $m\ge n$.
Applying $f$ at the canonical $n$-boundaries sends each canonical
$m$-block to a $V^m$-block. Thus the image lies in
$\bigcap_{m\ge n}\Seg(V^m)=X_{\mathcal V}$ and all higher-level
boundaries agree. The map is continuous by finite recognizability
and commutes with the shift. Its image is a nonempty compact invariant
subset of the minimal system $X_{\mathcal V}$, so it is onto and aligned.
Conversely, let $f^*$ be an aligned factor map and $m\ge n$.
Alignment gives $\pi_{\W}(x)=\pi_{\V}(f^*(x))$, hence
$k_m(f^*(x))=k_m(x)$ for every $x\in X_{\W}$ by
Lemma~\ref{lem:canonical-odometer}. Thus each $W^m$-block $w$ of $x$
is sent to a $V^m$-block at the same position. By definition of $f^*$,
this image is $f\blockexp{m}(w)$.
Every $W^m$-word occurs in this way by Lemma~\ref{lem:finite-language} and Lemma~\ref{lem:unique-readability}, giving
$f\blockexp{m}(W^m)\subseteq V^m$. Surjectivity of $f^*$, alignment, and occurrence of every $V^m$-word in $X_{\V}$
give the reverse inclusion. If $f$ is bijective, replacing each
$V^n$-block by its unique preimage gives the inverse map on the image;
recognizability makes this inverse continuous.
\end{proof}

\begin{defn}
	Let $\W$ and $\mathcal{V}$ be two construction sequences with the same scale.
	Suppose $\phi \colon  X_{\W} \to X_{\V}$ is a factor map. We say that $\phi$ is \emph{tame}, if $\phi = \sigma^k \circ f \blockexp{*}$
	for some $n$ at least as large as both initial indices, $k\in\Z$, and $f:W^n\to V^n$.
\end{defn}

When $f^*$ is aligned and $m\ge n$, the block-factor criterion gives
$(f\blockexp{m})^*=f^*$. Among representations by aligned block maps,
the exponent of the shift is unique: it is minus the offset.
If $\sigma^l g^*=\sigma^k f^*$, both block maps are aligned, and $g$
is defined at level $m\ge n$, then comparison of offsets gives $l=k$.
Every $W^m$-word occurs in the system, so $g=f\blockexp{m}$.

The next two lemmas give sufficient conditions for every factor map
to be tame.
\begin{lem}
	\label{lem:odometer}
	Let $\cL$ be the odometer with scale $(L_n)_{n\in\N}$, and $\alpha \in \cL$ be a non-integer element.
	There exist arbitrarily large $n\ge3$ such that
	$$
	\alpha \not \equiv \pm k \pmod{L_n} \quad \textrm{ for } k \in \{0, 1, \dots, L_{n-1} - L_{n-2} \}.
	$$
\end{lem}
\begin{proof}
	Fix $n\ge3$. If $\alpha\not\equiv k\pmod{L_n}$ for every
$|k|<L_{n-1}$, then $n$ satisfies the conclusion. Otherwise, choose
$k$ with $|k|<L_{n-1}$ and $\alpha\equiv k\pmod{L_n}$. Let $j > 0$ be the smallest integer such that $\alpha \not \equiv k \pmod{L_{n+j}}$; it exists because $\alpha$ is not an integer.
	So for some integer $0 < c < L_{n+j} / L_{n+j-1}$ we have
	$$
	\alpha \equiv r:= c L_{n+j-1} + k \pmod{L_{n+j}}.
	$$
	Since $|k|<L_{n-1}\le L_{n+j-2}$ and $1\le c\le L_{n+j}/L_{n+j-1}-1$, we have
	$$
	L_{n+j-1}-L_{n+j-2}<r<L_{n+j}-L_{n+j-1}+L_{n+j-2},
	$$
	so $r\not\equiv \pm k'\pmod{L_{n+j}}$ for every $k'\in\{0,1,\dots,L_{n+j-1}-L_{n+j-2}\}$. Hence $n+j$ satisfies the conditions of the lemma.
\end{proof}

\begin{lem} \label{lem_no_nontame_factors}
	Let $\W$ and $\mathcal{V}$ be two construction sequences with the same scale $(L_n)_{n \in \N}$.
	Suppose there is $N_0\ge2$ at least as large as the initial indices of both sequences such that $3\mid L_n$ and the following conditions hold for $n\ge N_0$:
	\begin{itemize}
	  \item[(d1)] For all $v \neq v' \in V^n$, we have that
	  $$
	  v_{[ L_n/3 ,\, 2L_n/3 )} \neq  v'_{[ L_n/3 ,\, 2L_n/3 )},
	  $$
	  that is, $V^n$ has distinctive cores;

	  \item[(d2)] for all $c\in \N$ with $1 \le c \le L_{n+1}/L_n - 2$, all functions
	  	$\psi_0 : W^n \times W^n \to V^n$, all
	  $w, w' \in W^{n+1}$ and all $v \in V^{n+1}$, there exists $j\in [0, L_{n+1}/L_n - 1]$ such that
	  $$
	  v \angl{j} \neq \psi_0(ww' \angl{j+c}, ww' \angl{j+c+1}).
	  $$
	  Here $v \angl{j}$ denotes the subword $v_{[j L_n , (j+1) L_n)}$, that is, $v \angl{j}$ is the $j$-th $W^n$-block (respectively, $V^n$-block) of a word $v$ over $W^n$ (respectively, $V^n$);
	
	  \item[(d3)] $L_n \ge 6 L_{n-1}$.
	\end{itemize}
Then every factor map $\phi:X_{\W}\to X_{\V}$ has a representation $\phi=\sigma^k f^*$ with $f^*$ aligned, and in particular is tame.
\end{lem}

\begin{proof}
	Let $\phi \colon  X_{\W} \to X_{\V}$ be a factor map. We consider two cases:
	
	\emph{Case 1: $\offset(\phi)$ is an integer.}
	Let $k:=\offset(\phi)$. Then $\sigma^k \circ \phi$ is aligned, by Corollary \ref{cor:aligned-almost}.
	So it is enough to show that if $\phi$ is aligned,
	then $\phi = f \blockexp{*}$ for some $n \ge 1$ and $f\colon  W^n \to V^n$.
	
	Both $X_{\W}$ and $X_{\V}$ are subshifts over some finite alphabets $A, B$ (respectively).
	Thus, the factor map $\phi$ is given by a block code \cite[Theorem 6.2.9]{LindMarcus}: for some $r \ge 0$
	and function $\bar{\phi} \colon A^{2r + 1} \to B$, we have
	\begin{equation} \label{eqn:def_of_block_code}
		\phi(x)_j = \bar{\phi} \left( x_{j-r} \dots x_{j+r} \right) \quad \textrm{ for all } x \in X_{\W} \textrm{ and } j \in \Z.
	\end{equation}
	The function $\bar{\phi}$ is commonly called a \emph{block code}, and $r$ is
	the \emph{radius of the block code} $\bar{\phi}$.

	Choose $n$ large enough so that $L_n > 3r$ and $n \ge N_0$.
    For any $w \in W^n$, by Lemma~\ref{lem:finite-language} we can find $x\in X_{\W}$ such that $x_{[0 , L_n)} = w$; by Lemma~\ref{lem:unique-readability}, $0$ is then a $W^n$-cut point of $x$, so $k_n(x)=0$.
	Since $\phi$ is aligned, $k_n(\phi(x))=0$ as well, that is, $\phi(x)_{[0 , L_n)} \in V^n$.
	The block code determines $\phi(x)_{[r,L_n-r)}$ from $x_{[0,L_n)}$.
Since $r < \frac{1}{3}L_n < \frac{2}{3}L_n < L_n-r$, this determines
the core $\phi(x)_{[L_n/3,2L_n/3)}$, which in turn determines
$\phi(x)_{[0,L_n)}$ by (d1). Hence the following map is independent
of the choice of $x$:
	$$
	f \colon  W^n \ni w \mapsto \phi(x)_{[0 , L_n)} \in V^n.
	$$
	The same argument at every $W^n$-cut $p$ gives
$\phi(x)_{[p,p+L_n)}=f(x_{[p,p+L_n)})$. Thus $\phi=f\blockexp{*}$.
	
	\emph{Case 2: $\offset(\phi)$ is not an integer.} We will show that there are no such factor maps.
	Again, $\phi$ can be represented by a block code of some radius $r$.
	Pick $x \in X_{\W}$ such that $\pi_{\W}(x) = 0$.
	Set $\alpha:=\offset(\phi)=-\pi_{\V}(\phi(x))$, which is not an integer.
	
	For $n\in\N$ let $m_n\in\{0,\dots,L_{n+1}-1\}$ denote the residue of $\alpha$ modulo $L_{n+1}$.
	Applying Lemma \ref{lem:odometer} at the index $n+1$, we can pick $n \ge N_0$ such that
	\begin{itemize}
		\item[(a)] $L_n - L_{n-1} < m_n < L_{n+1} - L_n + L_{n-1}$, and
		\item[(b)] $r \le \frac{1}{6} L_n$.
	\end{itemize}
	Fix such an $n$ and write $m:=m_n$.
	Let $w = x_{[0 , L_{n+1})}$ and $w' = x_{[L_{n+1} , 2L_{n+1})}$; these are words in $W^{n+1}$, because $\pi_{\W}(x)=0$ means that $0$ is a $W^{n+1}$-cut point of $x$.
	By definition of $\pi_{\V}$, we have $k_{n+1}(\phi(x))\equiv m\pmod{L_{n+1}}$, so the word $v := \phi(x)_{[m , m + L_{n+1})}$ is in $V^{n+1}$.
	View $ww'$ and $v$ as sequences of length-$L_n$ blocks, with
$v\angl{i}:=v_{[iL_n,(i+1)L_n)}$.
	Consider the interval
	$$
		\Big[ m + \frac{1}{3} L_n - r, m + \frac{2}{3} L_n + r \Big).
	$$
	Combining (a), (b) and (d3) we estimate its endpoints and width:
	\begin{align*}
		m + \frac{1}{3} L_n - r & \ge L_n, \\
		m + \frac{2}{3} L_n + r & \le L_{n+1}, \\
		\frac{1}{3} L_n + 2r & \le \frac{2}{3} L_n.
	\end{align*}

    Let $L:=L_{n+1}/L_n$ and set
    \[
      I_0=[m+L_n/3-r,m+2L_n/3+r),\qquad
      a=\left\lfloor\frac{m+L_n/3-r}{L_n}\right\rfloor.
    \]
    The preceding inequalities give $I_0\subseteq[L_n,L_{n+1})$,
    $|I_0|<L_n$, and $1\le a\le L-1$. Note that $L\ge |W^n|^2+1\ge5$, since every $W^{n+1}$-word contains all $|W^n|^2$ pairs of $W^n$-words as subwords at distinct cut points. Put $c=\min\{a,L-2\}$.
    If $a\le L-2$, the interval starts in the $a$th old block and
    has length less than $L_n$, so it lies in $[aL_n,(a+2)L_n)$.
    If $a=L-1$, then $I_0\subseteq[(L-1)L_n,L_{n+1})$, and
    the two blocks starting at $(L-2)L_n$ contain it. Thus in both cases
    \begin{equation}\label{eq:tame-two-block-window}
      1\le c\le L-2,\qquad I_0\subseteq[cL_n,(c+2)L_n).
    \end{equation}
	For $0\le i\le L-1$, the pair
$ww'\angl{i+c},ww'\angl{i+c+1}$ determines
$x_{[(i+c)L_n,(i+c+2)L_n)}$. By \eqref{eq:tame-two-block-window},
it therefore determines
	$$
		x_{[iL_n + m + L_n/3 - r ,\;  iL_n + m + 2L_n/3 + r )},
	$$
	and the radius-$r$ block code then determines
	$$
		v \angl{i}_{[L_n/3 , 2L_n/3)} =
		\phi(x)_{[iL_n + m + L_n/3 ,\; iL_n + m + 2L_n/3 )}.
	$$
	Condition (d1) now determines $v\angl{i}$; see Figure~\ref{fig:tame-core}.
    The position of the required input window relative to the beginning of
    the two source blocks is independent of $i$. Hence repeated input
    pairs determine the same target core and, by (d1), the same target
    word. Define $\psi_0$ on the pairs occurring in these tests by that
    word, and assign an arbitrary fixed element of $V^n$ to every other
    pair. This defines a function $\psi_0:W^n\times W^n\to V^n$ satisfying
	$$
	v \angl{i} = \psi_0(ww' \angl{i+c}, ww' \angl{i+c+1}) \quad \textrm{ for all } i=0, \dots, L - 1.
	$$
	This contradicts (d2).
	Thus no factor map has noninteger offset.

\end{proof}
    \begin{figure}[htbp]
    \centering
    \begin{tikzpicture}
    \node (x) at (0,0) {$x = \cdots$};
    \draw (1,0.5) -- (7,0.5);
    \draw (1,-0.2) -- (7,-0.2);

     \node [font=\small] (c) at (2.3,0.9) {$cL_n$};
      \draw (2.3,0.5) -- (2.3,-0.2);
      
\node [font=\small] (c+2) at (5.7,0.9) {$(c+2)L_n$};

 \draw (5.7,0.5) -- (5.7,-0.2);

  \node (w) at (4,0.1) {$ww'$};

\draw (3.1,0.5) -- (3.1,-0.2);

\draw (4.9,0.5) -- (4.9,-0.2);

\draw (3.1,-0.2) -- (3.5,-1);

\draw (4.9,-0.2) -- (4.5,-1);

    \node (phi) at (-0.1,-1.5) {$\phi(x) =$ $\cdots$};
    \draw (1,-1) -- (7,-1);
    \draw (1,-1.7) -- (7,-1.7);

 \node (v) at (4,-1.4)  [text=black] {$v \angl{i}$};

\draw[dotted] (3.5,-1) -- (3.5,-1.7);

\draw[dotted] (4.5,-1) -- (4.5,-1.7);

\draw (3,-1) -- (3,-1.7);

\draw (5,-1) -- (5,-1.7);

\fill[fill=lightgray, opacity=0.15] (3,-1) -- (3,-1.7) -- (5,-1.7) -- (5,-1) -- cycle;
    
\end{tikzpicture}
        \caption{
            The core determines $v\angl{i}$ (light gray). The block code
determines this core from an interval of $ww'$ contained in
$[(c+i)L_n,(c+i+2)L_n)$.
        }
        \label{fig:tame-core}
    \end{figure}
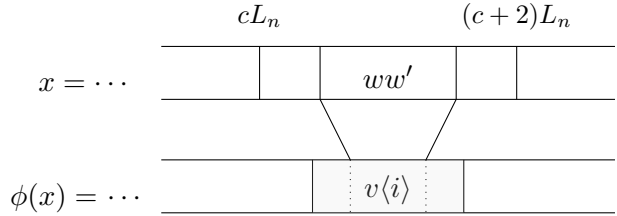

\subsection{Two probabilistic examples}
\label{subsec:thesis-examples}

The following examples, from \cite[Chapter~5, Examples~1 and~2]{DekaThesis},
separate the basic concatenation properties from the restrictions on
factor maps. We use the present indexing and include the repeated-label
estimates omitted in the thesis.
Neither example uses the tree groups. Throughout this subsection, $S\ge2$
is an integer and $\bA$ is a finite alphabet with at least two symbols.

We first specify a common initial code. Choose distinct symbols of $\bA$
and denote them by $0,1$. Take $r\ge1$ with $2^r\ge S$, and choose $S$
distinct binary words $b_0b_1\cdots b_{r-1}$. For each selected word use the codeword
\[
  00\,1\,b_0 1\,b_1 1\cdots b_{r-1} 1\,1^d,
\]
where $d\in\{0,\ldots,5\}$ is chosen so that $L_1=2r+3+d$ is divisible
by six. Each codeword has $00$ only at its beginning and ends in $1$.
Thus, in a concatenation of two codewords, $00$ occurs only at the two
block beginnings. An occurrence of a codeword must start at one of those
beginnings, so the code is nonoverlapping. Denote this $S$-element code
by $C$. Distinctive cores at the initial stage will not be needed.

\begin{exam}[A fixed number of words at every stage]
\label{ex:thesis-fixed-cardinality}
There is a construction sequence $\W=(W^n)_{n\ge1}$ over $\bA$ such that
$|W^n|=S$ for every $n$. Its successive length ratios can be chosen to be
multiples of six.
\end{exam}
\begin{proof}
Set $W^1=C$. Suppose $W^n$ has been constructed, and write $A=W^n$.
For an integer $l\ge6$ divisible by six, independently choose
\[
  U^a_i\in A,\qquad 1\le a\le S,\quad 0\le i<l,
\]
uniformly on $A$. Let $w_a=U^a_0\cdots U^a_{l-1}$, with concatenation
interpreted over $\bA$, and set $L_{n+1}=lL_n$. We seek a choice for which
$W^{n+1}=\{w_1,\ldots,w_S\}$ is nonoverlapping, has $S$ elements, and
contains every ordered pair of $A$-words in every $w_a$.

For a fixed ordered pair $(u,v)\in A^2$ and a fixed label $a$, consider
the disjoint block pairs $(U^a_{2k},U^a_{2k+1})$, where
$0\le k<\lfloor l/2\rfloor$. Then
\[
  \prob\bigl(uv\text{ is absent from }w_a\bigr)
 \le (1-S^{-2})^{\lfloor l/2\rfloor}.
\]
Taking the union over $(u,v,a)$ bounds failure of double faithfulness by
$S^3(1-S^{-2})^{\lfloor l/2\rfloor}$. For $a\ne b$, the probability
that $w_a=w_b$ is $S^{-l}$, so a collision of labels has probability
at most $\binom{S}{2}S^{-l}$.

For nonoverlap, fix labels $a,b,c$ and an internal occurrence of $w_a$
in $w_bw_c$. Since $A$ is nonoverlapping, its starting position must
be $hL_n$ for some $1\le h<l$. If $h\le l/2$, this occurrence implies
\[
 U^a_i=U^b_{i+h}\qquad(0\le i<l-h).
\]
When $a\ne b$, these are independent equalities between disjoint pairs
of uniform variables, so their probability is $S^{-(l-h)}$. When
$a=b$, the equalities join indices $i$ and $i+h$. The resulting graph
is a disjoint union of chains, with $l-h$ edges. Revealing the variables
along each chain gives one factor $S^{-1}$ for each edge, and hence
the same probability $S^{-(l-h)}$.

If $h>l/2$, use instead
\[
 U^a_{l-h+i}=U^c_i\qquad(0\le i<h).
\]
For $a\ne c$ the pairs of variables are disjoint. For $a=c$ their
equality graph is a disjoint union of chains with step $l-h$ and
$h$ edges. In either case the probability is $S^{-h}$. Thus, for
all choices of labels, including repeated labels, the probability of
the internal occurrence is at most $S^{-l/2}$. There are at most
$S^3(l-1)$ choices of $(a,b,c,h)$.

Combining the three estimates, set
\begin{equation}\label{eq:thesis-basic-bad-bound}
 B_S(l):=S^3(1-S^{-2})^{\lfloor l/2\rfloor}
       +S^3(l-1)S^{-l/2}+\binom{S}{2}S^{-l}.
\end{equation}
The probability that at least one required condition fails is at most
$B_S(l)$, which tends to zero as $l\to\infty$. Choose a multiple of six
for which $B_S(l)<1$, and a successful table. This completes the
induction. Lemmas~\ref{lem:constr_seq_as_intersection}
and~\ref{lem:canonical-odometer} show that the resulting subshift is
minimal and has the odometer of the chosen unbounded scale as a factor;
in particular it is infinite.
\end{proof}

\begin{exam}[Two sequences satisfying the tameness hypotheses]
\label{ex:thesis-tame-pair}
There are construction sequences $\W=(W^n)_{n\ge1}$ and
$\mathcal V=(V^n)_{n\ge1}$ over $\bA$, with a common scale and
$|W^n|=|V^n|=S$, that satisfy conditions \textup{(d1)--(d3)} of
Lemma~\ref{lem_no_nontame_factors} for $n\ge2$. Thus every factor map
$X_{\W}\to X_{\mathcal V}$, if one exists, has the aligned tame
representation given by that lemma.
\end{exam}
\begin{proof}
We may take $W^1=V^1=C$. At the induction step, fix the already chosen
$W^n,V^n$ and set $L_{n+1}=lL_n$, where $l\ge6$ is divisible by six.
Independently construct $S$ labeled words over $W^n$ and $S$ labeled
words over $V^n$, exactly as in Example~\ref{ex:thesis-fixed-cardinality}.
All variables in the two families are independent. By
\eqref{eq:thesis-basic-bad-bound}, the probability that either family
fails the three construction requirements is at most $2B_S(l)$.

For two distinct target labels, equality on their middle thirds has
probability $S^{-l/3}$: the middle third consists of exactly $l/3$
old blocks, chosen independently in the two words. Failure of
\textup{(d1)} at stage $n+1$ therefore has probability at most
$\binom{S}{2}S^{-l/3}$.

For \textup{(d2)}, fix $1\le h\le l-2$, a function
$\psi_0:(W^n)^2\to V^n$, two source labels $a,b$, and a target label $c$.
Write $w_a,w_b$ for the source words and $v_c$ for the target word.
Condition on the entire source table. Then all the blocks
\[
 \psi_0(w_aw_b\angl{i+h},w_aw_b\angl{i+h+1}),
       \qquad 0\le i<l,
\]
are fixed elements of $V^n$. All displayed input indices are valid,
since $i+h+1\le2l-2$. The $l$ blocks of $v_c$ remain independent and
uniform on $V^n$. The conditional, and hence the unconditional,
probability that all $l$ equalities with $v_c\angl{i}$ hold is exactly
$S^{-l}$. This argument also applies when $a=b$.

There are $S^{S^2}$ functions $\psi_0$, $S^3$ choices of labels, and
$l-2$ displacements. Thus failure of \textup{(d2)} at this transition
has probability at most
\begin{equation}\label{eq:thesis-pair-test-bound}
 (l-2)S^{S^2+3-l}.
\end{equation}
The simultaneous failure probability is at most
\[
 2B_S(l)+\binom{S}{2}S^{-l/3}+(l-2)S^{S^2+3-l},
\]
which tends to zero. Choose a successful table for a sufficiently large
multiple $l$ of six and continue. The construction gives
\textup{(d1)} at every stage $n\ge2$, \textup{(d2)} at every transition,
and \textup{(d3)} from $n=2$ onward. All hypotheses of
Lemma~\ref{lem_no_nontame_factors} therefore hold with $N_0=2$.
\end{proof}

Both examples use only finite choices at each stage. Fixing an order on
$\bA$, taking the least admissible $l$ for which the displayed bound is
less than one, and then the first successful table in lexicographic
order makes the choices deterministic. These examples establish consistency of the stated conditions.
The second does not assert that a factor map exists. The next section imposes the additional group-action
and map-exclusion conditions needed for the reduction.

\section{Inverse systems from construction sequences} \label{section:final}
Use the enumeration $(w^{(i)})_{i\ge0}$ fixed in
Section~\ref{section:trees}, and write $\xi(w^{(i)})=i$.
For $T\in\Trees$, enumerate its vertices as $(t_T(i))_{i\ge0}$
in increasing $\xi$-order. Then $t_T(0)=\emptyword$, every proper
prefix precedes its extensions, and there are infinitely many
vertices. For fixed $T,N$, the condition that $T'$ agrees with $T$
on all $w^{(i)}$ with $i\le\xi(t_T(N))$ defines a clopen neighborhood
on which the first $N$ nonroot vertices agree.

For the rest of this section, suppress the dependence on $T$ and write
$G_j=\treegroup{j}{T}$, $\rho_j=\treehom{j}{T}$, and $t(i)=t_T(i)$.
Define 
$$
    G^n_j := \gengroup{ \gbar{\treevertex(i)} \colon i \le n \textrm{ and } \depth{\treevertex(i)} = j}.
$$
These finite subgroups satisfy $G^n_j\subseteq G^{n+1}_j$.
Since parents precede their children in the enumeration,
$\rho_j(G^n_{j+1})\subseteq G^n_j$.
We also define  $s(n) := \max_{1 \le i \le n} \depth{\treevertex(i)}$ 
and $M(s)$ as the least $i$ such that $\depth{\treevertex(i)} = s$.

Recall the truncation notation $[\,\cdot\,]_t$ of Section~\ref{sec:constr-seq}: it cuts every column of a symbol, word, point, or shift space over $\{0,1\}^s$ (or over $\C$) to its top $t$ rows, $1\le t\le s$.
The next two lemmas construct a blended inverse system of subshifts adjacent to $T$.
 
\begin{lem}\label{lem:word-construction}
Let $T\in\Trees$. There are a scale $(L_n)_{n\ge1}$, finite word sets
$W^n\subseteq(\{0,1\}^{s(n)})^{L_n}$, and actions $a_j^n$ of $G_j^n$
on $[W^n]_j$, for $1\le j\le s(n)$, with the following properties.
We may take $L_1=6$ and $L_{n+1}/L_n$ to be a multiple of three at least six.
\begin{enumerate}[label=(\roman*)]
\item\label{property-1} For each $j\ge1$, the sequence
$\W_j=([W^n]_j)_{n\ge M(j)}$ is a construction sequence.
\item\label{property-2} Each action $a_j^n$ is free.
\item\label{property-3} For $1\le j< s(n)$, $g\in G_{j+1}^n$, and
$v\in[W^n]_{j+1}$, we have $[gv]_j=\rho_j(g)[v]_j$.
\item\label{property-4} For $j\le s(n)$, $g\in G_j^n$, and
$v\in[W^{n+1}]_j$,
\[
a_j^{n+1}(g,v)=a_j^n(g,\cdot)\blockexp{n+1}(v).
\]
\end{enumerate}
For each $n\ge1$ put $L=L_{n+1}/L_n$. Block indices below range over
integers and refer to blocks of length $L_n$.
\begin{itemize}
\item[(D1)] If $1\le j\le s(n)$ and $v\ne v'$ belong to $[W^n]_j$,
then $v_{[L_n/3,2L_n/3)}\ne v'_{[L_n/3,2L_n/3)}$.
\item[(D2)] Let $1\le j\le j'\le s(n)$, $1\le c\le L-2$,
$\psi_0:([W^n]_{j'})^2\to[W^n]_j$, $w,w'\in[W^{n+1}]_{j'}$, and
$u\in[W^{n+1}]_j$. There is $i\in\{0,\ldots,L-1\}$ such that
\[
u\angl{i}\ne\psi_0(ww'\angl{i+c},ww'\angl{i+c+1}).
\]
\item[(D3)] For $n>1$, $L_n\ge6L_{n-1}$.
\item[(D4)] Let $1\le j\le j'\le s(n)$ and
$f:[W^n]_{j'}\to[W^n]_j$. If
$f\blockexp{n+1}([W^{n+1}]_{j'})\subseteq[W^{n+1}]_j$, then there
is a (necessarily unique, by \ref{property-2}) $g\in G_j^n$ such that $f(v)=g[v]_j$ for every $v\in[W^n]_{j'}$.
\item[(D2')] Let $1\le j\le s(n)$, $1\le c\le L-2$,
$\psi_0:([W^n]_j)^2\to\rev{[W^n]_1}$,
$w,w'\in[W^{n+1}]_j$, and $v\in\rev{[W^{n+1}]_1}$.
There is $i\in\{0,\ldots,L-1\}$ such that
\[
v\angl{i}\ne\psi_0(ww'\angl{i+c},ww'\angl{i+c+1}).
\]
\item[(D4')] For $1\le j\le s(n)$ there is no
$f:[W^n]_j\to\rev{[W^n]_1}$ satisfying
$f\blockexp{n+1}([W^{n+1}]_j)\subseteq\rev{[W^{n+1}]_1}$.
\end{itemize}
The choices can be made so that $L_n,W^n,a_j^n$ depend only on
$t(1),\ldots,t(n)$.
\end{lem}

\begin{proof}
In this construction and in Lemma~\ref{lem:getting-inv-system}, we
write the finite groups $G_j^n$ multiplicatively. Thus $gh$ and $g^{-1}$
mean $g+h$ and $g$ in the additive notation of
Section~\ref{subsec:tree-groups}, and $e$ denotes the identity.
In addition to the stated properties, we maintain
\begin{itemize}
\item[(E1)] For $0\le j<s(n)$ every word in $[W^n]_j$ has exactly
$F_j^n>1$ extensions in $[W^n]_{j+1}$.
\end{itemize}
Here $[W^n]_0$ consists of a single placeholder $\star$.

\paragraph{Initial stage.}
The first nonroot vertex has depth one. Set
\[
L_1=6,\qquad W^1=\{001011,001111\},
\]
and let the nonidentity element of $G_1^1$ exchange these words.
This action is free, and $F_0^1=2$. The two middle thirds are $10$ and
$11$. In either word, $00$ occurs only at the beginning, and both words
end in $1$. In a concatenation of two such words, $00$ occurs only at
the two block beginnings. Thus a codeword can start only at a block
boundary, proving nonoverlap. The requirements involving stage two
will be imposed at the next step.

\paragraph{The random extension table.}
Suppose stage $n$ has been constructed. Put $d=\depth{t(n+1)}$ and
$S=s(n+1)$. For $t\le s(n)$ let $A_t=[W^n]_t$ and put $A_0=\{\star\}$.
If $S=s(n)+1$, introduce an auxiliary alphabet $A_S$ by adjoining to
each word of $A_{S-1}$ a constant last row, either zero or one.
The auxiliary group $G_S^n$ is trivial, and each word of $A_{S-1}$ has
two extensions. This auxiliary alphabet is used only in the present
step; it does not change the already chosen word set $W^n$.
Each $A_t$, $t\ge1$, is nonoverlapping: at the auxiliary level this
follows by projection to $A_{S-1}$. Write $F_{t-1}>1$ for the common
number of extensions from $A_{t-1}$ to $A_t$, and enumerate each fiber
in lexicographic order as $\ext(v,r)$, $0\le r<F_{t-1}$.

Choose a multiple $L$ of three with $L\ge6$, to be made sufficiently
large below, and set $L_{n+1}=LL_n$. Let $E_t=2$ if $t=d$ and $E_t=1$
otherwise. A level-$t$ label is
$\lambda=(e_1,g_1,\ldots,e_t,g_t)$ with
$0\le e_r<E_r$ and $g_r\in G_r^n$.
For each prefix $(e_1,g_1,\ldots,e_{t-1},g_{t-1},e_t)$ and
$i\in\{0,\ldots,L-1\}$ choose independent uniform variables
\[
X(e_1,g_1,\ldots,e_{t-1},g_{t-1},e_t;i)
       \in\{0,\ldots,F_{t-1}-1\}.
\]
Set $w(\emptyword)=\star^L$. If
$\alpha=(e_1,g_1,\ldots,e_{t-1},g_{t-1})$, define
\begin{equation}\label{eqn:worddef}
w(\alpha,e_t,g_t)\angl{i}
 =g_t\ext\bigl(w(\alpha)\angl{i},X(\alpha,e_t;i)\bigr).
\end{equation}
Let $B_t$ be the set of words with level-$t$ labels, and put
$W^{n+1}=B_S$.

For $g\in G_t^n$, let $\widetilde g$ act diagonally on $L$
consecutive $A_t$-blocks. The defining formulas give the following
identities without assuming injectivity of labels:
\begin{equation}\label{eqn:gw-multiplication}
\widetilde g\,w(\alpha,e_t,g_t)=w(\alpha,e_t,gg_t).
\end{equation}
For $t\ge2$, equivariance at stage $n$ gives
\begin{equation}\label{eqn:w-cutoff}
[w(e_1,g_1,\ldots,e_t,g_t)]_{t-1}
 =w(e_1,g_1,\ldots,e_{t-1},\rho_{t-1}(g_t)g_{t-1}).
\end{equation}
The level-zero projection is the placeholder. Iterating this identity,
\begin{equation}\label{eqn:w-cutoff-multiple}
[w(e_1,g_1,\ldots,e_t,g_t)]_r
 =\widetilde h\,w(e_1,g_1,\ldots,e_r,g_r),\qquad r<t,
\end{equation}
where $h=\prod_{k=r+1}^t\rho_{k,r}(g_k)$ and
$\rho_{k,r}=\rho_r\circ\cdots\circ\rho_{k-1}$.
Every lower-level label is realized by projecting a higher-level label
with subsequent group coordinates equal to the identity. Hence
$[W^{n+1}]_t=B_t$.

We also impose the following finite condition:
\begin{itemize}
\item[(E2)] Distinct labels at any level $1\le t\le S$ give distinct
words, even after restriction to the middle third.
\end{itemize}
In particular (E2) implies (D1) at stage $n+1$.

\paragraph{Probability estimates.}
All probability estimates use the independent table before any
success condition is imposed. For each fixed label at level $t$, its blocks
$w\angl{i}$ are independent and uniform on $A_t$; call this (R1).
Indeed, induction on $t$ first gives a uniform parent block, a uniform
choice among its equally numerous extensions then gives a uniform
$A_t$-block, and $g_t$ permutes $A_t$. Different block indices use
disjoint table variables.

We record two estimates used below. In both, $I$ is a finite subset
of $\Z$, all displayed indices lie in $\{0,\ldots,L-1\}$, and
$\epsilon_r\in\{-1,1\}$, $c_r\in\Z$ for $r=1,2,3$.

\begin{claim}\label{claim:R2}
Let $1\le t\le t'\le S$, let $w$ have a fixed level-$t'$ label, and
let $\psi:A_{t'}^2\to A_t$. If neither of
$(\epsilon_1,c_1)$ and $(\epsilon_2,c_2)$ equals
$(\epsilon_3,c_3)$, then
\[
\prob\bigl(\psi(w\angl{\epsilon_1i+c_1},w\angl{\epsilon_2i+c_2})
 =[w]_t\angl{\epsilon_3i+c_3}\text{ for all }i\in I\bigr)
 \le |A_t|^{-(|I|-2)/9}.
\]
\end{claim}
\begin{proof}[Proof of Claim~\ref{claim:R2}]
Write $q_r(i)=\epsilon_ri+c_r$ for $r=1,2,3$.
Each equation $q_3(i)=q_r(i)$, $r=1,2$, has at most one solution:
the two affine index maps are different by assumption. Remove these
at most two solutions from $I$, and call the remaining set $J$.
For $i\in J$ put $D_i=\{q_1(i),q_2(i),q_3(i)\}$.
For a fixed $i$, if $D_i\cap D_j\ne\varnothing$, then
$q_s(j)=q_r(i)$ for some $r,s\in\{1,2,3\}$. Each of the nine
choices determines at most one $j$. Thus $D_i$ meets at most nine
of these indexed sets, including itself. Greedily choosing a set
and deleting all sets that meet it produces $K\subseteq J$ with
\[
 |K|\ge |J|/9\ge (|I|-2)/9,
 \qquad D_i\cap D_j=\varnothing\quad(i\ne j\text{ in }K).
\]
For each $i\in K$, the target index differs from both input indices. By (R1), the target block is independent of the input blocks.
Its projection to $A_t$ is uniform, since iterating (E1) gives the
same number of preimages for every word of $A_t$. Conditional on the
input blocks, precisely one projected target value satisfies the
equality. Its probability is therefore $|A_t|^{-1}$.
The tests indexed by $K$ use disjoint table-index families, so they
are independent. If $E$ denotes the event of the claim, we obtain
\[
 \prob(E)\le\prod_{i\in K}|A_t|^{-1}
 =|A_t|^{-|K|}\le |A_t|^{-(|I|-2)/9}.
\]
The two input maps are allowed to agree with each other: only their
separation from the target map is needed.
\end{proof}

\begin{claim}\label{claim:R3}
Let $1\le t\le t'\le S$. Let $w$ and $u$ have fixed labels at
levels $t'$ and $t$, respectively, let $\psi:A_{t'}^2\to A_t$, and
consider the event
\begin{equation}\label{eqn:r2}
\psi(w\angl{\epsilon_1i+c_1},w\angl{\epsilon_2i+c_2})
 =u\angl{\epsilon_3i+c_3}\quad(i\in I).
\end{equation}
Either its probability is at most $2^{-(|I|-2)/9}$, or the labels of
$u$ and $w$ agree through $e_t$, the identity
$u=\widetilde h[w]_t$ holds for some fixed $h\in G_t^n$, and one of
the two input index maps equals the target index map.
Here agreement through $e_t$ means agreement of
$(e_1,g_1,\ldots,g_{t-1},e_t)$.
\end{claim}
\begin{proof}[Proof of Claim~\ref{claim:R3}]
If these prefixes agree, equation~\eqref{eqn:w-cutoff-multiple} gives
$u=\widetilde h[w]_t$. Replacing $\psi$ by $h^{-1}\psi$ turns
\eqref{eqn:r2} into the event of Claim~\ref{claim:R2}. Unless an input
index map equals the target index map, Claim~\ref{claim:R2} bounds its
probability by $|A_t|^{-(|I|-2)/9}\le2^{-(|I|-2)/9}$ when $|I|\ge2$; for
$|I|<2$ the required bound is automatic.

Suppose the prefixes differ, and write
\[
 \kappa'=(e'_1,g'_1,\ldots,g'_{t-1},e'_t),\qquad
 v=w(e'_1,g'_1,\ldots,e'_{t-1},g'_{t-1}).
\]
For $t=1$ the parent $v$ is the placeholder word. By
\eqref{eqn:worddef},
\begin{equation}\label{eq:thesis-fresh-target}
 u\angl{k}=g'_t\ext\bigl(v\angl{k},X(\kappa';k)\bigr)
 \qquad(0\le k<L).
\end{equation}
The variables $X(\kappa';k)$ occur in neither $v$ nor $w$. Indeed,
$v$ uses only lower-level prefixes, and the level-$t$ prefix used
in constructing $w$ differs from $\kappa'$.
Let $\mathscr F$ be the sigma-algebra generated by all table variables
except $X(\kappa';k)$, $0\le k<L$. Both $v$ and $w$ are
$\mathscr F$-measurable. Conditional on $\mathscr F$, the excluded
variables remain independent and uniform on a set of size $F_{t-1}$.

For each $i\in I$, the left side of \eqref{eqn:r2} is now fixed.
Equation~\eqref{eq:thesis-fresh-target} shows that at most one value of
$X(\kappa';\epsilon_3i+c_3)$ can make the right side equal to it:
$\ext(v\angl{k},\cdot)$ is an injection and $g'_t$ is a bijection.
The indices $\epsilon_3i+c_3$ are distinct for distinct $i$. Thus, if
$E_i$ denotes the equality indexed by $i$ and $E=\bigcap_{i\in I}E_i$,
\begin{equation}\label{eq:thesis-conditional-product}
 \prob(E\mid\mathscr F)=\prod_{i\in I}\prob(E_i\mid\mathscr F)
       \le F_{t-1}^{-|I|}\le2^{-|I|}.
\end{equation}
Averaging this inequality over the complementary table gives
$\prob(E)\le2^{-|I|}$, which implies the asserted bound. Since the
table is finite, these conditional statements hold on
its positive-probability atoms. No independence is asserted after
conditioning on a good event.
\end{proof}

We now impose the finite good events. For any pair of $A_t$-blocks
and any labeled word $w\in B_t$, the probability that the pair never
occurs in consecutive blocks of $w$ is at most
\begin{equation}\label{eqn:s1proof}
\prod_{i=0}^{\lfloor L/2\rfloor-1}(1-|A_t|^{-2})
 =(1-|A_t|^{-2})^{\lfloor L/2\rfloor}.
\end{equation}
This imposes double faithfulness, and in particular ensures that each
$A_t$-block occurs in every word of $B_t$.

For nonoverlap of $B_t$, an internal occurrence of a word $u\in B_t$
in $vw$, with $v,w\in B_t$, must start at $cL_n$, $0<c<L$, because
$A_t$ is nonoverlapping. If $c\le L/2$, the equality
$u\angl{i}=v\angl{i+c}$ holds for $0\le i<L-c$.
Apply Claim~\ref{claim:R3} with the projection onto one input, input offset $c$,
and target offset zero. The exceptional alternative is impossible.
If $c>L/2$, use $u\angl{L-c+i}=w\angl{i}$ for $0\le i<c$.
In either case the probability is at most
\begin{equation}\label{eqn:s15proof}
2^{-(L/2-2)/9}.
\end{equation}

For (E2), fix distinct labels at the same level $t$ and compare their
blocks for $L/3\le i<2L/3$. If their prefixes through $e_t$ differ,
Claim~\ref{claim:R3} bounds equality by $2^{-(L/3-2)/9}$. If the prefixes agree, then $t\le s(n)$ (at the auxiliary level $t=S=s(n)+1$ the group $G^n_S$ is trivial, so labels with equal prefixes through $e_S$ coincide), and
the words differ by the diagonal action of
$h=g_t'g_t^{-1}\ne e$. Equality at even one block contradicts
freeness of the old action on $A_t$. This uses only the formal label identities, not the label injectivity
we are proving.

For (D4), fix $f:A_{j'}\to A_j$, with $j\le j'\le s(n)$, and labels
of $w\in B_{j'}$ and $u\in B_j$. The equation
$f(w\angl{i})=u\angl{i}$, $0\le i<L$, is an instance of Claim~\ref{claim:R3} with
all signs positive and all offsets zero. In its exceptional case,
$u=\widetilde h[w]_j$. On the double-faithfulness event every
$A_{j'}$-block occurs in $w$, so $f(v)=h[v]_j$ on its entire domain.
In all other cases the probability is at most $2^{-(L-2)/9}$.
Avoiding these latter events, together with failure of double
faithfulness, therefore guarantees (D4).

For (D4'), an equality $f\blockexp{n+1}(w)=\rev{u}$ becomes
\[
\rev{f(w\angl{i})}=u\angl{L-1-i}\quad(0\le i<L).
\]
Use Claim~\ref{claim:R3} with input index maps $i,i$, target index map $L-1-i$, and
$\psi(a,b)=\rev{f(a)}$. The exceptional alternative is impossible,
so each such event has probability at most $2^{-(L-2)/9}$.
Avoiding these events prevents any labeled word from being sent
to a reversed target word, and hence implies (D4').

For (D2), failure gives
\begin{equation}\label{eqn:d2proof}
\psi_0(ww'\angl{i+c},ww'\angl{i+c+1})=u\angl{i}
 \quad(0\le i<L).
\end{equation}
If $c\le L/2$, restrict to $0\le i<L-c-1$ and use $w$ alone.
The input offsets are $c,c+1$ and the target offset is zero.
If $c>L/2$, substitute $i=L-c+r$ and use
\[
\psi_0(w'\angl{r},w'\angl{r+1})=u\angl{L-c+r}
 \quad(0\le r<c).
\]
Since $L-c\ge2$, neither input map equals the target map.
Thus Claim~\ref{claim:R3} bounds either event by $2^{-(L/2-3)/9}$.
For (D2'), write $v=\rev{u}$ and reverse the output block of
$\psi_0$. In the first restriction the target index is $L-1-i$;
in the second it is $c-1-r$. Its sign is negative while the two
input signs are positive, so Claim~\ref{claim:R3} gives the same bound.

\paragraph{Simultaneous choice and actions.}
At a fixed stage, the label sets, old alphabets, and sets of functions
between them are finite and independent of $L$. Only the displacement
choices contribute a factor of order $L$. Taking the union
of the bad events just bounded gives constants $C_n<\infty$ and
$c_n>0$, depending only on the finite stage data, such that
\[
\prob(\text{some required good event fails})
 \le C_n(1+L)e^{-c_nL}.
\]
Choose an admissible $L$ for which this is less than one and a table
on which all good events hold. In particular labels are injective, so
$|B_t|=\prod_{r\le t}E_r|G^n_r|\ge2$ for every $1\le t\le S$, which is condition~\ref{def:constr-seq:1} of Definition~\ref{defn:constrSeq} for the new sets.
The old groups act by~\eqref{eqn:gw-multiplication}. Only $G_d^n$
acquires a new basis generator $q=\gbar{t(n+1)}$.
If $d=1$, define
\[
\widetilde q\,w(e_1,g_1)=w(1-e_1,g_1).
\]
If $d>1$, let $p=\gbar{\parent{t(n+1)}}\in G_{d-1}^n$ and define
\begin{equation}\label{eqn:definition-of-new-group-action}
\widetilde q\,w(e_1,g_1,\ldots,e_{d-1},g_{d-1},e_d,g_d)
 =w(e_1,g_1,\ldots,e_{d-1},pg_{d-1},1-e_d,g_d).
\end{equation}
This is an involution commuting with the old action. The abstract
group is $G_d^{n+1}=G_d^n\oplus\langle q\rangle$, so these formulas
define its action. A nonidentity old element changes the last group
label, while an element involving $q$ changes $e_d$; since distinct labels give distinct words by (E2), all new
actions are free. For $q$ with $d>1$, compatibility with the projection, that is, $[\widetilde q\,w]_{d-1}=\widetilde p\,[w]_{d-1}$ for $w=w(e_1,g_1,\ldots,e_d,g_d)$, follows from \eqref{eqn:w-cutoff} and \eqref{eqn:gw-multiplication}: both sides equal
\[
w(e_1,g_1,\ldots,e_{d-1},\,p\rho_{d-1}(g_d)g_{d-1}),
\]
because $\rho_{d-1}(q)=p$ and $G^n_{d-1}$ is abelian; for $d=1$ there is nothing to check.
Compatibility for old generators follows by applying the old
compatibility blockwise. This proves (ii)--(iv).

Finally, a fixed word of $B_{t-1}$ has
\[
F_{t-1}^{n+1}=E_t|G_t^n|
\]
extensions in $B_t$. For $t\ge2$, choose $e_t$ and $h'\in G_t^n$;
the projection equation uniquely determines the preceding group
coordinate $h$ from $\rho_{t-1}(h')h=g_{t-1}$.
For $t=1$ simply count its $E_1|G_1^n|$ labels.
Label injectivity makes these counts exact. At every old depth
$G_t^n$ is nontrivial. At a newly introduced depth $G_t^n$ is trivial
but $E_t=2$. Thus all extension counts exceed one, establishing (E1).
The good events give (i), (D1), (D2), (D4), (D2'), and (D4'); our
choice of $L$ gives (D3).

To make the choices depend only on the processed vertices, order
the finite groups by their coordinates in the ordered tree basis,
the extension fibers lexicographically, and the finite tables
lexicographically. Take the least admissible $L$ for which a successful
table exists and then the first successful table. All success tests
are finite and depend only on the preceding finite data and $t(n+1)$.
The probability bound proves that the search terminates. Induction
now proves the asserted dependence on $t(1),\ldots,t(n)$.
\end{proof}

\begin{lem} \label{lem:getting-inv-system}
Let $T\in\Trees$, and let $L_n,W^n,a_j^n$ be given by
Lemma~\ref{lem:word-construction}. For $j\ge1$, let $X_j$ be the
subshift generated by $\W_j$, and set
$\pi_j\colon X_{j+1}\ni x\mapsto[x]_j\in X_j$.
Then $(X_j,\pi_j)_{j\in\N}$ is a blended inverse system of infinite minimal subshifts adjacent to $T$. 
Moreover, $\rev{X_1}$ is not a factor of $X_j$ for any $j \ge 1$.
\end{lem}

\begin{proof}
The coordinate projections satisfy $[X_{j+1}]_j=X_j$.
Indeed, the finite word sets project onto one another at every
sufficiently large stage; the same holds for their segmented spaces,
and projection commutes with a decreasing intersection of nonempty
compact sets. Thus $\pi_j$ is an onto factor map and preserves the
canonical block boundaries. Each $X_j$ is minimal by
Lemma~\ref{lem:constr_seq_as_intersection} and has an infinite odometer
factor by Lemma~\ref{lem:canonical-odometer}. It is therefore infinite.

For $g\in G_j^n$, write $g_{(n)}\colon[W^n]_j\ni w\mapsto gw\in[W^n]_j$.
We proceed in four steps.

\smallskip\noindent\textbf{Step 1.} \emph{For $j\ge1$, $g\in G_j$, and $n\ge M(j)$ with $g\in G_j^n$,
the map $g_{(n)}\blockexp{*}$ is an automorphism of $X_j$ independent
of $n$. We denote it by $g\blockexp{*}$.}

\noindent\emph{Proof of Step 1.}
        Choose the least such $n$. Iterating~\ref{property-4} gives
$g_{(m)}=g_{(n)}\blockexp{m}$ for every $m>n$, and hence
$g_{(m)}\blockexp{*}=g_{(n)}\blockexp{*}$. Each $g_{(m)}$ is a
bijection of $[W^m]_j$. Lemma~\ref{lem_constr_seq_isom} therefore
makes the common map $g\blockexp{*}$ an automorphism of $X_j$.

\smallskip\noindent\textbf{Step 2.} \emph{For any $j' \ge j \ge 1$, if $\phi \colon X_{j'} \to X_j$ is a factor map,
        then $\phi = \sigma^k \circ g \blockexp{*} \circ \pi_{j', j}$ for some $k \in \Z$, $g \in G_j$.}

\noindent\emph{Proof of Step 2.}
        Apply Lemma~\ref{lem_no_nontame_factors} to $\W=\W_{j'}$ and
$\V=\W_j$ with $N_0:=\max\{2,M(j')\}$. We have $3\mid L_n$ for
every $n$, and (D1)--(D3) of Lemma~\ref{lem:word-construction} give
(d1)--(d3) for $n\ge N_0$. Thus $\phi=\sigma^k f^*$ for some
$n\ge N_0$, $k\in\Z$, and $f:[W^n]_{j'}\to[W^n]_j$, with $f^*$ aligned. 
        Lemma~\ref{lem_constr_seq_isom} gives
$f\blockexp{n+1}([W^{n+1}]_{j'})\subseteq[W^{n+1}]_j$. 
        By (D4), this implies that $f=g_{(n)}\circ[\,\cdot\,]_j$ for some $g \in G^n_j$. 
        So $\phi = \sigma^k \circ f \blockexp{*} = \sigma^k \circ g \blockexp{*} \circ \pi_{j', j}$. 

\smallskip\noindent\textbf{Step 3.} \emph{The map $G_j \ni g \mapsto g \blockexp{*} \gengroup{\sigma} \in \AutQ{X_j}$ is an isomorphism of groups.}

\noindent\emph{Proof of Step 3.}
        For any $g,h\in G_j$, choose $n\ge M(j)$ large enough that $g,h\in G_j^n$.
        Then $g \blockexp{*} h \blockexp{*} = g_{(n)} \blockexp{*} h_{(n)} \blockexp{*} = 
            (gh)_{(n)} \blockexp{*} = (gh) \blockexp{*}$, so the map is a group homomorphism.
            
        For injectivity, suppose $g\blockexp{*}\gengroup{\sigma}=h\blockexp{*}\gengroup{\sigma}$. 
        Then, $(g h^{-1}) \blockexp{*} = \sigma^k$ for some $k$. Comparing offsets, we get
        $ -k = \offset(\sigma^k) = \offset((g h^{-1}) \blockexp{*}) = 0$. 
        So $(g h^{-1}) \blockexp{*}$ is the identity map. 
        Choose $n$ large enough so that $g h^{-1} \in G^n_j$. 
        If $(g h^{-1}) \blockexp{*} = (g h^{-1})_{(n)} \blockexp{*}$ is the identity map,
        then $(g h^{-1})_{(n)}$ is the identity map on $[W^n]_j$, since every word of $[W^n]_j$ occurs in $X_j$ by Lemma~\ref{lem:finite-language}. 
        By \ref{property-2}, we conclude that $g=h$. 

        Surjectivity follows from Step 2 with $j'=j$. 

\smallskip\noindent\textbf{Step 4.} \emph{For all $j \ge 1$, $g \in G_{j+1}$ we have $\pi_j g \blockexp{*} = \rho_j(g) \blockexp{*} \pi_j$.}

\noindent\emph{Proof of Step 4.}
        Viewing both sides as operations on $[W^n]_{j+1}$-blocks, we get
        $$
            \pi_j g \blockexp{*} = \left( [\; \cdot\;]_j \circ g_{(n)} \right) \blockexp{*} 
            = \left( \rho_j(g)_{(n)} \circ [\; \cdot\;]_j \right) \blockexp{*} 
            = \rho_j(g) \blockexp{*} \pi_j,
        $$
        where the second equality comes from \ref{property-3}.

\smallskip
Steps 2 and 3 imply $\Aut(X_j)=\{\sigma^k g^*:k\in\Z,\ g\in G_j\}$. Together with Step 4, this shows that every automorphism $\sigma^kg^*$ of $X_{j+1}$ admits the compatible automorphism $\sigma^k\rho_j(g)^*$ of $X_j$, so each $\pi_j$ is \strongfactor. 
Steps 3 and 4 give adjacency to $T$. The isomorphisms
$\phi_j\colon\AutQ{X_j}\to G_j$ in Definition~\ref{def:adjacent}
are the inverses of the maps in Step 3, and Step 4 gives
$\rho_j\phi_{j+1}=\phi_j\rho'_{\pi_j}$.
Step 2 shows that the constructed inverse system is blended.

It remains to exclude factor maps onto $\rev{X_1}$. 
Suppose $\phi \colon X_j \to \rev{X_1}$ is a factor map.
By Lemma~\ref{lem:reversal-construction}, the reversed word sets for
$X_1$ form a construction sequence generating $\rev{X_1}$.
Reversing a word reverses its middle third within the same interval,
so (D1) supplies distinctive cores for this reversed sequence, which is condition (d1) of Lemma~\ref{lem_no_nontame_factors} for the target $\V=\rev{\W_1}$.
Conditions (d2) and (d3) are (D2') and (D3), so the hypotheses of
Lemma~\ref{lem_no_nontame_factors} hold for $\W=\W_j$, $\V=\rev{\W_1}$ and $N_0=\max\{2,M(j)\}$, and 
we get $\phi=\sigma^k f^*$ with $f^*$ aligned and $f:[W^n]_j\to[\rev{W^n}]_1$ at some sufficiently large $n$. 
Lemma~\ref{lem_constr_seq_isom} gives
$f\blockexp{n+1}([W^{n+1}]_j)\subseteq[\rev{W^{n+1}}]_1$,
contrary to (D4').
\end{proof}

\begin{lem}\label{lem:parameter-continuity}
Use the deterministic choices of Lemma~\ref{lem:word-construction} and
write $X(T)=\varprojlim(X_j(T),\pi_j(T))$.
Embed $\{0,1\}^j$ in $\C$ by appending zeros and view the limits as
subsystems of $\mathscr A=(\C^{\Z})^{\N}$.
Then $T\mapsto X(T)$ is continuous into $\K(\mathscr A)$.
For $R=\Psi(T)$ set
\[
\alpha_j(R)=\gbar{\ctup{j}{j}}\blockexp{*},\qquad
\Xother(R)=\varprojlim(X_j(R),\alpha_j(R)\pi_j(R)),
\]
where $\ctup{j}{j}\in V_j(\Psi(T))$ is the endpoint of the $j$th
trunk. Under the identification of Lemma~\ref{lem:getting-inv-system},
$\alpha_j(R)\gengroup{\sigma}$ corresponds to $\gbar{\ctup{j}{j}}$.
The map $T\mapsto\Xother(\Psi(T))$ is continuous as well.
\end{lem}
\begin{proof}
We apply Lemma~\ref{lem:inv-lim-continuity}. Fix $T\in\Trees$ and
let $T_k\to T$ in $\Trees$, with $T_0=T$. For every fixed $N$,
agreement on all vertices whose enumeration indices are at most
$\xi(t_T(N))$ gives the same first $N$ processed vertices.
The deterministic choices then give the same construction data
through stage $N$. This agreement holds for all sufficiently large $k$.

Fix a row level $j$. Choose $N$ with $s_T(N)\ge j$, where the
subscript records the tree used in the construction. For all large $k$,
\[
 [W^N(T_k)]_j=[W^N(T)]_j.
\]
Corollary~\ref{cor:common-stage-hausdorff} gives
\[
 d_{\Sigma,H}(X_j(T_k),X_j(T))
 \le 2^{1-\lfloor L_N(T)/2\rfloor}.
\]
Since $L_N(T)\to\infty$, the components converge at every fixed level.
For the untwisted systems, let $p_j:\mathscr S\to\mathscr S$ be
row projection to the first $j$ rows, followed by appending zeros.
Then $\pi_j(T_k)=p_j|_{X_{j+1}(T_k)}$ for every $k\in\N_0$.
The common-map hypothesis of Lemma~\ref{lem:inv-lim-continuity}
holds with $E_j=\mathscr S$ and $Q_j=p_j$. The lemma proves
$X(T_k)\to X(T)$.

For the twisted systems put $R_k=\Psi(T_k)$ and $R=\Psi(T)$.
Continuity of $\Psi$ gives $R_k\to R$, and the component-convergence
argument above applies to these trees as well. Fix $j\ge1$ and choose
$N$ with $s_R(N)\ge j+1$ so large that the generator
$g_j=\gbar{\ctup{j}{j}}$ has been processed by stage $N$.
Define the common closed domain
\[
 E_j=\Seg([W^N(R)]_{j+1})\subseteq\mathscr S.
\]
It is compact because its code is finite. Row projection maps $E_j$
into $\Seg([W^N(R)]_j)$. On the latter segmented space, define
$\beta_j$ by applying the permutation $a_j^N(R)(g_j,\cdot)$ to
every block of its unique stage-$N$ segmentation, leaving the cut
positions unchanged. Lemma~\ref{lem:unique-readability} determines
both the cut and the block covering zero from a finite window.
Thus the zero coordinate of $\beta_j(x)$ is locally determined;
the shift congruence in that lemma makes $\beta_j$ shift-commuting.
It follows that $\beta_j$ is continuous on the entire segmented space.
Set
\[
 Q_j=\beta_j\circ p_j|_{E_j}:E_j\to\mathscr S.
\]

For all sufficiently large $k$, the data through stage $N$ agree
for $R_k$ and $R$, and $X_{j+1}(R_k)\subseteq E_j$; these statements
also hold for $k=0$. Property \textup{(iv)} of
Lemma~\ref{lem:word-construction} identifies the restriction of
$\beta_j$ to $X_j(R_k)$ with $\alpha_j(R_k)$. Consequently
\[
 Q_j|_{X_{j+1}(R_k)}=\alpha_j(R_k)\pi_j(R_k)
\]
for $k=0$ and eventually in $k$. These are surjective bonding maps,
since each $\alpha_j$ is an automorphism and each $\pi_j$ is a factor
map. All hypotheses of Lemma~\ref{lem:inv-lim-continuity} are now
verified, so $\Xother(R_k)\to\Xother(R)$. The parameter spaces are
metrizable, and the sequential conclusions prove both continuity claims.
\end{proof}

\begin{prop}\label{prop:limit-automorphisms}
For the untwisted system of Lemma~\ref{lem:getting-inv-system}, there
are group isomorphisms
\[
\Aut(X(T))\cong\Z\times\varprojlim(G_j(T),\rho_j(T)),
\qquad
\Aut'(X(T))\cong\varprojlim(G_j(T),\rho_j(T)).
\]
These are isomorphisms of abstract groups; no group topology is asserted.
\end{prop}
\begin{proof}
The proof of Lemma~\ref{lem:getting-inv-system} gives a unique expression
$\sigma^k g^*$ for every automorphism of $X_j(T)$.
For uniqueness, compare offsets to determine $k$, and then use
freeness at a stage containing the two group elements to determine $g$.
The blended property and Lemma~\ref{lem:inv-lim-isomorphism} express
any automorphism of the inverse limit coordinatewise as
$f_j=\sigma^{k_j}g_j^*$, with $f_j\pi_j=\pi_jf_{j+1}$.
The coordinate projections are surjective, so these $f_j$ are unique.
Using $\pi_jg_{j+1}^*=\rho_j(g_{j+1})^*\pi_j$ and cancelling $\pi_j$
gives
\[
k_j=k_{j+1},\qquad g_j=\rho_j(g_{j+1}).
\]
Conversely, a common $k\in\Z$ and a coherent sequence $(g_j)$ define
compatible coordinate automorphisms, hence an automorphism of the
inverse limit. Composition multiplies the group coordinates and adds
the common shift exponent. The shift subgroup is exactly the first
factor, which gives the second isomorphism.
\end{proof}

\section{Main result and applications}

\subsection{Conjugacy and flip conjugacy}
We now prove the main theorem and its consequences.

\begin{proof}[Proof of Theorem~\ref{thm:main-thm}]
For $T\in\Trees$ apply Lemmas~\ref{lem:word-construction} and
\ref{lem:getting-inv-system} to $R=\Psi(T)$.
The resulting inverse system is blended and adjacent to $R$;
its reduced-group identifications are
$g\mapsto g^*\langle\sigma\rangle$, and no $X_j(R)$ factors onto
$\rev{X_1(R)}$. Define $\alpha_j(R)$ and the two inverse limits
$X(R)$ and $\Xother(R)$ as in Lemma~\ref{lem:parameter-continuity}.
Lemma~\ref{lem:main-reduction} gives
\[
T\in\IllFounded\quad\Longleftrightarrow\quad X(R)\cong\Xother(R),
\]
and, when $T$ is well-founded, the two systems are not flip conjugate.
The components are infinite and minimal, so both inverse limits are
Cantor minimal systems by Lemma~\ref{lem:thesis-inverse-minimal}.

By Lemma~\ref{lem:parameter-continuity}, the maps $T\mapsto X(T)$ and
$T\mapsto\Xother(\Psi(T))$ into the fixed ambient hyperspace
$\K((\C^{\Z})^{\N})$ are continuous; since $\Psi$ is continuous, so is
$T\mapsto X(\Psi(T))$.
The coordinate permutation
\[
 (\C^{\Z})^{\N}\ni(x_j)_{j\in\N}\longmapsto
 \bigl((x_j(n))_{j\in\N}\bigr)_{n\in\Z}\in(\C^{\N})^{\Z}
\]
is a homeomorphism intertwining the coordinatewise shift with the
shift. A fixed homeomorphism $\iota\colon\C^{\N}\to\C$, applied
at each coordinate $n\in\Z$, then conjugates
$((\C^{\N})^{\Z},\sigma)$ to $(\C^{\Z},\sigma)$.
Let $\kappa$ be the composition of these two conjugacies.
It carries $X(R)$ and $\Xother(R)$ into $\Min^\sigma_p(\C^{\Z})$.
By Lemma~\ref{lem:conjugacy-hyperspace}, the maps
$T\mapsto\kappa(X(\Psi(T)))$ and
$T\mapsto\kappa(\Xother(\Psi(T)))$ are continuous.
Apply the objectwise transfer $\Xi$ of
Theorem~\ref{thm:subshifts-and-systems-are-bireducible} to obtain
continuous maps $\Phi_1,\Phi_2$ into $\Min(\C)$ representing these
two systems. These transfers preserve each system up to conjugacy, so the required
conjugacy and non-flip-conjugacy conclusions are unchanged.
\end{proof}

\begin{proof}[Proof of Corollary~\ref{cor:main}]
The preceding map continuously reduces $\IllFounded$ to either
relation, proving analytic hardness. Conjugacy on $\Min(\C)$ is
analytic as established above. Flip conjugacy is the union of that
relation with its inverse-image under the continuous map
$(T,S)\mapsto(T,S^{-1})$, and is analytic as well.
\end{proof}

\begin{cor}\label{cor:transitive}
Conjugacy of transitive Cantor systems is complete analytic.
\end{cor}
\begin{proof}
For a countable clopen basis $(U_i)$ of $\C$, transitivity of a
homeomorphism $T$ is equivalent to $T^n(U_i)\cap U_j\ne\varnothing$
for some $n\ge0$, for every $i,j$. For fixed $i,j,n$, this is an open
condition on $T$; hence the transitive homeomorphisms form a Polish
$G_\delta$ subspace of $\Homeo(\C)$. Conjugacy on it is analytic by
the homeomorphism-witness description. Every minimal system is
transitive, so the continuous reduction of Theorem~\ref{thm:main-thm}
also proves completeness on this larger space.
\end{proof}
Vejnar \cite[Theorem~3.4]{VejnarChaotic} independently proved the
stronger statement that conjugacy of transitive Cantor homeomorphisms
with dense periodic points is Borel bireducible with a universal orbit
equivalence relation of $S_\infty$; his result also implies
Corollary~\ref{cor:transitive}.

\subsection{Automorphism groups}
\begin{proof}[Proof of Theorem~\ref{thm:centralizer}]
By Proposition~\ref{prop:limit-automorphisms} and Lemma~\ref{lem_trivial},
\[
T\in\IllFounded\quad\Longleftrightarrow\quad
\Aut'(X(T))\ne\{e\}.
\]
Lemma~\ref{lem:parameter-continuity} gives continuity of $T\mapsto X(T)$
into the hyperspace of the ambient Cantor shift $((\C^{\Z})^{\N},\sigma)$. Fix a homeomorphism $h\colon(\C^{\Z})^{\N}\to\C$ and set
$F=h\sigma h^{-1}$. Lemma~\ref{lem:conjugacy-hyperspace} makes
$T\mapsto h(X(T))\in\K^F_p(\C)$ continuous.
Lemma~\ref{lem:transport} gives the continuous assignment
\[
T\longmapsto
\varphi_{h(X(T))}^{-1}\circ F|_{h(X(T))}\circ\varphi_{h(X(T))}
 \in\Min(\C).
\]
Each output is conjugate to $X(T)$, so this is a reduction to the set
of systems with an automorphism outside the powers of the dynamics.

For analyticity, that set is the projection of
\[
\{(U,S)\in\Min(\C)\times\Homeo(\C):
 SU=US\text{ and }S\ne U^k\text{ for every }k\in\Z\}.
\]
Commutation and each equality $S=U^k$ are closed conditions by
continuity of the group operations. The witness set is therefore
Borel and its projection is analytic. Together with the reduction,
this proves complete analyticity.
\end{proof}

\subsection{Descriptive combinatorics}\label{sec:graph-application}
A graph on a topological space $V$ is a symmetric, irreflexive relation
$E\subseteq V^2$. A continuous $k$-coloring is a continuous map
$c:V\to\{0,\ldots,k-1\}$, with discrete range, such that
$c(x)\ne c(y)$ whenever $(x,y)\in E$. For the graphs considered here,
$\chi_c(V,E)$ is the least finite $k$ for which such a coloring exists, and is set to $\infty$ if there is no finite continuous coloring.
For graphs $(V,E)$ and $(V',E')$, write
\[
(V,E)\preceq_c^i(V',E')
\]
if there is an injective continuous map $f:V\to V'$ satisfying
$(x,y)\in E\Longrightarrow(f(x),f(y))\in E'$.
Write $(V,E)\approx_c^i(V',E')$ if reductions exist in both directions.
A \emph{topological graph isomorphism} is a homeomorphism $f:V\to V'$
satisfying $(f\times f)(E)=E'$; its relation is denoted by
$\cong_{\mathrm{top}}$.

We use a fixed vertex space $\C$ and code a graph by its edge set in
$\K(\C^2)$. Put
\[
C_{[2,3]}=\{E\in\K(\C^2):E=E^{\mathrm{op}},\ E\cap\Delta_{\C}
=\varnothing,\ 2\le\chi_c(\C,E)\le3\},
\]
where $E^{\mathrm{op}}=\{(y,x):(x,y)\in E\}$ and
$\Delta_{\C}=\{(x,x):x\in\C\}$. We also use the subspace
\[
 C_{\{2\}}=\{E\in C_{[2,3]}:\chi_c(\C,E)=2\}.
\]
These are spaces of compact edge relations, not codings of all Borel graphs.

\begin{prop}\label{prop:graph-coding}
The spaces $C_{\{2\}}$ and $C_{[2,3]}$ are Polish. Both
$\approx_c^i$ and $\cong_{\mathrm{top}}$ are analytic relations
on each of them.
\end{prop}
\begin{proof}
Symmetry is a closed condition on the compact hyperspace, and
avoiding the diagonal is open. Fix $k\in\{2,3\}$. A continuous map
$c:\C\to\{0,\ldots,k-1\}$ is a proper coloring of $E$ exactly when
$E\subseteq\{(x,y):c(x)\ne c(y)\}$, an open condition on $E$.
There are countably many such maps, since their fibers are clopen
and $\C$ has countably many clopen subsets. Existence of a continuous
proper $k$-coloring is therefore open. Edge sets in $\K(\C^2)$ are
nonempty, so a proper one-coloring is impossible. Taking $k=2$ and
$k=3$ shows that both $C_{\{2\}}$ and $C_{[2,3]}$ are open in the
closed subspace of symmetric compact sets. Both spaces are therefore
Polish.

In $\Cont(\C,\C)$, injectivity is $G_\delta$: for each $m\ge1$
require the images of all pairs at distance at least $1/m$ to be
separated by a positive distance. Compactness makes each such
condition open. The condition $(f\times f)(E)\subseteq E'$ is closed
in the product of this function space with the two hyperspaces.
Consequently $\preceq_c^i$ is analytic by projection; using two
witnesses gives analyticity of $\approx_c^i$.
The condition $(h\times h)(E)=E'$ with $h\in\Homeo(\C)$ is also
closed, proving analyticity of $\cong_{\mathrm{top}}$.
\end{proof}

For $U\in\Min(\C)$ define its undirected graph by
\[
E_U=\{(x,Ux):x\in\C\}\cup\{(Ux,x):x\in\C\}.
\]
It is compact, symmetric and irreflexive. The reduction below is the
one used by Lecomte in \cite[Theorem~13.1]{lecomte2023continuous};
we include the needed graph and dynamical arguments.

\begin{lem}\label{lem:minimal-graphs}
The map $U\mapsto E_U$ is continuous into $C_{[2,3]}$.
For $U,V\in\Min(\C)$, an injective continuous homomorphism from
$(\C,E_U)$ to $(\C,E_V)$ is precisely a flip conjugacy, with the same
witness. In particular it is a topological graph isomorphism.
\end{lem}
\begin{proof}
Uniform convergence of $U$ gives Hausdorff convergence of its graph.
Adjoining the transpose is also continuous. It remains
to verify the chromatic bound and the assertion about witnesses.

Choose a nonempty clopen set $A$ with $A\cap U(A)=\varnothing$;
this is possible because $U$ has no fixed point. By minimality, the
first-return time to $A$ is bounded and is a locally constant
function on $A$. Its values $h$ are at least two, and the sets
$U^i(A_h)$, $0\le i<h$, where $A_h$ is the return-time-$h$ subset,
form a finite clopen partition of $\C$. On a tower of even height
color the levels alternately $0,1$. On a tower of odd height color
the first $h-1$ levels alternately $0,1$ and its last level $2$.
Adjacent levels have different colors, every base has color $0$,
and every top has color $1$ or $2$. A top maps into a base, so this
is a continuous proper three-coloring of $E_U$. Nonemptiness of
$E_U$ gives $\chi_c(\C,E_U)\ge2$.

Let $f$ be an injective continuous graph homomorphism. Since $V$ has
no orbit of period one or two, for each $x$ there is a unique sign
$\epsilon(x)\in\{-1,1\}$ such that
$f(Ux)=V^{\epsilon(x)}f(x)$. The sets for the two signs are disjoint
closed sets covering $\C$, so the sign function is continuous.
If $\epsilon(Ux)=-\epsilon(x)$, then $f(U^2x)=f(x)$, contradicting
injectivity and absence of period-two points for $U$.
Thus $\epsilon$ is $U$-invariant; minimality makes it constant.
The image $f(\C)$ is then a nonempty compact $V$-invariant set and,
by minimality of $V$, equals $\C$. Hence $f$ is a homeomorphism and
is a conjugacy either to $V$ or to $V^{-1}$.
Conversely, either kind of conjugacy carries $E_U$ exactly onto $E_V$.
This proves the witness assertion, also recorded in
\cite[Lemma~7.11]{lecomte2023continuous}.
\end{proof}

\begin{lem}[Two-colorability of the reduction systems]
\label{lem:reduction-two-colorable}
Let $\Phi_1,\Phi_2$ be the maps constructed in the proof of
Theorem~\ref{thm:main-thm}. For every $T\in\Trees$ and $i\in\{1,2\}$,
the system $(\C,\Phi_i(T))$ has a factor onto
$(\Z/2\Z,+1)$. Consequently $E_{\Phi_i(T)}\in C_{\{2\}}$.
\end{lem}
\begin{proof}
Fix $T\in\Trees$ and put $R=\Psi(T)$. The construction in
Lemma~\ref{lem:word-construction} has $L_1=6$. Let
$\pi_{\W_1}:X_1(R)\to\cL$ be its canonical odometer factor from
Lemma~\ref{lem:canonical-odometer}. Its first coordinate takes values
in $\Z/6\Z$. Reduction modulo two therefore gives a factor map
\[
 b:X_1(R)\longrightarrow\Z/2\Z,\qquad
 b(z):=(\pi_{\W_1}(z))_1\pmod 2,
 \qquad b(\sigma z)=b(z)+1.
\]
Both $X(R)$ and $\Xother(R)$ project equivariantly and surjectively
onto $X_1(R)$: in both limits, the dynamics is the coordinatewise shift
and the bonding maps are surjective. Composing each projection with $b$ gives the
required factors for the two inverse limits. Transport them through
the conjugacies used to define $\Phi_1(T)$ and $\Phi_2(T)$.
For each $i$ this gives a continuous map
$c_i:\C\to\Z/2\Z$ satisfying
\[
 c_i(\Phi_i(T)x)=c_i(x)+1\qquad(x\in\C).
\]
Thus $c_i$ is a proper continuous two-coloring of $E_{\Phi_i(T)}$.
That edge set is nonempty, so its continuous chromatic number is
exactly two.
\end{proof}

\begin{proof}[Proof of Theorem~\ref{thm:graph-intro}]
By Theorem~\ref{thm:main-thm} and Lemmas~\ref{lem:minimal-graphs}
and~\ref{lem:reduction-two-colorable}, the map
\[
 \Trees\ni T\longmapsto
 \bigl(E_{\Phi_1(T)},E_{\Phi_2(T)}\bigr)\in C_{\{2\}}^2
\]
is continuous. The main reduction and the witness characterization
in Lemma~\ref{lem:minimal-graphs} give
\begin{align*}
 T\in\IllFounded
 &\quad\Longleftrightarrow\quad
 (\C,E_{\Phi_1(T)})\approx_c^i(\C,E_{\Phi_2(T)})\\
 &\quad\Longleftrightarrow\quad
 (\C,E_{\Phi_1(T)})\cong_{\mathrm{top}}(\C,E_{\Phi_2(T)}).
\end{align*}
Since $\IllFounded$ is complete analytic, both graph relations are
analytic-hard on $C_{\{2\}}$. Proposition~\ref{prop:graph-coding}
supplies their analytic upper bounds, proving completeness.
The inclusion $C_{\{2\}}\hookrightarrow C_{[2,3]}$ carries the same
reduction into the larger domain, where the analytic upper bounds
also hold. This proves the second assertion.
The equivalence between the two graph relations is used only on
the family of graphs of minimal homeomorphisms; it is not asserted
for arbitrary members of either parameter space.
\end{proof}

\appendix
\section{An alternative construction using reversal}
\label{sec:legacy-reversal}

We give an alternative construction of Cantor minimal systems $X_T$,
indexed by $T\in\Trees$. These systems are not identified with $X(T)$
or $\bar X(T)$ from the main construction.

In the main construction, the full tree group acts by aligned
automorphisms, and the bonding maps of the second inverse limit
determine the obstruction to conjugacy. The construction below uses a
specified parity homomorphism: old actions extend diagonally for even
elements and skew-diagonally for odd elements. Odd elements describe
conjugacies to the reversal, and even elements describe automorphisms.

The outcome is a system $X_T$ whose conjugacies onto $\revop X_T$
are parametrized by an integer shift and an odd coherent sequence of
the tree groups (Corollary~\ref{legacy:cor:big-group-odd}), and whose
automorphism group is $\Z\times\mathcal G(T)^{\mathrm{even}}$
(Proposition~\ref{legacy:even-automorphisms}). In particular,
$X_T\cong\revop X_T$ if and only if $T$ has an infinite branch, and
$T\mapsto X_T$ is continuous (Lemma~\ref{legacy:lem:continuity}), so
$T\mapsto(X_T,\revop X_T)$ is another continuous reduction of
ill-foundedness to conjugacy of Cantor minimal systems. Since the pair
$(X_T,\revop X_T)$ is always flip conjugate, this map cannot serve as
the reduction for flip conjugacy; that additional conclusion uses the
reversal-exclusion conditions of the main construction. We do not
repeat the completeness argument, which is that of
Theorem~\ref{thm:main-thm}. We record instead the description of the
automorphism group of $X_T$ and the alternative proof of
Theorem~\ref{thm:centralizer} that it affords.

\subsection{Tree conventions and tree groups}
\label{legacy:tree-strategy}
Recall that a tree in $\Trees$ is rooted, closed under prefixes, and of
unbounded depth; no restriction is placed on its first-level vertices. Enumerate its vertices as
$t_T(0),t_T(1),\ldots$ using the fixed enumeration from
Section~\ref{section:trees}, with $t_T(0)=\emptyword$.
For $i\ge1$, let $\tau(i)$ be the index of the parent of $t_T(i)$ and let
$\delta(i)=\depth{t_T(i)}$. Then $\tau(i)<i$.
Set $s(n)=\max_{1\le i\le n}\delta(i)$ and
$M(t)=\min\{i:\delta(i)=t\}$. These conventions agree with those used in the main
construction, and every initial finite list of vertices is locally constant
as a function of the tree.

For $x\in A^{\Z}$, recall that $\revop(x)_j=x_{-j}$ for $j\in\Z$.
For a subshift $X\subseteq A^{\Z}$, its \emph{reverse shift} is
$\revop(X)=\{\revop(x):x\in X\}$.

Reversal conjugates $(\revop(X),\sigma)$ to $(X,\sigma^{-1})$.

Following \cite{FRW}, set
$\mathcal G_t:=\bigoplus_{i:\delta(i)=t}\Z/2\Z$, with canonical generators
$\mathcal B_t:=\{\g_i:\delta(i)=t\}$ indexed by the vertices at depth $t$.
The homomorphism $\rho:\mathcal G_{t+1}\to\mathcal G_t$ sends
$\g_i$ to $\g_{\tau(i)}$; we specify its domain level when needed.
An element is \emph{odd} if it is a sum of an odd number of elements
of $\mathcal B_t$, and \emph{even} otherwise.
Thus $(\mathcal G_t,\rho)_{t\ge1}$ is the $T$-directed inverse system
of groups of Section~\ref{subsec:tree-groups}, with each vertex $t_T(i)$
replaced by its index $i$. Odd elements will parametrize the conjugacies
from the finite-alphabet components of $X_T$ onto their reversals, and even
elements will parametrize their automorphisms.

The word sets below are construction sequences in the sense of
Definition~\ref{defn:constrSeq}, and we use the notation $X_{\W}$ and
$\Seg(W^n)$, the canonical odometer of Lemma~\ref{lem:canonical-odometer},
alignment as in Definition~\ref{def:aligned-factor}, the block
extensions $f\blockexp{m}$ and $f\blockexp{*}$, and the factor-map and
tameness criteria of Lemmas~\ref{lem_constr_seq_isom}
and~\ref{lem_no_nontame_factors}, all from Section~\ref{sec:constr-seq}.

\subsection{The skew-diagonal word construction}
\label{legacy:subsection:mainconstr}
For a group $G$ equipped with a homomorphism
$\epsilon:G\to\mathbb Z/2\mathbb Z$, an element is \emph{even} if
$\epsilon(g)=0$ and \emph{odd} if $\epsilon(g)=1$. In the tree groups below,
$G$ is an $\mathbb F_2$-vector space with a specified basis and $\epsilon$ is
the sum of the basis coefficients. Thus parity is part of the data; it is not
defined by an arbitrary generating set.

Suppose $G$ acts on a set $Y$. For a positive integer $L$, define
\[
 J_g(i)=\begin{cases}i,&\epsilon(g)=0,\\L-1-i,&\epsilon(g)=1,\end{cases}
 \qquad
 (\widetilde g y)_i=g(y_{J_g(i)})\quad(0\le i<L).
\]
This is the \emph{skew-diagonal action} on $Y^L$. Indeed, the maps $J_g$
compose according to parity, and coordinate permutations commute with
coordinatewise applications of the action on $Y$; hence
$\widetilde{gh}=\widetilde g\,\widetilde h$.

For each $n$ and $1\le t\le s(n)$, let $G_t^n$ be the finite
$\mathbb F_2$-vector space with specified basis
\[
 \mathcal B_t^n=\{\g_i^n:1\le i\le n,\ \delta(i)=t\}.
\]
Identify old basis elements at successive stages. For $t<s(n)$, the map
$\rho:G_{t+1}^n\to G_t^n$ sends $\g_i^n$ to $\g_{\tau(i)}^n$.
It preserves parity. In formulas for actions we use multiplicative notation
for these groups.

\begin{lem}\label{legacy:wordgroup}
Let $T\in\Trees$. There exist a scale $(L_n)$ with $L_1=6$ and
$L_{n+1}/L_n\in3\N$, $L_{n+1}/L_n\ge6$, and nonempty finite sets
$W^n\subseteq(\{0,1\}^{s(n)})^{L_n}$ with the following properties.
\begin{itemize}
\item[(S1)] For every $t\ge1$, $\W_t=([W^n]_t)_{n\ge M(t)}$ is a
construction sequence.
\item[(S2)] The groups $G_t^n$ are the specified elementary abelian
$2$-groups above; every element has order dividing $2$.
\item[(S3)] The specified sets $\mathcal B_t^n$ are bases, and parity is their
coefficient-sum homomorphism.
\item[(S4)] $G_t^n$ acts freely on $[W^n]_t$.
\item[(S5)] For $t<s(n)$, $g\in G_{t+1}^n$ and $w\in[W^n]_{t+1}$,
$[gw]_t=\rho(g)[w]_t$.
\item[(S6)] At stage $n+1$, the restriction of the action to $G_t^n$ is its
skew-diagonal extension on the $n$-blocks.
\end{itemize}
The following additional conditions hold. Write $L=L_{n+1}/L_n$ and
$u\angl{i}=u_{[iL_n,(i+1)L_n)}$.
\begin{itemize}
\item[(D1$^{\mathrm{r}}$)] For $1\le t\le s(n)$, distinct words in $[W^n]_t$ have distinct
restrictions to $[L_n/3,2L_n/3)$.
\item[(D2$^{\mathrm{r}}$)] Let $1\le t\le t'\le s(n)$, $1\le c\le L-2$, and
$\psi_0:[W^n]_{t'}^2\to\revop[W^n]_t$. For all
$w,w'\in[W^{n+1}]_{t'}$ and $u\in\revop[W^{n+1}]_t$, there is
$0\le i<L$ such that
\[
 u\angl{i}\ne\psi_0(ww'\angl{i+c},ww'\angl{i+c+1}).
\]
\item[(D3$^{\mathrm{r}}$)] $L_{n+1}\ge6L_n$.
\item[(D4$^{\mathrm{r}}$)] Let $1\le t\le t'\le s(n)$ and
$f:[W^n]_{t'}\to\revop[W^n]_t$. If there exists
$w\in[W^{n+1}]_{t'}$ for which $f\blockexp{n+1}(w)$ belongs to
$\revop[W^{n+1}]_t$, then there is a single odd $g\in G_t^n$ such that
\[
 f(v)=\revop(g[v]_t)\qquad\text{for every }v\in[W^n]_{t'}.
\]
\end{itemize}
Moreover, the choices through stage $n$ can be made to depend only on
$t_T(1),\ldots,t_T(n)$.
\end{lem}

\begin{rem}\label{legacy:X_T}
Projecting the nested compact segmented spaces gives
$[X_{\W_{t+1}}]_t=X_{\W_t}$, so $(X_{\W_t},[\,\cdot\,]_t)_{t\ge1}$ is an
inverse system of subshifts. Each component is minimal and has the
odometer of unbounded scale as a factor, so it is infinite. By
Lemma~\ref{lem:thesis-inverse-minimal}, the inverse limit is a Cantor
minimal system; the coordinate projections identify it with the shift
space
\[
 X_T:=\{x\in\C^{\Z}:[x]_t\in X_{\W_t}\text{ for every }t\ge1\},
\]
which satisfies $[X_T]_t=X_{\W_t}$ for every $t\ge1$.
\end{rem}

\subsection{The probabilistic proof}
\label{legacy:section:lemma}

\begin{proof}[Proof of Lemma~\ref{legacy:wordgroup}]
We use the random variables and extension labels from the proof of
Lemma~\ref{lem:word-construction}, with skew-diagonal actions on block
positions. We give the changes, including the exceptional case in the
probability estimate. The word sets are constructed anew, not taken
from the main construction.

We also require every word in $[W^n]_{t-1}$ to have the same number
$F_{t-1}^n>1$ of extensions in $[W^n]_t$.
The set at level $0$ consists of the unique word over the one-symbol alphabet.
For the base, put
\[
 L_1=6,\qquad W^1=\{001011,001111\},
\]
and let the odd basis element of $G_1^1$ interchange the two words.
Their middle thirds are $10$ and $11$. In either word, $00$ occurs only at
the beginning, and both words end in $1$. An occurrence of either word in a
concatenation of two of them therefore begins at a block boundary. This
proves nonoverlap, freeness, and the initial extension count $F_0^1=2$.
Conditions involving two successive stages are imposed in the induction step.

Fix the construction through stage $n$. Write $S=s(n+1)$ and $d=\delta(n+1)$.
Let $A_t=[W^n]_t$ for $t\le s(n)$. If $S=s(n)+1$, use the auxiliary alphabet
$A_S$ obtained by adjoining to each $A_{S-1}$-word a constant final row,
with either value $0$ or $1$, and let $G_S^n$ be trivial. Every old word
then has two auxiliary extensions. This notation is used only in the current
step and does not change $W^n$. All auxiliary words have length $L_n$;
their nonoverlap follows by projection. Order each fiber of $A_t\to A_{t-1}$
and write $\operatorname{ext}_t(v,r)$ for its $r$th member,
$0\le r<F_{t-1}^n$.

Set $E_t=2$ for $t=d$ and $E_t=1$ otherwise. Labels of height $t$ are
\[
 \lambda=(e_1,g_1,\ldots,e_t,g_t),\qquad
 0\le e_j<E_j,\quad g_j\in G_j^n.
\]
For every label prefix $(\alpha,e_t)$ and every $0\le i<L$, choose
$X(\alpha,e_t;i)$ independently and uniformly in
$\{0,\ldots,F_{t-1}^n-1\}$. Here $L\ge6$ is a multiple of three to be
chosen later. Define labelled words recursively, starting with the unique
height-zero word, by
\begin{equation}\label{legacy:eqn:worddef}
 w(\alpha,e_t,g_t)=\widetilde{g_t}
 \bigl(\operatorname{ext}_t(w(\alpha)\angl{i},
             X(\alpha,e_t;i))\bigr)_{0\le i<L}.
\end{equation}
Let $B_t$ be the set of these words. Old group elements act on labels by
\begin{equation}\label{legacy:eqn:gw-multiplication}
 \widetilde h\,w(\alpha,e_t,g_t)=w(\alpha,e_t,hg_t).
\end{equation}
Since projection commutes with the old action and $\rho$ preserves parity,
\begin{equation}\label{legacy:eqn:w-cutoff}
 [w(\alpha,e_t,g_t)]_{t-1}=\widetilde{\rho(g_t)}\,w(\alpha).
\end{equation}
More generally, projection to height $r<t$ changes only the last group
entry of the truncated label:
\begin{equation}\label{legacy:eqn:w-cutoff-multiple}
 [w(e_1,g_1,\ldots,e_t,g_t)]_r
 =w(e_1,g_1,\ldots,e_r,H),\quad
 H=g_r\rho(g_{r+1})\cdots\rho^{t-r}(g_t).
\end{equation}
Thus $[B_t]_{t-1}=B_{t-1}$. We will choose the random data so that
\emph{distinct labels have distinct middle thirds}; call this condition (E2).
The preceding label identities do not depend on this condition.

For a fixed label at height $t$, its $L$ blocks are independent and uniformly
distributed on $A_t$. This follows inductively from the constant fiber sizes:
a uniform parent and an independent uniform fiber index give a uniform
extension, and the skew action merely permutes the block indices and applies
a fixed bijection to the blocks. Random data belonging to distinct final
block indices remain disjoint, even after these index permutations.

We use two probability estimates. All index expressions below are assumed
to take values in $\{0,\ldots,L-1\}$ on the finite index set $I$.
Let $a(i),b(i),c(i)$ be affine maps with slopes in $\{-1,1\}$.
For integers $1\le t\le t'$, a fixed label $w$ at height $t'$,
and a map $\psi:A_{t'}^2\to A_t$, the simultaneous equalities
\[
 \psi(w\angl{a(i)},w\angl{b(i)})=[w]_t\angl{c(i)}
 \qquad(i\in I)
\]
have probability at most $2^{-(|I|-2)/9}$ whenever $c$ is distinct from
both $a$ and $b$ as an affine map. Indeed, discard the at most two indices
at which an input index equals the target index. Among the remaining
index triples choose pairwise disjoint triples greedily; there are at least
$(|I|-2)/9$ of them. For each chosen triple the target is an independent
uniform $A_t$-word, so the probability of the equality, conditional on the
inputs, is $1/|A_t|\le1/2$. The chosen triples involve independent random
data.

For a separately labelled word $u$ at height $t\le t'$, consider instead
\begin{equation}\label{legacy:eqn:two-label-test}
 \psi(w\angl{a(i)},w\angl{b(i)})=u\angl{c(i)}\qquad(i\in I).
\end{equation}
Truncate the label of $w$ to height $t$ using
\eqref{legacy:eqn:w-cutoff-multiple}. If its prefix through $e_t$ differs
from that of $u$, condition on all random data except the final fiber
variables used by $u$. These variables do not occur in $w$ or in the parent
word used to construct $u$; after the fixed index permutation, they are
independent at distinct target indices. Each equality has conditional
probability at most $1/F_{t-1}^n\le1/2$. Averaging gives the bound
$2^{-|I|}$.
If the prefixes agree, there is a fixed $h\in G_t^n$ with
$u=\widetilde h[w]_t$. Applying $h^{-1}$ to the output of $\psi$
reduces \eqref{legacy:eqn:two-label-test} to the preceding one-label test,
with target index $J_h(c(i))$. Consequently the bound
$2^{-(|I|-2)/9}$ holds unless
\begin{equation}\label{legacy:eqn:exception}
 u=\widetilde h[w]_t
 \quad\text{and}\quad
 a=J_h\circ c\ \text{or}\ b=J_h\circ c
\end{equation}
as affine maps. This is the skew-diagonal exceptional case.
These estimates are unconditional: no construction requirement has
been imposed on the random table.

We now bound failures of those requirements. Each fixed ordered pair of
$A_t$-words appears in a fixed labelled word except on an event of probability
at most
\begin{equation}\label{legacy:eqn:s1proof}
 \prod_{i=0}^{\lfloor L/2\rfloor-1}
 \mathbb P\{w\angl{2i}w\angl{2i+1}\ne v_1v_2\}
 =\bigl(1-|A_t|^{-2}\bigr)^{\lfloor L/2\rfloor}.
\end{equation}
This gives double faithfulness. To check new nonoverlap, old nonoverlap
first forces any occurrence offset to be a multiple $cL_n$, $0<c<L$.
For $c\le L/2$, such an occurrence implies
$u\angl{i}=v\angl{i+c}$ for $0\le i<L-c$; for $c>L/2$, it implies
$u\angl{L-c+i}=v\angl{i}$ for $0\le i<c$.
In either case, \eqref{legacy:eqn:exception} is impossible: for even $h$
the two positive-slope index maps have different offsets, and for odd $h$
the transformed target has negative slope. The failure probability is at
most $2^{-(L/2-2)/9}$.

For (E2), compare two distinct labels on $L/3\le i<2L/3$. Distinct prefixes
are handled by the conditional estimate. For equal prefixes write
$u=\widetilde h w$. If $h$ is even and nonidentity, freeness of the old
action prevents equality of even a single aligned block. If $h$ is odd,
comparison with the index $L-1-i$ is bounded by the one-label estimate.
The case $h=e$ would mean that the labels coincide and is excluded.
Thus the failure probability is at most $2^{-(L/3-2)/9}$.
Condition (E2) implies both (D1$^{\mathrm{r}}$) and distinctness of the whole labelled words.

For (D4$^{\mathrm{r}}$), write the target word as $\revop u$ with $u\in B_t$.
The corresponding equations are
\[
 \revop\bigl(f(w\angl{i})\bigr)=u\angl{L-1-i}
 \qquad(0\le i<L).
\]
The exceptional case occurs only for odd $h$, since precisely then
$J_h(L-1-i)=i$. In that case the equations give
$f(w\angl{i})=\revop(h[w\angl{i}]_t)$.
Once double faithfulness holds, every $A_{t'}$-word occurs in $w$, so this
is the asserted identity for $f$ on its entire domain. Every other case has
probability at most $2^{-(L-2)/9}$.

For (D2$^{\mathrm{r}}$), write the target word as $\revop u$ and reverse each output block. If $c\le L/2$, restrict to
$0\le i\le L-c-2$. The input indices lie in the first labelled word and
are $i+c,i+c+1$; the target index is $L-1-i$. For even $h$ the transformed
target has negative slope, and for odd $h$ it is $i$, distinct from both
inputs. If $c>L/2$, put $i=L-c+r$, $0\le r<c$. The input indices in the
second labelled word are $r,r+1$; the target index is $c-1-r$.
For odd $h$ its transform is $L-c+r$, whose offset is at least $2$;
for even $h$ it still has negative slope. Thus the exceptional case is again
excluded. The failure probability in either case is at most
$2^{-(L/2-3)/9}$.

At the fixed stage, the label sets, old alphabets, and sets of maps
between them are finite and independent of $L$. Only the nonoverlap
offset and the displacement $c$ in (D2$^{\mathrm{r}}$) contribute factors bounded by $L$.
The union bound therefore gives constants $C_n<\infty$ and $c_n>0$,
independent of $L$, such that
\[
 \mathbb P\{\text{one of the preceding requirements fails}\}
 \le C_n(1+L)e^{-c_nL}.
\]
For (D4$^{\mathrm{r}}$), we exclude only the nonexceptional tests. Their complement,
together with double faithfulness, gives (D4$^{\mathrm{r}}$) by the preceding argument.
Choose a successful table with $L\in3\N$, $L\ge6$, and put
$L_{n+1}=LL_n$ and $W^{n+1}=B_S$.

It remains to add the new basis generator $q=\g_{n+1}^{n+1}$ at depth $d$.
At depth $1$, let it replace $e_1$ by $1-e_1$. At depth $d>1$, let
$p=\g_{\tau(n+1)}^n$ and set
\[
 q:w(\alpha,e_{d-1},g_{d-1},e_d,g_d)\longmapsto
 w(\alpha,e_{d-1},pg_{d-1},1-e_d,g_d).
\]
It is an involution and commutes with the old action, which changes only the
last group label. A nonidentity old element changes that last label, whereas
an element of the coset $qG_d^n$ changes $e_d$. By (E2), all these actions
are free, so the prescribed old basis together with $q$ remains independent.
The projection identity shows that $q$ descends to the skew-diagonal action
of its parent $p$. Higher old generators satisfy the same compatibility.
This proves (S2)--(S6).

Finally, a fixed word at height $t-1$ has exactly
\[
 F_{t-1}^{n+1}=E_t|G_t^n|
\]
extensions. For $t\ge2$, choose $e_t$ and $h'\in G_t^n$ freely;
\eqref{legacy:eqn:w-cutoff} determines the preceding group label uniquely.
Distinct choices give distinct words by (E2). At $t=1$, this is the direct
count of all height-one labels. At every old depth the old group is
nontrivial; at a newly introduced depth $E_t=2$. Hence the count is always
greater than one, and the inductive invariant is preserved.

Fix lexicographic orders on the finite alphabets and labels. Choose the
least admissible $L$ for which a successful table exists and then the first
successful table in that finite ordering. The preceding estimates guarantee
existence. This rule depends only on the previously constructed finite data
and the new vertex. Induction proves the asserted finite dependence.
\end{proof}

\subsection{Finite-level conjugacies and factor maps}
\label{legacy:finite-level-maps}
Fix a tree $T$ and write $X=X_T$ and $X_t=[X_T]_t=X_{\W_t}$.
Thus $X_t$ is a finite-alphabet factor of the reversal construction,
not the component $X_j(T)$ of the main construction.
\begin{defn}
	Let $X,Y$ be Cantor shift spaces and $t\ge1$. For conjugacies
$\phi_t\colon[X]_t\to[Y]_t$ and
$\phi_{t+1}\colon[X]_{t+1}\to[Y]_{t+1}$, we say that
$\phi_{t+1}$ is \emph{subordinate} to $\phi_t$ if
	$$
		[\phi_{t+1}(x)]_t = \phi_t([x]_t) \quad \textrm{for all } x \in [X]_{t+1}.
	$$
\end{defn}

Given a set of words $V$ we define $\revop:V\to\revop(V)$ 
as the function that reverses the order of a word.

We now describe all conjugacies from $X_t$ to $\revop(X_t)$.

\begin{defn} \label{legacy:defn:psi_maps}
    Let $g\in\mathcal G_t$ be odd, with expansion
$g=\g_{i_1}\dots\g_{i_k}$ in the generators $\mathcal B_t$.
Choose $n$ large enough that all $\g^n_{i_1},\dots,\g^n_{i_k}$ belong
to $G_t^n$, and put $g^n=\g^n_{i_1}\dots\g^n_{i_k}$.
    
    We define the map $\psi_{g} := (\revop \circ g^n) \blockexp{*}$. 
    \end{defn}
    \begin{prop}
        
    The map $\psi_g$ is independent of the sufficiently large stage $n$
and is a conjugacy from $X_t$ to $\revop(X_t)$. 
\end{prop}

\begin{proof}
    Fix such a stage $n$. In the notation of Lemma~\ref{lem_constr_seq_isom},
we show that $(\revop\circ g^n)\blockexp{n+1}=\revop\circ g^{n+1}$.
For $w\in[W^{n+1}]_t$, write $w=w_1\dots w_k$ with
$w_i\in[W^n]_t$. Then
    \begin{multline*}
        (\revop \circ g^n)\blockexp{n+1} (w) = \revop(g^n w_1) \dots \revop(g^n w_k) = \\
        \revop( g^n w_k \dots g^n w_1) = (\revop \circ g^{n+1}) (w).
    \end{multline*}
    Iteration gives $(\revop\circ g^n)\blockexp{m}=\revop\circ g^m$
for every $m\ge n$. Hence
$(\revop\circ g^n)\blockexp{*}=(\revop\circ g^m)\blockexp{*}$, proving
independence of $n$. 

    Each $g^m$ permutes $[W^m]_t$, and reversal maps $[W^m]_t$
bijectively onto $\revop([W^m]_t)$. Thus
$(\revop\circ g^n)\blockexp{m}=\revop\circ g^m$ is a bijection between
these sets for every $m\ge n$. Lemma~\ref{lem_constr_seq_isom}
therefore makes $\psi_g$ a conjugacy.
\end{proof}

\begin{lem} \label{legacy:lem:factor_psi_form}
    Let $t' \ge t \ge 1$. A function $\phi\colon X_{t'}\to\revop(X_t)$ is a factor map if and only if
    $$
        \phi(x) = \sigma^k \psi_g ( [x]_t ) 
    $$
    for some $g \in \mathcal{G}_t$ odd and $k \in \Z$.
\end{lem}

\begin{proof}
    Since $x \mapsto [x]_t$ is a factor map $X_{t'} \to X_t$ and $\psi_g$ is a conjugacy $X_t \to \revop(X_t)$,
    any map of this form is indeed a factor map.

    Conversely, suppose $\phi$ is a factor map. Both $X_{t'}$ and
$\revop(X_t)$ are generated by construction sequences with scale $(L_n)$. 
    Lemma \ref{lem_no_nontame_factors} applies, since (D1$^{\mathrm{r}}$), (D2$^{\mathrm{r}}$), (D3$^{\mathrm{r}}$) hold. So $\phi = \sigma^k \circ f^*$
    with $f^*$ aligned, for some $n\ge M(t')$, $k\in\Z$ and $f:[W^n]_{t'}\to\revop[W^n]_t$.
    Since $f^*$ is an aligned factor map, by Lemma \ref{lem_constr_seq_isom} we get that
    $f\blockexp{n+1}([W^{n+1}]_{t'})=\revop[W^{n+1}]_t$. 
    In particular, if we pick any word $w \in [W^{n+1}]_{t'}$, then
    $$
    	f(w \angl{0}) \dots f(w \angl{L_{n+1}/L_n-1}) \in \revop[W^{n+1}]_t,
    $$
    so by (D4$^{\mathrm{r}}$) we conclude that for some odd $g^n \in G^n_t$, the function $f$ is of the form
    $f(\cdot) = \revop(g^n([\cdot]_t))$. Let $g\in\mathcal G_t$ correspond to $g^n$ under the identification
of the specified bases. Then $g$ is odd and
$\psi_g=(\revop\circ g^n)\blockexp{*}$.
For every $x\in X_{t'}$, projecting its canonical
$[W^n]_{t'}$-segmentation gives a $[W^n]_t$-segmentation of $[x]_t$
at the same cut points. By Lemma~\ref{lem:unique-readability},
this is the canonical segmentation of $[x]_t$.
Applying the identity for $f$ on each block therefore gives
\[
  f\blockexp{*}(x)
  =(\revop\circ g^n)\blockexp{*}([x]_t)
  =\psi_g([x]_t).
\]
Since $\phi=\sigma^k\circ f\blockexp{*}$, we obtain
$\phi(x)=\sigma^k\psi_g([x]_t)$ for every $x\in X_{t'}$,
as required.
\end{proof}

Taking $t'=t$ shows that the conjugacies $X_t\to\revop(X_t)$ are
precisely the maps $\sigma^k\psi_g$, with $k\in\Z$ and
$g\in\mathcal G_t$ odd. We next characterize subordination.

\begin{lem} \label{legacy:lem:psi_subordinacy}
    Let $t\ge1$. For $i=0,1$, let $k_i\in\Z$, let
$g_i\in\mathcal G_{t+i}$ be odd, and set
$\phi_{t+i}=\sigma^{k_i}\psi_{g_i}$. Then
    $$
        [\phi_{t+1}(x)]_t = \phi_t([x]_t) \quad \textrm{for all } x \in X_{t+1}
    $$
    iff $k_0 = k_1$ and $g_0 = \rho(g_1)$.
\end{lem}

\begin{proof}
    Substituting the expressions for $\phi_t$ and $\phi_{t+1}$ gives
    \begin{equation} \label{legacy:eqn:lem_subordinate_isoms}
        \sigma^{k_1} [\psi_{g_1}(x)]_t = \sigma^{k_0}\psi_{g_0}([x]_t) \quad \textrm{for all } x \in X_{t+1}.
    \end{equation}
    The maps $\psi_{g_0}$ and $\psi_{g_1}$ are aligned. Comparing offsets
therefore gives $k_0=k_1$ whenever the equality holds. It remains to show that 
    \begin{equation} \label{legacy:eqn:lem_subordinate_isoms2}
        [\psi_{g_1}(x)]_t = \psi_{g_0}([x]_t) \quad \textrm{for all } x \in X_{t+1}
    \end{equation}
    iff $g_0 = \rho(g_1)$. 
    
    For large enough $n$, we can write $\psi_{g_0} = (\revop \circ g^n_0)\blockexp{*}$, 
    $\psi_{g_1} = (\revop \circ g^n_1)\blockexp{*}$, 
    where $g^n_0$ and $g^n_1$ are elements of $G^n_t, G^n_{t+1}$ as in Definition \ref{legacy:defn:psi_maps}.
    Equation~\eqref{legacy:eqn:lem_subordinate_isoms2} holds if and only if 
    \begin{equation*} 
        [g_1^n(w)]_t = g_0^n [w]_t \textrm{ for all } w \in [W^n]_{t+1}.
    \end{equation*}
    If $g_0 = \rho(g_1)$, then by (S5) this condition holds. 
    If $g_0 \neq \rho(g_1)$, then the freeness of the $G^n_t$ action guarantees that
    $g_0^n[w]_t \neq \rho(g_1^n)[w]_t = [g_1^n(w)]_t$.
\end{proof}

\subsection{Inverse-limit symmetries and the automorphism-group reduction}
\label{legacy:automorphisms}
All groups in this subsection refer to the family $X_T$ constructed here.
Set
\[
 \mathcal G(T)=\varprojlim(\mathcal G_t(T),\rho).
\]
Since $\rho$ preserves parity, all coordinates of a coherent sequence
have the same parity. Write $\mathcal G(T)^{\mathrm{even}}$ for the kernel of
this parity homomorphism. By Lemma~\ref{lem_trivial},
$\mathcal G(T)$ is nontrivial exactly when $T$ has an infinite branch.

For an odd coherent sequence $\mathbf g=(g_t)$, the maps $\psi_{g_t}$
are subordinate by Lemma~\ref{legacy:lem:psi_subordinacy} and define a
conjugacy $\Psi_{\mathbf g}:X_T\to\revop X_T$ by
Lemma~\ref{lem:small-inv-lim-isomorphism}(i).
\begin{cor}\label{legacy:cor:big-group-odd}
The conjugacies from $X_T$ to $\revop X_T$ are precisely
\[
 \{\sigma^k\Psi_{\mathbf g}:k\in\Z,\
                         \mathbf g\in\mathcal G(T)\text{ odd}\}.
\]
The parameters are unique:
$\sigma^k\Psi_{\mathbf g}=\sigma^l\Psi_{\mathbf h}$ implies
$k=l$ and $\mathbf g=\mathbf h$.
\end{cor}
\begin{proof}
Lemma~\ref{lem:small-inv-lim-isomorphism}(ii) and
Lemma~\ref{legacy:lem:factor_psi_form}
express each coordinate of any conjugacy in this form. Surjectivity of the
coordinate projections and Lemma~\ref{legacy:lem:psi_subordinacy} force a
common shift exponent and a coherent sequence. Conversely, subordinate
coordinate conjugacies induce a conjugacy of the limits by
Lemma~\ref{lem:small-inv-lim-isomorphism}(i).
At a fixed level, comparison of offsets determines the shift exponent.
Equality of the remaining aligned maps gives equality of their actions on
every sufficiently high construction word; freeness determines the group
element. Applying this at all levels proves uniqueness.
\end{proof}

For even $e\in\mathcal G_t$, define $\theta_e=(e^n)^*$ at a stage
$n\ge M(t)$ containing its finite support. This is independent of
$n$ by (S6), since the skew-diagonal extension of an even element is
diagonal. The block-factor criterion makes $\theta_e$ an aligned
automorphism of $X_t$.

\begin{prop}\label{legacy:even-automorphisms}
Every factor map $X_{t'}\to X_t$, $t'\ge t$, has the unique form
$\sigma^k\theta_e[\,\cdot\,]_t$, where $k\in\Z$ and
$e\in\mathcal G_t$ is even. Moreover, as abstract groups,
\[
 \Aut(X_T)\cong\Z\times\mathcal G(T)^{\mathrm{even}},
 \qquad
 \Aut'(X_T)\cong\mathcal G(T)^{\mathrm{even}}.
\]
\end{prop}
\begin{proof}
Choose an odd $a\in\mathcal G_t$; such an element exists because the tree
has a vertex at depth $t$. If $\phi:X_{t'}\to X_t$ is a factor map, apply
Lemma~\ref{legacy:lem:factor_psi_form} to $\psi_a\phi$. It yields
$\psi_a\phi=\sigma^k\psi_g[\,\cdot\,]_t$ with $g$ odd. On construction
blocks, the two word reversals cancel, giving
$\psi_a^{-1}\psi_g=\theta_{a^{-1}g}$; the group element $a^{-1}g$ is even.
This proves existence of the claimed form. Surjectivity of the projection,
comparison of offsets, and freeness prove uniqueness.

In particular, the coordinate inverse system is blended and all its
automorphisms descend. The induced map sends
$\sigma^k\theta_e$ to $\sigma^k\theta_{\rho(e)}$.
An automorphism of the inverse limit is coordinatewise by the blended
criterion (Lemma~\ref{lem:inv-lim-isomorphism}). Its coordinate expressions therefore have a common exponent
$k$ and coherent even elements $e_t$. Conversely, any such data give
compatible coordinate automorphisms and hence an automorphism of the limit.
Composition agrees with addition of the shift exponents and multiplication
of the coherent elements, proving the group statements.
\end{proof}

\begin{lem}\label{legacy:lem:continuity}
The map $\Trees\ni T\mapsto X_T\in\Min_p^\sigma(\C^{\Z})$ is continuous.
\end{lem}
\begin{proof}
Fix the first $t$ rows and a time window of length $q$. Choose $n$
with $s(n)\ge t$ and $L_n\ge q$. The first $n$ processed vertices
determine the construction through stage $n$. This finite list is
locally constant, so nearby trees have the same $[W^n]_t$. By
Lemma~\ref{lem:finite-language}, their projected systems have the
same legal words of length $q$. This gives equality of all
finite-cylinder incidence tests under consideration. Such tests form
the clopen hyperspace basis, proving continuity.
\end{proof}

\begin{proof}[Alternative proof of Theorem~\ref{thm:centralizer}]
If $T$ has exactly one branch, its inverse-limit group consists of the
identity and the branch sequence, which is odd, by
Remark~\ref{rem:unique-branch-group}. Thus its even subgroup is
trivial and every automorphism of $X_T$ is a shift power.
Equivalently, the conjugacy $\Psi_{\mathbf g}$ associated with the
unique odd coherent sequence satisfies
$\Psi_{\mathbf g}\phi=\sigma^k\Psi_{\mathbf g}$ for every automorphism
$\phi$, so $\phi=\sigma^k$.

If $T$ has at least two branches, their distinct odd sequences
$\mathbf g,\mathbf h$ give an automorphism
$\Psi_{\mathbf h}^{-1}\Psi_{\mathbf g}$ outside $\langle\sigma\rangle$.
Indeed, equality with $\sigma^k$ would contradict the uniqueness in
Corollary~\ref{legacy:cor:big-group-odd}.
Define
\[
 F(T)=\{(x_1+1,\ldots,x_k+1):(x_1,\ldots,x_k)\in T\}
       \cup\{1^k:k\ge0\}.
\]
The shifted copy uses first coordinates at least $2$, so it meets the
added branch only at the root.
Membership in $F(T)$ is determined by a fixed finite membership test in
$T$, hence $F$ is continuous. It has exactly one branch when $T$ is
well-founded and at least two when $T$ is ill-founded.
Thus $T\mapsto X_{F(T)}$ reduces $\IllFounded$ to the systems with an
automorphism outside the shift powers, and this map is continuous by
Lemma~\ref{legacy:lem:continuity}. Transfer to $\C$ by
Theorem~\ref{thm:subshifts-and-systems-are-bireducible}.
The analytic upper bound is the commuting-witness argument already given
in the first proof of Theorem~\ref{thm:centralizer}.
\end{proof}

\end{document}